\documentclass[leqno,french]{smfart}
\usepackage{hyperref}
\usepackage{amssymb}
\usepackage[square,numbers]{natbib}
\usepackage[all]{xy}
\usepackage{afterpage}

\renewcommand{\thesubsection}{\arabic{subsection}}
\renewcommand{\thesubsubsection}{\thesubsection.\arabic{subsubsection}}

\renewcommand{\theenumi}{\arabic{enumi}}
\renewcommand{\labelenumi}{(\theenumi)}
\renewcommand{\theenumii}{\roman{enumii}}
\renewcommand{\labelenumii}{\theenumii.}

\renewcommand{\theequation}{*}

\theoremstyle{plain}

\newtheorem{mainthm}{Th\'eor\`eme}

\newtheorem{mainsectionthm}{Th\'eor\`eme}
\newtheorem{mainsectionlemm}[mainsectionthm]{Lemme}
\numberwithin{mainsectionthm}{subsection}

\newtheorem*{mainsectioncor}{Corollaire}

\newtheorem{thm}[subsubsection]{Th\'eor\`eme}
\newtheorem{lemm}[subsubsection]{Lemme}
\newtheorem{prop}[subsubsection]{Proposition}
\newtheorem{obsv}[subsubsection]{Observation}

\DeclareMathOperator{\Mod}{\mathtt{Mod}}
\DeclareMathOperator{\Surj}{\mathit{Surj}}

\DeclareMathOperator{\dg}{\mathit{dg}}

\DeclareMathOperator{\sh}{\mathit{sh}}
\DeclareMathOperator{\sgn}{sgn}
\DeclareMathOperator{\Res}{\mathit{R}}
\DeclareMathOperator{\Free}{\mathit{F}}
\DeclareMathOperator{\Dec}{\mathit{Dec}}

\DeclareMathOperator{\X}{\mathcal{X}}
\DeclareMathOperator{\Z}{\mathcal{Z}}

\DeclareMathOperator{\Cat}{\mathit{Cat}}
\DeclareMathOperator{\Gr}{\mathit{Gr}}
\DeclareMathOperator{\Ob}{\mathrm{Ob}}

\DeclareMathOperator{\Tor}{\mathrm{Tor}}
\DeclareMathOperator{\Ext}{\mathrm{Ext}}

\DeclareMathOperator{\Mor}{\mathrm{Mor}}
\DeclareMathOperator{\Hom}{\mathrm{Hom}}
\DeclareMathOperator{\End}{\mathrm{End}}
\DeclareMathOperator{\In}{\mathit{In}}

\DeclareMathOperator{\id}{\mathrm{id}}

\DeclareMathOperator{\tr}{\mathit{tr}}

\DeclareMathOperator{\pt}{\mathit{b}}

\DeclareMathOperator*{\colim}{\mathrm{colim}}

\DeclareMathOperator{\NN}{\mathbb{N}}
\DeclareMathOperator{\ZZ}{\mathbb{Z}}
\DeclareMathOperator{\kk}{\Bbbk}

\DeclareMathOperator{\EOp}{\mathsf{E}}

\DeclareMathOperator{\dset}{\underline{\mathsf{d}}}
\DeclareMathOperator{\eset}{\underline{\mathsf{e}}}
\DeclareMathOperator{\fset}{\underline{\mathsf{f}}}
\DeclareMathOperator{\rset}{\underline{\mathsf{r}}}
\DeclareMathOperator{\sset}{\underline{\mathsf{s}}}
\DeclareMathOperator{\tset}{\underline{\mathsf{t}}}

\DeclareMathOperator{\aobj}{\underline{a}}
\DeclareMathOperator{\bobj}{\underline{b}}
\DeclareMathOperator{\xobj}{\underline{x}}
\DeclareMathOperator{\phiobj}{\underline{\phi}}
\DeclareMathOperator{\sigmaobj}{\underline{\sigma}}
\DeclareMathOperator{\tauobj}{\underline{\tau}}
\DeclareMathOperator{\thetaobj}{\underline{\theta}}

\title[La cat\'egorie des arbres \'elagu\'es de Batanin est de Koszul]{La cat\'egorie des arbres \'elagu\'es de Batanin\\est de Koszul}
\alttitle{Batanin's category of pruned trees is Koszul}

\author{Benoit Fresse}
\date{29 septembre 2009 (manuscrit r\'evis\'e le 7 f\'evrier 2010}
\address{UMR 8524 de l'Universit\'e Lille 1 - Sciences et Technologies - et du CNRS\\
Cit\'e Scientifique -- B\^atiment M2\\
F-59655 Villeneuve d'Ascq C\'edex (France)}
\email{Benoit.Fresse@math.univ-lille1.fr}
\urladdr{http://math.univ-lille1.fr/\~{ }fresse}
\subjclass{Primary: 57T30; Secondary: 05C05, 18G15, 18G55, 18D20, 55P48, 06A11}
\thanks{Recherche soutenue en partie par le contrat ANR-06-JCJC-0042 ``OBTH''}

\begin{document}

\begin{abstract}
La d\'efinition de la cat\'egorie des arbres \'elagu\'es,
dont les objets sont des arbres planaires \`a $n$-niveaux avec toutes les feuilles au niveau sup\'erieur,
a \'et\'e d\'egag\'ee dans les travaux de M. Batanin,
en partie pour comprendre la structure cellulaire de certaines $E_n$-op\'erades
en termes cat\'egoriques.
Le but de cet article est de montrer que la version enrichie en $\kk$-modules de la cat\'egorie des arbres \'elagu\'es est de Koszul.
Ce r\'esultat nous donne un mod\`ele diff\'erentiel gradu\'e minimal de cette cat\'egorie,
des petits complexes pour calculer des foncteurs $\Tor$ et $\Ext$ dans les cat\'egories de diagrammes qui lui sont associ\'es,
et permet d'\'etendre aux $E_n$-alg\`ebres un r\'esultat de M. Livernet et B. Richter sur l'interpr\'etation
des constructions bar it\'er\'ees en termes de foncteurs $\Tor$ cat\'egoriques.
\end{abstract}

\begin{altabstract}
The category of pruned trees has been defined by M. Batanin with the aim of understanding
the cell structure of certain $E_n$-operads in categorical terms.
The objects of this category are planar trees with $n$ levels
so that all leaves are at the top level of the tree.
The goal of this article is to prove that the category of pruned trees is Koszul.
This result gives us a minimal differential graded model of this category,
small complexes to computing Tor and Ext functors in associated categories of diagrams,
and allows us to extend to $E_n$-algebras a result of M. Livernet and B. Richter about the interpretation of iterated bar complexes
in terms of categorical Tor functors.
\end{altabstract}

\maketitle

\section*{Introduction}

La cat\'egorie ensembliste des arbres \'elagu\'es \`a $n$-niveaux, que l'on notera $\Omega_n^{epi}$,
est une sous-cat\'egorie de la cat\'egorie des arbres $\Omega_n$
qui est d\'efinie dans les travaux de M. Batanin sur les $\omega$-cat\'egories faibles~\cite{BataninGlobular}.
Les \'el\'ements de $\Omega_n$ repr\'esentent, dans l'id\'ee de~\cite{BataninGlobular},
des sch\'emas de compositions
dans les cat\'egories sup\'erieures (on renvoie \'egalement aux articles~\cite{BataninStreet,Joyal} pour des points de vue diff\'erents sur cette id\'ee).
La cat\'egorie $\Omega_n^{op}$, oppos\'ee \`a $\Omega_n$,
poss\`ede aussi une d\'efinition purement combinatoire,
en termes de produits en couronne it\'er\'es de la cat\'egorie simpliciale,
qui est utilis\'ee dans~\cite{BergerWreathCategories}
pour montrer qu'une localisation de la cat\'egorie des $\Omega_n^{op}$-espaces
d\'efinit un mod\`ele de la cat\'egorie des espaces de lacets $n$-it\'er\'es.

Les objets de $\Omega_n$ sont des arbres planaires \`a $n$-niveaux.
Les objets de $\Omega_n^{epi}$ sont les arbres de $\Omega_n$
dont toutes les feuilles sont situ\'ees au niveau sup\'erieur.
Les notions de code-barres~\cite{GetzlerJones} et d'ordre compl\'ementaire~\cite{KontsevichSoibelman}
qui apparaissent dans la d\'efinition
de certaines $E_n$-op\'erades
correspondent \`a diff\'erentes repr\'esentations des objets de $\Omega_n^{epi}$.
L'interpr\'etation de ces cat\'egories d'objets en termes de cat\'egories sup\'erieures
permet, selon M. Batanin~\cite{BataninEckmannHilton,BataninEnOperads},
d'obtenir une caract\'erisation intrins\`eque
du type d'homotopie des $E_n$-op\'erades,
du moins dans le cadre topologique.

Les objets de $\Omega_n^{epi}$ interviennent \'egalement dans le mod\`ele de Milgram des espaces de lacets it\'er\'es~\cite{BergerCell,Milgram},
dans les travaux de Fox et Neuwirth sur la pr\'esentation des groupes de tresses d'Artin~\cite{FoxNeuwirth},
et dans la d\'efinition des complexes bar it\'er\'es~\cite{FresseIteratedBar}.

\medskip
On consid\`ere dans cet article la version enrichie en $\kk$-modules de $\Omega_n^{epi}$,
pour un anneau de base fix\'e $\kk$.
Cette cat\'egorie elle-m\^eme peut se regarder comme une cat\'egorie enrichie en modules diff\'erentiels gradu\'es (une dg-cat\'egorie)
concentr\'ee en degr\'ee $0$.
Les constructions usuelles de l'alg\`ebre diff\'erentielle gradu\'ee ont une g\'en\'eralisation
naturelle dans le cadre des dg-cat\'egories.
On peut ainsi d\'efinir un analogue cat\'egorique des constructions bar et cobar classiques de l'alg\`ebre,
puis d\'efinir un analogue des complexes de Koszul de~\cite{Priddy} pour les cat\'egories munies d'une graduation en poids,
et \'etendre la notion d'alg\`ebre de Koszul, telle qu'elle est d\'efinie dans~\cite{Priddy},
aux cat\'egories.
On notera $\dg\Cat_{\Omega_n^{epi}}$ la cat\'egorie des dg-cat\'egories $\Theta$
qui ont $\Ob\Theta = \Ob\Omega_n^{epi}$ comme ensemble d'objets.

Le premier objectif de cet article est de montrer que la cat\'egorie $\Omega_n^{epi}$
est de Koszul~:
l'homologie de sa construction bar $B(\Omega_n^{epi})$
se r\'eduit \`a une cocat\'egorie enrichie en~$\kk$-modules gradu\'es $K(\Omega_n^{epi})$
dont les \'el\'ements $\alpha\in K(\Omega_n^{epi})(\tauobj,\sigmaobj)$
repr\'esentent des cycles de degr\'e maximal dans le complexe $B(\Omega_n^{epi})(\tauobj,\sigmaobj)$,
pour tout couple d'objets $(\tauobj,\sigmaobj)\in\Ob\Omega_n^{epi}\times\Ob\Omega_n^{epi}$.
Ce r\'esultat nous permettra d'obtenir~:
\begin{enumerate}
\item\label{Result:MinimalComplex}
un complexe naturel minimal $K(S,\Omega_n^{epi},T)$
pour calculer les foncteurs\break $\Tor^{\Omega_n^{epi}}_*(S,T)$
sur les cat\'egories de diagrammes associ\'ees \`a $\Omega_n^{epi}$,
ainsi qu'un complexe de cocha\^\i nes naturel minimal $C(S,\Omega_n^{epi},T)$
pour calculer les foncteurs $\Ext_{\Omega_n^{epi}}^*(S,T)$~;
\item\label{Result:MinimalModel}
un mod\`ele cofibrant minimal de $\Omega_n^{epi}$
dans $\dg\Cat_{\Omega_n^{epi}}$,
donn\'e par une construction cobar $B^c(K(\Omega_n^{epi}))$
sur la cocat\'egorie $K(\Omega_n^{epi})$.
\end{enumerate}
La construction bar de $\Omega_n^{epi}$
dans $\dg\Cat_{\Omega_n^{epi}}$
s'identifie en fait au complexe de cha\^\i nes du nerf de la version ensembliste de $\Omega_n^{epi}$.
Le r\'esultat obtenu dans cet article permet donc de d\'eterminer l'homologie du nerf de $\Omega_n^{epi}$.

On ne suit pas le plan habituel pour montrer qu'une alg\`ebre est de Koszul.
On d\'efinit d'abord une cocat\'egorie $K(\Omega_n^{epi})$ de fa\c con directe
que l'on forme avec les $\kk$-modules duaux (d\'ecal\'es en degr\'e)
des $\kk$-modules engendr\'es par les morphismes ensemblistes de $\Omega_n^{epi}$.
On d\'efinit aussi un complexe $K(S,\Omega_n^{epi},T)$
de fa\c con directe, \`a partir de cette cocat\'egorie $K(\Omega_n^{epi})$.
On d\'emontre, en \'etendant un argument de~\cite{LivernetRichter},
que ce complexe satisfait une propri\'et\'e d'acyclicit\'e,
ce qui entraine indirectement que la cat\'egorie $\Omega_n^{epi}$ est de Koszul,
et que la cocat\'egorie $K(\Omega_n^{epi})$, duale de $\Omega_n^{epi}$ dans les $\kk$-modules,
est aussi la cocat\'egorie duale de $\Omega_n^{epi}$ au sens de la dualit\'e de Koszul
des cat\'egories.
On en d\'eduit ensuite que les propri\'et\'es (\ref{Result:MinimalComplex}-\ref{Result:MinimalModel})
ci-dessus sont satisfaites pour cette cocat\'egorie $K(\Omega_n^{epi})$,
que nous avons d\'efini de fa\c con directe,
et le complexe $K(S,\Omega_n^{epi},T)$
qui lui est associ\'e.

\medskip
Cette article comprend une section pr\'eliminaire pour fixer le cadre de nos constructions,
deux parties principales 1-2, et une section de conclusion.

La section pr\'eliminaire servira principalement \`a pr\'eciser nos conventions de base
sur les modules diff\'erentiels gradu\'es.

La Partie 1 est consacr\'ee \`a la d\'emonstration de la propri\'et\'e de Koszul
et aux applications du point~(\ref{Result:MinimalComplex}).
La d\'efinition de la cat\'egorie des arbres \'elagu\'es $\Omega_n^{epi}$
est rappel\'ee au d\'ebut de cette partie.
On introduit ensuite la cocat\'egorie $K(\Omega_n^{epi})$,
puis les complexes \`a coefficients $K(S,\Omega_n^{epi},T)$,
avant de prouver les r\'esultats annonc\'es.

Pour cette partie du travail,
on n'aura besoin que des notions de base de l'alg\`ebre homologique
-- complexes, foncteurs Tor, suites spectrales --
dans le cadre additif des cat\'egories de $\Omega_n^{epi}$-diagrammes.
Cependant,
pour des raisons de coh\'erence avec la suite de l'article,
on utilisera la terminologie d'\'equivalence faible - issue du langage des cat\'egories mod\`eles -
pour d\'esigner tout morphisme d'objets diff\'erentiels gradu\'es
induisant un isomorphisme en homologie.
En outre,
on parlera de modules diff\'erentiels gradu\'es (dg-modules en abr\'eg\'e) pour d\'esigner les objets de notre cat\'egorie
de base plut\^ot que de complexes de cha\^\i nes.
En fait,
on r\'eservera la terminologie de complexe \`a certaines constructions sp\'ecifiques
sur les modules diff\'erentiels gradu\'es.
On renvoie le lecteur \`a la section pr\'eliminaire pour un expos\'e rapide de nos conventions.

La Partie 2 est consacr\'ee \`a la d\'efinition du mod\`ele cofibrant de $\Omega_n^{epi}$,
que l'on a annonc\'ee en~(\ref{Result:MinimalModel}),
et aux applications de ce mod\`ele pour la d\'efinition d'une bonne cat\'egorie de $\Omega_n^{epi}$-diagrammes
homotopiques.
On commencera cette partie par une section d'introduction exposant, dans le cadre conceptuel de l'alg\`ebre homotopique,
les applications des constructions de l'alg\`ebre diff\'erentielle gradu\'ee
aux dg-cat\'egories.

Le but initial de ce travail \'etait d'\'etendre aux $E_n$-alg\`ebres
un r\'esultat de M. Livernet et B. Richter sur l'interpr\'etation des constructions bar it\'er\'ees
en termes de foncteurs $\Tor$-cat\'egoriques~\cite{LivernetRichter}.
Ces applications seront abord\'ees dans la section de conclusion de cet article.

\paragraph*{Remerciements}
Je remercie Muriel Livernet pour une s\'erie de discussions qui sont \`a l'origine de ce travail.
Je la remercie aussi, ainsi que Bernhard Keller et \'Eric Hoffbeck, pour des remarques et questions
sur la version pr\'eliminaire du manuscrit
qui m'ont permis d'am\'eliorer la pr\'esentation de certaines constructions.

\section*{Cadre g\'en\'eral}\label{Background}

Pour commencer, on reprend la d\'efinition de la cat\'egorie des modules diff\'e\-ren\-tiels gradu\'es
qui fournira le cadre de nos constructions,
et on revoit rapidement la d\'efinition de sa structure de cat\'egorie mod\`ele.
On rappelle aussi des constructions fondamentales --
produit tensoriel des modules diff\'e\-ren\-tiels gradu\'es, hom-interne, suspension et torsion -
qui seront utilis\'ees tout au long de l'article.

\subsubsection{Cat\'egories de dg-modules}\label{Background:DGModules}
On travaille sur un anneau de base fix\'e $\kk$.
On utilisera la cat\'egorie des $\kk$-modules
et la cat\'egorie des modules diff\'erentiels gradu\'es sur $\kk$
comme cat\'egories de base.
La cat\'egorie des modules diff\'erentiels gradu\'es (on dira dg-modules pour abr\'eger)
sera not\'ee $\dg\Mod$.
On suppose par convention qu'un dg-module de $\dg\Mod$
est gradu\'e inf\'erieurement
et poss\`ede une diff\'erentielle interne $\delta: C\rightarrow C$
qui diminue le degr\'e de $1$.
Un $\kk$-module peut se voir comme un dg-module concentr\'e en degr\'e $0$.

On garde aussi la convention habituelle de nos articles qui est de supposer que les dg-modules de $\dg\Mod$
sont $\ZZ$-gradu\'es.
Le lecteur pourra faire un choix inverse et supposer que les dg-modules de notre cat\'egorie de base
sont concentr\'es en degr\'e $*\geq 0$,
mais certaines constructions (les foncteurs de dg-modules d'homomorphismes notamment)
produisent naturellement des modules $\ZZ$-gradu\'es.

\subsubsection{Structures tensorielles et homomorphismes de dg-modules}\label{Background:TensorStructures}
On munit la ca\-t\'e\-go\-rie des dg-modules
de son produit tensoriel usuel
\begin{equation*}
\otimes: \dg\Mod\times\dg\Mod\rightarrow\dg\Mod,
\end{equation*}
qui en fait une cat\'egorie mono\"\i dale sym\'etrique,
avec l'isomorphisme de sym\'etrie $\tau: C\otimes D\rightarrow D\otimes C$
d\'efini par la r\`egle des signes.
On utilisera la notation $\pm$
pour d\'esigner tout signe produit par une application
de cet isomorphisme de sym\'etrie.

La cat\'egorie $\dg\Mod$ poss\`ede
un hom-interne
\begin{equation*}
\Hom_{\dg\Mod}(-,-): \dg\Mod^{op}\times\dg\Mod\rightarrow\dg\Mod
\end{equation*}
et forme donc une cat\'egorie mono\"\i dale sym\'etrique ferm\'ee.

Le dg-module $\Hom_{\dg\Mod}(C,D)$
est engendr\'e en degr\'e $d$
par les morphismes de $\kk$-modules $f: C\rightarrow D$
qui augmentent le degr\'e de $d$.
La diff\'erentielle de $f: C\rightarrow D$ dans $\Hom_{\dg\Mod}(C,D)$
est donn\'ee par le commutateur gradu\'e de $f$
avec les diff\'erentielles internes~:
\begin{equation*}
\delta(f) = \delta f - \pm f \delta.
\end{equation*}
Le signe $\pm$ est d\'etermin\'e par la commutation de $f$, de degr\'e $d$,
avec le symbole de diff\'erentielle interne $\delta$, de degr\'e $-1$.
Donc, dans cette formule, on a $\pm = (-1)^{\deg f}$.
On utilise la terminologie d'homomorphisme de dg-modules
pour d\'esigner les \'el\'ements de ce dg-hom $\Hom_{\dg\Mod}(C,D)$
et les distinguer des morphismes actuels de la cat\'egorie des dg-modules.

Le dg-hom $\Hom_{\dg\Mod}(C,D)$ est naturellement $\ZZ$-gradu\'e, m\^eme lorsque les dg-modules sont concentr\'es en degr\'e $*\geq 0$.
On peut prendre une troncature en degr\'e $*\geq 0$
de ce dg-module
pour construire un hom interne dans la cat\'egorie des dg-modules $\NN$-gradu\'es,
cependant c'est toujours de la version $\ZZ$-gradu\'ee du dg-hom $\Hom_{\dg\Mod}(C,D)$
dont on aura besoin dans l'article.

\subsubsection{Suspension et d\'esuspension des dg-modules}\label{Background:Suspension}
Soit $\kk[d]$ le dg-module de rang $1$ concentr\'e en degr\'e $d$,
muni d'une diff\'erentielle \'evidemment triviale.
On associe \`a tout dg-module $C$ le dg-module d\'ecal\'e $C[d]$
tel que $C[d] = \kk[d]\otimes C$.
Pour $d=1$, on obtient ainsi l'op\'eration de suspension des dg-modules $\Sigma C = \kk[1]\otimes C$.
Pour $d=-1$, on obtient l'op\'eration de d\'esuspension $\Sigma^{-1} C = \kk[-1]\otimes C$.
Ces op\'erations seront utilis\'ees au~\S\ref{DGConstructions}.

\subsubsection{Homomorphismes de torsion de dg-modules}\label{Background:TwistedDGModules}
On utilise dans certaines constructions des dg-modules $C$ munis d'un homomorphisme $\partial: C\rightarrow C$,
de degr\'e $-1$,
qui additionn\'e \`a la diff\'erentielle interne de $C$
d\'efinit une nouvelle diff\'erentielle $\delta+\partial: C\rightarrow C$
sur le module gradu\'e sous-jacent \`a $C$.
On obtient ainsi un nouveau dg-module que l'on d\'esignera par la donn\'ee du couple $(C,\partial)$.
La relation $(\delta+\partial)^2 = 0$, n\'ecessaire pour que la diff\'erentielle de $(C,\partial)$ soit bien d\'efinie,
est \'equivalente \`a l'\'equation $\delta\partial + \partial\delta + \partial^2 = 0$
puisque la diff\'erentielle interne de $C$ v\'erifie d\'ej\`a $\delta^2 = 0$.
On dit alors que $\partial$ est un homomorphisme de torsion du dg-module $C$.
On note que la relation $\delta\partial + \partial\delta + \partial^2 = 0$
s'interpr\`ete comme une identit\'e
\begin{equation*}
\delta(\partial) + \partial^2 = 0
\end{equation*}
dans $\Hom_{\dg\Mod}(C,C)$.

\subsubsection{Complexes de dg-modules}\label{Background:DGModuleComplexes}
On utilisera la terminologie de complexe pour des dg-modules tordus particuliers $(C,\partial)$
naturellement d\'efinis par des suites
\begin{equation}
\cdots\rightarrow C_{d}\xrightarrow{\partial} C_{d-1}\rightarrow\cdots\xrightarrow{\partial} C_0
\end{equation}
comme dans l'alg\`ebre homologique classique.
Les exemples consid\'er\'es dans l'article comprennent~:
la construction bar $B(\Omega_n^{epi})$,
la version \`a coefficients de la construction bar $B(S,\Omega_n^{epi},T)$,
et les constructions de Koszul correspondantes.

Les composantes $C_d$ d'une telle suite~(*) sont en g\'en\'eral des dg-modules
qui poss\`edent eux-m\^emes une graduation
et une diff\'erentielle interne $\delta: C_d\rightarrow C_d$.
Le dg-module tordu associ\'e \`a la suite~(*)
est alors d\'efini par la somme de ces dg-modules $C = \bigoplus_{d=0}^{\infty} C_d$
avec un homomorphisme de torsion induit par composante par composante
par les homomorphismes $\partial: C_{d}\rightarrow C_{d-1}$
donn\'es dans la suite~(*).
On ne suppose pas n\'ecessairement que ces homomorphismes v\'erifient $\partial^2 = 0$.
On aura seulement besoin de la relation $\delta(\partial) + \partial^2 = 0$
du~\S\ref{Background:TwistedDGModules}
pour que le dg-module tordu $(C,\partial)$
soit bien d\'efini.

Si la suite~(*) est constitu\'ee de $\kk$-modules,
sans graduation ni diff\'erentielle interne,
alors on identifie tacitement chaque composante~$C_d$ \`a un dg-module concentr\'e en degr\'e~$d$
pour appliquer notre d\'efinition.
La composante de degr\'e~$*$ de $(C,\partial)$ s'identifie alors au $\kk$-module donn\'e $C_*$
et la diff\'erentielle de $(C,\partial)$
se r\'eduit \`a la somme des homorphismes $\partial: C_*\rightarrow C_{*-1}$
donn\'es dans la suite~(*).

La composante de degr\'e $*$ du dg-module tordu $(C,\partial)$
associ\'e \`a une suite~(*)
est, dans le cas g\'en\'eral, constitu\'ee de la somme des composantes de degr\'e~$*$
de chacun des dg-modules $C_d$, $d\in\NN$.
On parlera alors de la \emph{composante de degr\'e $d$ du complexe}~$(C,\partial)$
pour d\'esigner le dg-module $C_d$
et \'eviter toute \'equivoque.

\subsubsection{Le langage des cat\'egories mod\`eles}\label{Background:ModelStructure}
On munit la cat\'egorie des dg-modules de sa structure mod\`ele projective standard,
telle qu'elle est d\'efinie dans~\cite[\S 2]{Hovey},
de sorte qu'un morphisme de dg-modules $f: C\rightarrow D$
forme~:
\begin{itemize}
\item
une \'equivalence faible s'il induit un isomorphisme en homologie~;
\item
une fibration s'il est surjectif en tout degr\'e~;
\item
une cofibration s'il poss\`ede la propri\'et\'e de rel\`evement \`a droite par rapport
aux fibrations acycliques (les fibrations qui sont des \'equivalences faibles).
\end{itemize}

On renvoie le lecteur d\'ebutant \`a~\cite{Hovey}
pour un expos\'e complet sur les applications
des cat\'egories mod\`eles
aux dg-modules.
Rappelons simplement que les cofibrations de dg-modules se d\'efinissent de fa\c con effective
comme des r\'etracts d'inclusions $i: C\rightarrow(C\oplus E,\partial)$
o\`u $D = (C\oplus E,\partial)$
est un module tordu tel que $E = \bigoplus_{\alpha\in\Lambda} \kk e_\alpha$
est un $\kk$-module gradu\'e libre (sans diff\'erentielle interne)
muni d'une filtration
d\'efinie au niveau de sa base
\begin{equation*}
0 = \underbrace{\bigoplus_{\alpha\in\Lambda_0} \kk e_\alpha}_{E_0}\subset\cdots
\subset\underbrace{\bigoplus_{\alpha\in\Lambda_{\lambda}} \kk e_\alpha}_{E_\lambda}\subset\cdots
\subset\colim_{\lambda\in\NN}\bigl\{\bigoplus_{\alpha\in\Lambda_{\lambda}} \kk e_\alpha\bigr\} = E
\end{equation*}
de sorte que l'homomorphisme de torsion v\'erifie $\partial(C) = 0$
et $\partial(E_{\lambda})\subset C\oplus E_{\lambda-1}$,
pour tout $\lambda>0$.
Cette caract\'erisation nous permet de voir ais\'ement que le dg-module tordu $D = (C,\partial)$
form\'e \`a partir d'un complexe de $\kk$-modules projectifs,
ou plus g\'en\'eralement \`a partir d'un complexe de dg-modules cofibrants,
d\'efinit lui-m\^eme un objet cofibrant de la cat\'egorie des dg-modules (rappelons qu'un objet $D$ est cofibrant
lorsque le morphisme initial $0\rightarrow D$ est une cofibration).

On adoptera la terminologie de dg-cofibration pour d\'esigner tout morphisme, d\'efini \`a l'int\'erieur d'une cat\'egorie,
qui par oubli de structure d\'efinit une cofibration dans la cat\'egorie des dg-modules.
On utilisera des conventions similaires en parlant de dg-modules cofibrants
et d'objets dg-cofibrants.
On supposera g\'en\'eralement que les objets consid\'er\'es dans ce travail
sont dg-cofibrants.
Cette pr\'ecaution est superflue lorsque l'anneau de base est un corps
car tout les dg-modules sur $\kk$ sont alors cofibrants.

On utilisera le langage des cat\'egories mod\`eles
tout au long de l'article,
m\^eme si on n'aura besoin - dans la premi\`ere partie de l'article du moins -
que des id\'ees de l'alg\`ebre homologique classique.

\section*{Partie 1. Le complexe de Koszul des la cat\'egorie des arbres \'elagu\'es}

Cette partie est consacr\'ee \`a l'\'etude des complexes de Koszul $K(S,\Omega_n^{epi},T)$ et aux applications
de ces complexes pour le calcul des foncteurs Tor dans les cat\'egories
de $\Omega_n^{epi}$-diagrammes.
D'abord (au~\S\ref{PrunedTrees})
on rappelle la d\'efinition de la cat\'egorie des arbres \'elagu\'es~$\Omega_n^{epi}$.
Ensuite (au~\S\ref{KoszulConstruction})
on donne une d\'efinition ad-hoc de la construction de Koszul
associ\'ee \`a cette cat\'egorie et on montre que cette construction~$K(\Omega_n^{epi})$
se plonge dans le complexe bar usuel~$B(\Omega_n^{epi})$
dont on rappelle \'egalement la d\'efinition.
Puis (aux~\S\S\ref{CoefficientConstructions}-\ref{DGCoefficientConstructions})
on d\'efinit et \'etudie les versions \`a coefficients de la construction de Koszul de~$\Omega_n^{epi}$.
On explique que ces complexes de Koszul~$K(S,\Omega_n^{epi},T)$
sont faiblement \'equivalents aux constructions bar \`a coefficients~$B(S,\Omega_n^{epi},T)$
et per\-met\-tent de d\'eterminer des foncteurs~$\Tor^{\Omega_n^{epi}}_*(S,T)$
d\`es lors que certains de ces complexes sont acycliques,
ce que l'on d\'emontre dans la section finale de cette partie (au~\S\ref{KoszulAcyclicity}).

\subsection{La cat\'egorie des arbres \'elagu\'es}\label{PrunedTrees}
On reprend dans cette section la d\'efinition de la cat\'egorie des arbres \'elagu\'es~$\Omega_n^{epi}$
telle qu'elle appara\^\i t dans les travaux
de M. Batanin~\cite{BataninEckmannHilton,BataninEnOperads}.
On revoit aussi la d\'efinition d'une cat\'egorie comma,
form\'ee \`a partir de~$\Omega_n^{epi}$,
qui intervient dans notre travail sur les complexes bar it\'er\'es~\cite{FresseIteratedBar}.
On fera des rappels suppl\'ementaires sur ces d\'efinitions
en cours d'article.

On renvoie aux r\'ef\'erences cit\'ees dans l'introduction
et \`a la bibliographie de~\cite{BataninEckmannHilton,BataninEnOperads}
pour d'autres pr\'esentations de la cat\'egorie des arbres \'elagu\'es
et de ses variantes.

\subsubsection{La cat\'egorie des arbres \'elagu\'es~: d\'efinition formelle et repr\'esentation graphique}\label{PrunedTrees:Definition}
On note $\rset$ l'ensemble ordonn\'e $\rset = \{1<\dots<r\}$.

Les objets de $\Omega_n^{epi}$, d\'esign\'es par des lettres grecques soulign\'ees $\tauobj$,
sont les suites de surjections croissantes
\begin{equation*}
\tset_0\xrightarrow{\tau_1}\tset_1\xrightarrow{\tau_2}\cdots
\xrightarrow{\tau_n}\tset_n = *
\end{equation*}
avec l'ensemble final $* = \underline{\mathsf{1}}$
comme dernier ensemble but.
Une telle suite de surjections
d\'efinit un arbre planaire dont les sommets sont arrang\'es
sur des niveaux $0,\dots,n$~:
l'ensemble $\tset_i$
d\'efinit l'ensemble des sommets au niveau $i$
de l'arbre, ordonn\'es de la gauche vers la droite~;
la surjection $\tau_i$
d\'efinit l'ensemble des ar\^etes des sommets de niveau~$i-1$ vers les sommets de niveau~$i$.
Un exemple,
repr\'esentant une suite de surjections de la forme
\begin{equation*}
\underline{\mathsf{3}}\xrightarrow{\tau_1}
\underline{\mathsf{2}}\xrightarrow{\tau_2}
\underline{\mathsf{2}}\xrightarrow{\tau_3}
\underline{\mathsf{1}}\xrightarrow{\tau_4}\underline{\mathsf{1}},
\end{equation*}
est donn\'e Figure~\ref{Fig:LevelTree}.
\begin{figure}[h]
\begin{equation*}
\xymatrix@W=0mm@H=0mm@R=4.5mm@C=4.5mm@M=0mm{ \ar@{-}[d]\ar@{.}[r] & \ar@{-}[dr]\ar@{.}[rr] && \ar@{-}[dl]  \\
\ar@{-}[d]\ar@{.}[rr] && \ar@{-}[d]\ar@{.}[r] & \\
\ar@{-}[dr]\ar@{.}[rr] && \ar@{-}[dl]\ar@{.}[r] & \\
\ar@{.}[r] & \ar@{-}[d]\ar@{.}[rr] && \\
\ar@{.}[rrr] &&& }
\end{equation*}
\caption{}\label{Fig:LevelTree}\end{figure}

Un morphisme de $\Omega_n^{epi}$
\begin{equation*}
\{\tset_0\xrightarrow{\tau_1}\tset_1\xrightarrow{\tau_2}\cdots
\xrightarrow{\tau_n}\tset_n\}
\xrightarrow{u}\{\sset_0\xrightarrow{\sigma_1}\sset_1\xrightarrow{\sigma_2}\cdots
\xrightarrow{\sigma_n}\sset_n\}
\end{equation*}
est une suite d'applications surjectives $u: \tset_i\rightarrow\sset_i$
telles que le diagramme
\begin{equation*}
\xymatrix{ \tset_0\ar@{.>}[d]_{u}\ar[r]^{\tau_1} &
\tset_1\ar@{.>}[d]_{u}\ar[r]^{\tau_2} &
\cdots\ar[r]^{\tau_n} &
\tset_n\ar@{.>}[d]_{u} \\
\sset_0\ar[r]^{\sigma_1} &
\sset_1\ar[r]^{\sigma_2} &
\cdots\ar[r]^{\sigma_n} &
\sset_n }
\end{equation*}
commute
et $u: \tset_i\rightarrow\sset_i$
est croissante sur chaque sous-ensemble $\tau_{i+1}^{-1}(x)\subset\tset_i$,
pour~$x\in\tset_{i+1}$.

On utilise la notation $\In\tauobj$ pour d\'esigner l'ensemble source $\In\tauobj = \tset_0$
de la premi\`ere surjection d'un objet $\tauobj\in\Omega_n^{epi}$.
L'application $\In: \tauobj\mapsto\In\tauobj$
d\'efinit un foncteur de $\Omega_n^{epi}$
dans la cat\'egorie $\Surj$
constitu\'ee des ensembles finis avec les applications surjectives comme morphismes.

\subsubsection{Le poset des arbres \'elagu\'es avec une source fix\'ee}\label{PrunedTrees:CommaPoset}
On utilise dans~\cite{FresseIteratedBar} la cat\'egorie comma $\eset/\Omega_n^{epi}$
(not\'ee $\Pi^n(\eset)$ dans~\cite[\S A]{FresseIteratedBar}),
associ\'ee \`a chaque ensemble fini~$\eset$,
et dont les objets sont les suites
\begin{equation*}
\eset\xrightarrow{\tau_0}\tset_0\xrightarrow{\tau_1}\cdots\xrightarrow{\tau_n}\tset_n = *
\end{equation*}
telles que $\tauobj = \{\tset_0\xrightarrow{\tau_1}\cdots\xrightarrow{\tau_n}\tset_n\}$
d\'efinit un \'el\'ement de $\Omega_n^{epi}$
et $\{\eset\xrightarrow{\tau_0}\tset_0\}$
d\'efinit un morphisme $\tau_0: \eset\rightarrow\In\tauobj$
dans la cat\'egorie des surjections~$\Surj$.
L'application $\tau_0: \eset\rightarrow\tset_0$
se repr\'esente dans l'arbre associ\'e \`a $\tauobj$
par un \'etiquetage des noeuds de niveau $0$
par les sous-ensembles $\tau_0^{-1}(x)$, pour $x\in\tset_0$, comme dans l'exemple de la Figure~\ref{Fig:LabelLevelTree}.
\begin{figure}[h]
\begin{equation*}
\xymatrix@W=0mm@H=0mm@R=4.5mm@C=4.5mm@M=0mm{ \ar@{-}[d]\ar@{.}[r]^<{\displaystyle e} &
\ar@{-}[dr]\ar@{.}[rr]^<{\displaystyle b\,c}^>{\displaystyle a\,d} && \ar@{-}[dl]  \\
\ar@{-}[d]\ar@{.}[rr] && \ar@{-}[d]\ar@{.}[r] & \\
\ar@{-}[dr]\ar@{.}[rr] && \ar@{-}[dl]\ar@{.}[r] & \\
\ar@{.}[r] & \ar@{-}[d]\ar@{.}[rr] && \\
\ar@{.}[rrr] &&& }
\end{equation*}
\caption{Un arbre \'elagu\'e \'etiquet\'e.
Dans cet exemple, on a $\eset = \{a,b,c,d,e\}$, et la surjection $\tau_0: \eset\rightarrow\tset_0$
est d\'efinie par $\tau_0(e) = 1$, $\tau_0(b) = \tau_0(c) = 2$ et $\tau_0(a) = \tau_0(d) = 3$.}\label{Fig:LabelLevelTree}\end{figure}
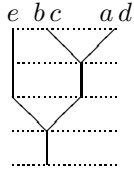

Une suite d'applications surjectives $u: \tset_i\rightarrow\sset_i$
d\'efinit un morphisme dans la cat\'egorie comma
\begin{equation*}
\{\eset\xrightarrow{\tau_0}\tset_0\xrightarrow{\tau_1}\tset_1\xrightarrow{\tau_2}\cdots
\xrightarrow{\tau_n}\tset_n\}
\xrightarrow{u}\{\eset\xrightarrow{\sigma_0}\sset_0\xrightarrow{\sigma_1}\sset_1\xrightarrow{\sigma_2}\cdots
\xrightarrow{\sigma_n}\sset_n\}
\end{equation*}
si chaque application $u: \tset_i\rightarrow\sset_i$
est croissante sur les sous-ensembles $\tau_{i+1}^{-1}(x)\subset\tset_i$, $x\in\tset_{i+1}$ (condition pour avoir un morphisme dans $\Omega_n^{epi}$),
et le diagramme
\begin{equation*}
\xymatrix{ \eset\ar[d]_{=}\ar[r]^{\tau_0} &
\tset_0\ar@{.>}[d]_{u}\ar[r]^{\tau_1} &
\tset_1\ar@{.>}[d]_{u}\ar[r]^{\tau_2} &
\cdots\ar[r]^{\tau_n} &
\tset_n\ar@{.>}[d]_{u} \\
\eset\ar[r]^{\sigma_0} &
\sset_0\ar[r]^{\sigma_1} &
\sset_1\ar[r]^{\sigma_2} &
\cdots\ar[r]^{\sigma_n} &
\sset_n }
\end{equation*}
commute dans son ensemble.

On montre dans~\cite[\S A.7]{FresseIteratedBar} que $\eset/\Omega_n^{epi}$ forme un poset~:
l'ensemble des morphismes dans~$\eset/\Omega_n^{epi}$
d'un objet source $\{\eset\xrightarrow{\tau_0}\tset_0\xrightarrow{\tau_1}\cdots\xrightarrow{\tau_n}\tset_n\}$
vers un objet but donn\'e $\{\eset\xrightarrow{\sigma_0}\sset_0\xrightarrow{\sigma_1}\cdots\xrightarrow{\sigma_n}\sset_n\}$
est soit vide,
soit r\'eduit \`a un unique \'el\'ement,
et dans ce dernier cas on \'ecrit
\begin{equation*}
\{\eset\xrightarrow{\tau_0}\tset_0\xrightarrow{\tau_1}\cdots\xrightarrow{\tau_n}\tset_n\}
\geq\{\eset\xrightarrow{\sigma_0}\sset_0\xrightarrow{\sigma_1}\cdots\xrightarrow{\sigma_n}\sset_n\}.
\end{equation*}
Ce r\'esultat entraine qu'un morphisme $u: \tauobj\rightarrow\sigmaobj$
de la cat\'egorie $\Omega_n^{epi}$
est enti\`erement determin\'e par la donn\'ee~:
\begin{itemize}
\item
d'un objet source $\tauobj = \{\tset_0\xrightarrow{\tau_1}\cdots\xrightarrow{\tau_n}\tset_n\}$,
\item
d'un objet but $\sigmaobj = \{\sset_0\xrightarrow{\sigma_1}\cdots\xrightarrow{\sigma_n}\sset_n\}$,
\item
et d'une application $u: \In\tauobj\rightarrow\In\sigmaobj$,
\end{itemize}
de sorte qu'on a la relation
\begin{equation*}
\{\In\tauobj\xrightarrow{=}\tset_0\xrightarrow{\tau_1}\cdots\xrightarrow{\tau_n}\tset_n\}
\geq\{\In\tauobj\xrightarrow{u}\sset_0\xrightarrow{\sigma_1}\cdots\xrightarrow{\sigma_n}\sset_n\}
\end{equation*}
dans le poset $\In\tauobj/\Omega_n^{epi}$.
On utilisera cette observation pour repr\'esenter graphiquement les morphismes de~$\Omega_n^{epi}$
par la donn\'ee de l'arbre source $\tauobj$
et la donn\'ee de l'arbre but $\sigmaobj$ \'etiquet\'e par les \'el\'ements de~$\In\tauobj$,
ou de tout ensemble de variables (muettes) $\eset$
mis en bijection avec~$\In\tauobj$
(voir par exemple~\S\ref{KoszulAcyclicity}).

\subsubsection*{Remarque}
Les r\'esultats que l'on obtiendra aux~\S\S\ref{KoszulConstruction}-\ref{KoszulAcyclicity}
montrent que le poset $\eset/\Omega_n^{epi}$
est de Cohen-Macaulay.
On peut cependant observer qu'il n'admet pas de CL-shelling au sens de~\cite{BjornerWachs}
sauf lorsque $n$
ou le cardinal de $\eset$
sont petits.

\subsubsection{La graduation de la cat\'egorie des arbres \'elagu\'es}\label{PrunedTrees:Grading}
Dans nos constructions,
on utilisera de fa\c con essentielle que la cat\'egorie des arbres \'elagu\'es $\Omega_n^{epi}$
poss\`ede une graduation naturelle.
On d\'efinit d'abord le degr\'e d'un objet $\tauobj\in\Omega_n^{epi}$,
d\'efini par une suite de surjections $\tset_0\xrightarrow{\tau_1}\cdots\xrightarrow{\tau_n}\tset_n$,
en posant~:
\begin{equation*}
\deg(\tauobj) = t_0 + \dots + t_{n-1}.
\end{equation*}
Graphiquement, cette expression repr\'esente le nombre de sommets de niveau $<n$
de l'arbre repr\'esentant $\tauobj$,
ou, de fa\c con \'equivalente,
le nombre d'ar\^etes de $\tauobj$.
Le degr\'e d'un morphisme $u: \tauobj\rightarrow\sigmaobj$
est ensuite d\'efini par la diff\'erence~:
\begin{equation*}
\deg(u) = \deg(\tauobj)-\deg(\sigmaobj),
\end{equation*}
repr\'esentant le nombre de sommets de $\tauobj$
identifi\'es par~$u$.

\subsection{La construction de Koszul de la dg-cat\'egorie des arbres \'elagu\'es}\label{KoszulConstruction}
On a rappel\'e la d\'efinition de la cat\'egorie des arbres \'elagu\'es $\Omega_n^{epi}$
dans un cadre ensembliste
au~\S\ref{PrunedTrees}.
On utilise maintenant une version enrichie en $\kk$-modules
de cette cat\'egorie,
avec le m\^eme ensemble d'objets,
mais dont les hom-objets sont les $\kk$-modules librement engendr\'es par les morphismes ensemblistes de $\Omega_n^{epi}$.
On gardera la m\^eme notation $\Omega_n^{epi}$
pour d\'esigner cette cat\'egorie enrichie en $\kk$-modules.
On r\'eserve simplement la notation $\Mor_{\Omega_n^{epi}}(\tauobj,\sigmaobj)$
pour d\'esigner l'ensemble des morphismes ensemblistes de $\Omega_n^{epi}$
et on prendra la notation $\Omega_n^{epi}(\tauobj,\sigmaobj)$
pour d\'esigner le $\kk$-module des homomorphismes de la cat\'egorie $\Omega_n^{epi}$
dans sa version enrichie.
L'expression $u: \tauobj\rightarrow\sigmaobj$
ne sera utilis\'ee que pour d\'esigner des morphismes ensemblistes $u\in\Mor_{\Omega_n^{epi}}(\tauobj,\sigmaobj)$.
On a donc l'identit\'e~:
\begin{equation*}
\Omega_n^{epi}(\tauobj,\sigmaobj)
= \bigoplus_{u: \tauobj\rightarrow\sigmaobj} \kk\{u\},
\end{equation*}
en notant $\{u\}$ les \'el\'ements g\'en\'erateurs de $\Omega_n^{epi}(\tauobj,\sigmaobj)$
associ\'es aux morphismes ensemblistes $u: \tauobj\rightarrow\sigmaobj$.

On montrera plus loin comment les constructions de l'alg\`ebre diff\'erentielle gradu\'ee
s'appliquent \`a cette cat\'egorie enrichie en $\kk$-modules $\Omega_n^{epi}$,
que l'on voit elle-m\^eme
comme une dg-cat\'egorie concentr\'ee en degr\'e $0$.
On utilisera en particulier une construction bar naturellement associ\'ee
\`a toute dg-cat\'egorie $\Theta$ qui a $\Ob\Theta = \Ob\Omega_n^{epi}$
comme ensemble d'objet.
On expliquera que cette construction bar forme une dg-cocat\'egorie.

On se contente pour le moment de consid\'erer des structures,
appel\'ees dg-graphes,
d\'efinies par des collections de dg-modules~$\Gamma(\tauobj,\sigmaobj)$
index\'ees par les couples $(\tauobj,\sigmaobj)\in\Ob\Omega_n^{epi}\times\Ob\Omega_n^{epi}$.

La construction bar (normalis\'ee r\'eduite) de~$\Omega_n^{epi}$
est le dg-graphe $B(\Omega_n^{epi})(\tauobj,\sigmaobj)$
engendr\'e en degr\'e~$d$
par les $d$-uplets de morphismes composables
\begin{equation*}
\{\tauobj_0\xleftarrow{u_1}\cdots\xleftarrow{u_d}\tauobj_d\}
\in\Mor_{\Omega_n^{epi}}(\tauobj_1,\tauobj_0)\times\dots\times\Mor_{\Omega_n^{epi}}(\tauobj_d,\tauobj_{d-1})
\end{equation*}
satisfaisant la condition de non-d\'eg\'en\'erescence~$u_i\not=\id$, $\forall i$,
et tels que~$(\tauobj_d,\tauobj_0) = (\tauobj,\sigmaobj)$,
avec la diff\'erentielle
d\'efinie par la formule usuelle
\begin{equation*}
\partial\{\tauobj_0\xleftarrow{u_1}\cdots\xleftarrow{u_d}\tauobj_d\}
= \sum_{i=1}^{d-1} (-1)^i\{\tauobj_0\xleftarrow{u_1}\cdots
\xleftarrow{u_i u_{i+1}}\cdots
\xleftarrow{u_d}\tauobj_d\},
\end{equation*}
pour tout \'el\'ement $\{\tauobj_0\xleftarrow{u_1}\cdots\xleftarrow{u_d}\tauobj_d\}\in B(\Omega_n^{epi})$.
La construction bar d\'efinit un complexe
au sens du~\S\ref{Background:DGModuleComplexes}.

Le but principal de cette section est de construire un sous-objet de $B(\Omega_n^{epi})$,
la construction de Koszul $K(\Omega_n^{epi})$,
dont les \'el\'ements repr\'esenteront des cycles de degr\'e maximal
dans les diff\'erentes composantes de la construction bar~$B(\Omega_n^{epi})(\tauobj,\sigmaobj)$.

On verra plus loin que $K(\Omega_n^{epi})$ forme un sous-objet de $B(\Omega_n^{epi})$
dans la cat\'egorie des dg-cocat\'egories.
Cette cocat\'egorie $K(\Omega_n^{epi})$ est un analogue, dans le cadre des cat\'egories,
du dual de Koszul des alg\`ebres associatives
d\'efini dans~\cite{Priddy}.
La m\'ethode habituelle pour construire le dual de Koszul d'une alg\`ebre $A$
consiste \`a d\'eterminer une pr\'esentation de $A$
par g\'en\'erateurs et relations,
puis \`a prendre des relations orthogonales pour d\'efinir la cog\`ebre $K(A)$.
Dans notre cas, il sera plus simple de construire la cocat\'egorie $K(\Omega_n^{epi})$
de fa\c con directe.
L'identification de $K(\Omega_n^{epi})$
avec un dual de Koszul r\'esultera alors de l'identit\'e, \'etablie au Th\'eor\`eme~\ref{PrunedTreeCobar:KoszulModel},
d'une construction cobar $B^c(K(\Omega_n^{epi}))$ avec un mod\`ele minimal de $\Omega_n^{epi}$.
On appliquera en fait la dualit\'e
de Koszul en sens inverse
pour obtenir une pr\'esentation par g\'en\'erateurs et relations de $\Omega_n^{epi}$
\`a partir de ce r\'esultat.

On s'appuie n\'eanmoins sur la graduation naturelle de $\Omega_n^{epi}$
pour d\'efinir~$K(\Omega_n^{epi})$,
comme dans la th\'eorie classique de la dualit\'e de Koszul des alg\`ebres~\cite{Priddy}.
On commence par r\'eviser la structure des morphismes de degr\'e $1$
de la cat\'egorie des arbres \`a niveaux.

\subsubsection{Les morphismes d'arbres \'elagu\'es de degr\'e $1$}\label{KoszulConstruction:IndecomposableMorphisms}
Les observations de~\cite[Lemma 3.4]{LivernetRichter} (voir \'egalement~\cite[\S A.8]{FresseIteratedBar})
impliquent que les morphismes de degr\'e $1$
de $\Omega_n^{epi}$
sont donn\'es par des diagrammes de la forme
\begin{equation*}
\xymatrix{ \tset_0\ar[r]^{\tau_1}\ar[d]_{\sh_0} &
\cdots\ar[r]^{\tau_{k-1}} &
\tset_{k-1}\ar[r]^{\tau_k}\ar[d]_{\sh_{k-1}} &
\tset_{k}\ar[r]^{\tau_{k+1}}\ar[d]_{d_a} &
\tset_{k+1}\ar[r]^{\tau_{k+2}}\ar[d]_{\id} &
\cdots\ar[r]^{\tau_n} &
\tset_n\ar[d]_{\id} \\
\tset_0\ar[r]^{\sigma_1} &
\cdots\ar[r]^{\sigma_{k-1}} &
\tset_{k-1}\ar[r]^{\sigma_k} &
\underline{\mathsf{t_{\mathit{k}}-1}}\ar[r]^{\sigma_{k+1}} &
\tset_{k+1}\ar[r]^{\sigma_{k+2}} &
\cdots\ar[r]^{\sigma_n} &
\tset_n }
\end{equation*}
et que l'on construit de la fa\c con suivante~:
\begin{enumerate}\renewcommand{\theenumi}{\arabic{enumi}}
\item\label{IndecomposableMorphism:Bottom}
l'application $u_i: \tset_i\rightarrow\sset_i$ est l'identit\'e pour $i>k$,
ce qui implique la relation $\sigma_i = \tau_i$ pour $i>k+1$~;
\item\label{IndecomposableMorphism:VertexMerging}
l'application $d_a$ identifie deux sommets cons\'ecutifs $(a,a+1)$ d'une fibre $\tau_{k+1}^{-1}(b)$ au niveau $k$,
de sorte que l'on a
\begin{align*}
d_a(x) & = \begin{cases} x, & \text{pour $x = 1,\dots,a$}, \\ x-1, & \text{pour $x = a+1,\dots,s_k$}, \end{cases} \\
\text{et n\'ecessairement}
\quad\sigma_{k+1}(x) & = \begin{cases} \tau_{k+1}(x), & \text{pour $x = 1,\dots,a$}, \\ \tau_{k+1}(x+1), & \text{pour $x = a+1,\dots,t_k-1$}\ ; \end{cases}
\end{align*}
\item\label{IndecomposableMorphism:LegShuffle}
puis, lorsque $k>0$, on a n\'ecessairement
\begin{equation*}
\sigma_k^{-1}(x) = \begin{cases} \tau_k^{-1}(x), & \text{pour $x = 1,\dots,a-1$}, \\
\tau_k^{-1}(a)\cup\tau_k^{-1}(a+1), & \text{pour $x = a$}, \\
\tau_k^{-1}(x), & \text{pour $x = a+2,\dots,s_{k+1}$},
\end{cases}
\end{equation*}
ce qui d\'etermine $\sigma_k$~;
l'application $\sh_{k-1}$ au niveau $k-1$
est la juxtaposition des applications identiques sur les fibres $\tau_k^{-1}(x)$, $x\not=a,a+1$,
avec une bijection
\begin{equation*}
\tau_k^{-1}(a)\amalg\tau_k^{-1}(a+1)\xrightarrow{\simeq}\sigma_k^{-1}(a)
\end{equation*}
pr\'eservant l'ordre entre les \'el\'ements de $\tau_k^{-1}(a)$ et $\tau_k^{-1}(a+1)$~
(en d'autres termes, cette application est une permutation de battage)~;
\item\label{IndecomposableMorphism:FiberShuffle}
les applications $\sigma_i: \tset_{i-1}\rightarrow\tset_i$ et les bijections $\sh_{i-1}: \tset_{i-1}\rightarrow\tset_{i-1}$
sont ensuite d\'etermin\'ees inductivement~;
d'abord $\sigma_i$,
par les relations $\sigma_i^{-1}(x) = \tau_i^{-1}(\sh_i^{-1}(x))$, $x\in\tset_{i}$~;
puis $\sh_{i-1}$,
par la relation $\sigma_i\sh_{i-1} = \sh_i\tau_i$
et la propri\'et\'e de croissance sur les fibres $\tau_i^{-1}(x)$, $x\in\tset_{i}$.
\end{enumerate}
Un morphisme de degr\'e $1$
est donc enti\`erement d\'etermin\'e, \`a partir de son domaine $\tauobj$,
par la fusion de deux sommets cons\'ecutifs $(a,a+1)$
sur une fibre $\tau_{k+1}^{-1}(b)$ de l'arbre $\tauobj$
et, lorsque $k>0$, par une bijection
\begin{equation*}
\sh_{k-1}: \tau_k^{-1}(a)\amalg\tau_k^{-1}(a+1)\xrightarrow{\simeq}\sigma_k^{-1}(a)
\end{equation*}
m\'elangeant les fibres de ces sommets.

La Figure~\ref{Fig:LevelTreeMorphism}
donne la repr\'esentation graphique sch\'ematique
d'un tel morphisme de degr\'e $1$.
La notation $\alpha$ (respectivement $\beta$)
repr\'esente l'ensemble du sous-arbre au dessus du sommet fusionn\'e $a$ (respectivement, $b = a+1$).
L'observation~(\ref{IndecomposableMorphism:FiberShuffle})
de la construction formelle donn\'ee dans ce paragraphe
signifie graphiquement que l'ensemble du sous-arbre au dessus d'un sommet $x\in\tau_k^{-1}(a)\cup\tau_k^{-1}(b)$
est d\'eplac\'e avec $x$.
Le morphisme se comprend donc globalement comme une fusion des sommets $a$ et $b = a+1$
et d'une permutation de battage des sous-arbres au dessus des sommets $x\in\tau_k^{-1}(a)\cup\tau_k^{-1}(b)$
que repr\'esente l'expression $\sh(\alpha,\beta)$ dans la figure.
\begin{figure}[h]
\begin{equation*}
\vcenter{\xymatrix@W=0mm@H=0mm@R=4.5mm@C=4.5mm@M=0mm{ \ar@{}[d]|{\displaystyle\cdots} &
\ar@{-}[dr] & \ar@{}[d]|(0.3){\displaystyle\alpha} & \ar@{-}[dl]
&& \ar@{-}[dr] & \ar@{}[d]|(0.3){\displaystyle\beta} & \ar@{-}[dl] &
\ar@{}[d]|{\displaystyle\cdots} \\
\ar@{-}[drrrr] && \bullet\ar@{-}[drr]\ar@{.}[rr]\ar@{.}[ll] &&&&
\bullet\ar@{-}[dll]\ar@{.}[rr]\ar@{.}[ll] && \ar@{-}[dllll] \\
\ar@{.}[rrrr] &&&& \ar@{}[d]|{\vdots} &&&& \ar@{.}[llll] \\
&&&& &&&& }}
\quad\xrightarrow{u}
\quad\vcenter{\xymatrix@W=0mm@H=0mm@R=4.5mm@C=4.5mm@M=0mm{ \ar@{}[d]|{\displaystyle\cdots} &
\ar@{-}[drrr] &&& \ar@{}[d]|(0.1){\displaystyle\sh(\alpha,\beta)} &&& \ar@{-}[dlll] &
\ar@{}[d]|{\displaystyle\cdots} \\
\ar@{-}[drrrr] &&&& \bullet\ar@{-}[d]\ar@{.}[rrrr]\ar@{.}[llll] &&&& \ar@{-}[dllll] \\
\ar@{.}[rrrr] &&&& \ar@{}[d]|{\vdots} &&&& \ar@{.}[llll] \\
&&&& &&&& }}
\end{equation*}
\caption{}\label{Fig:LevelTreeMorphism}\end{figure}

Ces morphismes de degr\'e $1$ correspondent dans le formalisme de~\cite[\S A.8]{FresseIteratedBar}
aux relations de recouvrement du poset $\eset/\Omega_n^{epi}$.

\subsubsection{Signe des morphismes de degr\'e $1$}\label{KoszulConstruction:Signs}
Les morphismes de degr\'e $1$, avec un signe $\sgn(u)$ associ\'e \`a chaque morphisme $u$,
d\'eterminent les termes de la diff\'erentielle du complexe bar $n$-it\'er\'e $B^n(A)$
d'une alg\`ebre commutative $A$ (voir~\cite[Proposition A.10]{FresseIteratedBar}).
On utilise la relation de la cat\'egorie $\Omega_n^{epi}$
avec le complexe bar $n$-it\'er\'e $B^n(A)$
pour d\'efinir le signe $\sgn(u)$ associ\'e \`a chaque morphisme $u$
de fa\c con indirecte dans~\cite{FresseIteratedBar},
en renvoyant aux d\'efinitions de~\cite[\S X.12]{MacLane}
pour la construction bar des alg\`ebres commutatives.
Le but de ce paragraphe est de donner une d\'efinition directe de ce signe
et dans un contexte g\'en\'eralis\'e.

On d\'efinit en fait un signe pour chaque morphisme de $\eset/\Omega_n^{epi}$,
pour un ensemble d'entr\'ees donn\'e $\eset$,
avec un degr\'e $\deg(e)$ associ\'e \`a chaque $e\in\eset$.
Cette convention interviendra au~\S\ref{KoszulAcyclicity}.
On pourra supposer que ce degr\'e est nul jusque l\`a, ce que l'on fera tacitement
dans les constructions qui suivront ce paragraphe,
et on appliquera la d\'efinition de $\sgn(u)$ aux morphismes $u: \tauobj\rightarrow\sigmaobj$
de $\Omega_n^{epi}$
en prenant $\eset = \In\tauobj$
comme ensemble d'entr\'ees.

On se donne un objet $\tauobj\in\eset/\Omega_n^{epi}$.
On associe un symbole de degr\'e $1$ \`a chaque sommet $x\in\tset_i$
de niveau $i<n$ de $\tauobj$
et on suppose que les entr\'ees $e\in\eset$ de $\tauobj\in\eset/\Omega_n^{epi}$
sont munies du degr\'e $\deg(e)$
qui leur est associ\'e.
On ordonne ces \'el\'ements en parcourant l'arbre sommet par sommet,
en suivant les ar\^etes du bas vers le haut,
puis de gauche \`a droite,
pour former un tenseur diff\'erentiel gradu\'e associ\'e \`a $\tauobj$.
On fixera simplement un ordre arbitraire pour les entr\'ees $e\in\eset$
sur une m\^eme fibre $\tau_0^{-1}(k)$, $k\in\tset_0$,
qui peuvent toujours \^etre permut\'ees dans la suite.
Pour l'arbre de la Figure~\ref{Fig:LabelLevelTree},
on obtient ainsi l'ordre
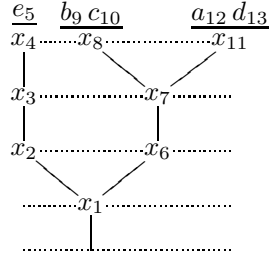
\begin{equation*}
\xymatrix@W=0mm@H=0mm@R=4.5mm@C=4.5mm@M=0mm{ &&& \\
*+<2pt>{x_4}\ar@{-}[d]\ar@{.}[r]\ar@{}[u]|(0.6){\displaystyle\underline{e_5}} &
*+<2pt>{x_8}\ar@{-}[dr]\ar@{.}[rr]\ar@{}[u]|(0.6){\displaystyle\underline{b_{9}\,c_{10}}} &&
*+<2pt>{x_{11}}\ar@{-}[dl]\ar@{}[u]|(0.6){\displaystyle\underline{a_{12}\,d_{13}}} \\
*+<2pt>{x_3}\ar@{-}[d]\ar@{.}[rr] && *+<2pt>{x_7}\ar@{-}[d]\ar@{.}[r] & \\
*+<2pt>{x_2}\ar@{-}[dr]\ar@{.}[rr] && *+<2pt>{x_6}\ar@{-}[dl]\ar@{.}[r] & \\
\ar@{.}[r] & *+<2pt>{x_1}\ar@{-}[d]\ar@{.}[rr] && \\
\ar@{.}[rrr] &&& }
\end{equation*}
et un tenseur de la forme
\begin{equation*}
x_1\otimes x_2\otimes x_3\otimes x_4\otimes e_5\otimes x_6\otimes x_7\otimes x_8\otimes b_9\otimes c_{10}\otimes x_{11}\otimes a_{12}\otimes d_{13}.
\end{equation*}

On se donne un morphisme $u: \tauobj\rightarrow\sigmaobj$
de degr\'e $1$
fusionnant des sommets $(x_i,x_j)$ dans $\tauobj$.
Le signe $\sgn(u)$ se d\'etermine par l'application des r\`egles de l'alg\`ebre diff\'e\-ren\-tielle gradu\'ee
aux permutations de tenseurs impliqu\'ees dans l'op\'eration de fusion d\'efinie par~$u$.
De fait, cette op\'eration de fusion se d\'ecompose comme suit~:
\begin{enumerate}
\item
on part du tenseur d\'efini par l'ordre de l'arbre source $\tauobj$,
\item\label{MergingOrdering}
on place les sommets $(x_i,x_j)$ en premi\`ere position par une permutation de tenseurs,
\item
on effectue la fusion,
\item\label{MergedReordering}
puis on effectue une nouvelle permutation de tenseurs
pour replacer les facteurs du produit tensoriel
dans l'ordre fix\'e par l'arbre but $\sigmaobj$.
\end{enumerate}
Le signe $\sgn(u)$ est produit par les op\'erations~(\ref{MergingOrdering})
et~(\ref{MergedReordering}).

Le signe associ\'e au morphisme
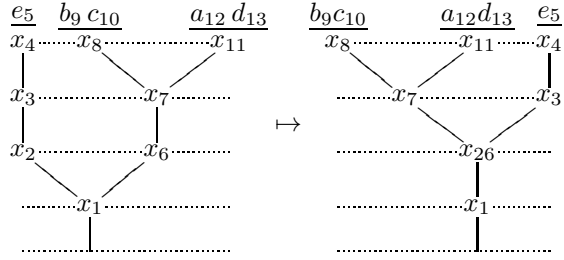
\begin{equation*}
\vcenter{\xymatrix@W=0mm@H=0mm@R=4.5mm@C=4.5mm@M=0mm{ &&& \\
*+<2pt>{x_4}\ar@{-}[d]\ar@{.}[r]\ar@{}[u]|(0.6){\displaystyle\underline{e_5}} &
*+<2pt>{x_8}\ar@{-}[dr]\ar@{.}[rr]\ar@{}[u]|(0.6){\displaystyle\underline{b_{9}\,c_{10}}} &&
*+<2pt>{x_{11}}\ar@{-}[dl]\ar@{}[u]|(0.6){\displaystyle\underline{a_{12}\,d_{13}}} \\
*+<2pt>{x_3}\ar@{-}[d]\ar@{.}[rr] && *+<2pt>{x_7}\ar@{-}[d]\ar@{.}[r] & \\
*+<2pt>{x_2}\ar@{-}[dr]\ar@{.}[rr] && *+<2pt>{x_6}\ar@{-}[dl]\ar@{.}[r] & \\
\ar@{.}[r] & *+<2pt>{x_1}\ar@{-}[d]\ar@{.}[rr] && \\
\ar@{.}[rrr] &&& }}
\mapsto\vcenter{\xymatrix@W=0mm@H=0mm@R=4.5mm@C=4.5mm@M=0mm{ &&& \\
*+<2pt>{x_8}\ar@{-}[dr]\ar@{.}[rr]\ar@{}[u]|(0.6){\displaystyle\underline{b_{9} c_{10}}} &&
*+<2pt>{x_{11}}\ar@{-}[dl]\ar@{.}[r]\ar@{}[u]|(0.6){\displaystyle\underline{a_{12} d_{13}}} &
*+<2pt>{x_4}\ar@{-}[d]\ar@{}[u]|(0.6){\displaystyle\underline{e_5}} \\
\ar@{.}[r] & *+<2pt>{x_7}\ar@{-}[dr]\ar@{.}[rr] && *+<2pt>{x_3}\ar@{-}[dl] \\
\ar@{.}[rr] && *+<2pt>{x_{2 6}}\ar@{-}[d]\ar@{.}[r] & \\
\ar@{.}[rr] && *+<2pt>{x_1}\ar@{-}[d]\ar@{.}[r] & \\
\ar@{.}[rrr] &&& }}
\end{equation*}
est ainsi d\'etermin\'e par les op\'erations
\begin{align*}
&&& x_1\otimes x_2\otimes x_3\otimes x_4\otimes e_5\otimes x_6
\otimes x_7\otimes x_8\otimes b_9\otimes c_{10}\otimes x_{11}\otimes a_{12}\otimes d_{13} \\
& \simeq && x_2\otimes x_6\otimes x_1\otimes x_3\otimes x_4\otimes e_5
\otimes x_7\otimes x_8\otimes b_9\otimes c_{10}\otimes x_{11}\otimes a_{12}\otimes d_{13} \\
& \mapsto && x_{2 6}\otimes x_1\otimes x_3\otimes x_4\otimes e_5
\otimes x_7\otimes x_8\otimes b_9\otimes c_{10}\otimes x_{11}\otimes a_{12}\otimes d_{13} \\
& \simeq && x_1\otimes x_{2 6}\otimes x_7\otimes x_8\otimes b_9\otimes c_{10}
\otimes x_{11}\otimes a_{12}\otimes d_{13}\otimes x_3\otimes x_4\otimes e_5.
\end{align*}
La premi\`ere permutation de tenseurs produit un signe $(-1)^p$ d'exposant $p = 4+\deg(e)$,
la seconde un signe $(-1)^q$ d'exposant $q = 1+(2+\deg(e))\cdot(3+\deg(a)+\deg(b)+\deg(c)+\deg(d))$,
et on somme ces exposants pour obtenir le signe de~$u$.

\subsubsection{La construction de Koszul associ\'ee \`a la cat\'egorie $\Omega_n^{epi}$}\label{KoszulConstruction:KoszulCategory}
On forme le dg-graphe $K(\Omega_n^{epi})$
d\'efinis par les dg-modules
\begin{equation*}
K(\Omega_n^{epi})(\tauobj,\sigmaobj) = \bigoplus_{u\in\Mor_{\Omega_n^{epi}}(\tauobj,\sigmaobj)}\kk\{u\},
\end{equation*}
avec des \'el\'ements g\'en\'erateurs $\{u\}$
munis d'une diff\'erentielle triviale
et du degr\'e
\begin{equation*}
\deg\{u\} = \deg(\tauobj) - \deg(\sigmaobj)
\end{equation*}
d\'eduit de la graduation des arbres \`a niveau.
On appelle ce dg-graphe la construction de Koszul de $\Omega_n^{epi}$.

On consid\`ere l'application
\begin{equation*}
K(\Omega_n^{epi})(\tauobj,\sigmaobj)
\xrightarrow{\iota}B(\Omega_n^{epi})(\tauobj,\sigmaobj)
\end{equation*}
qui applique un \'el\'ement g\'en\'erateur $\{u\}\in K(\Omega_n^{epi})(\tauobj,\sigmaobj)$
sur la somme
\begin{equation*}
\iota\{u\} = \sum_{\substack{u = u_1\cdots u_d\\ \deg u_1 = \dots = \deg u_d = 1}}
\underbrace{\sgn(u_1)\cdots\sgn(u_d)}_{\pm}\cdot\,\bigl\{\tauobj_0\xleftarrow{u_1}\cdots\xleftarrow{u_d}\tauobj_d\bigr\}
\end{equation*}
s'\'etendant sur l'ensemble des d\'ecompositions de $u$
en morphismes de degr\'e $1$
dans la cat\'egorie ensembliste $\Omega_n^{epi}$.
On observe que~:

\begin{prop}\label{KoszulConstruction:KoszulEmbedding}
On a $\partial\iota\{u\} = 0$ dans $B(\Omega_n^{epi})$,
de sorte que notre application
d\'efinit un morphisme de dg-modules
\begin{equation*}
K(\Omega_n^{epi})(\tauobj,\sigmaobj)\xrightarrow{\iota}B(\Omega_n^{epi})(\tauobj,\sigmaobj),
\end{equation*}
pour tout couple $(\tauobj,\sigmaobj)\in\Ob\Omega_n^{epi}\times\Ob\Omega_n^{epi}$.
\end{prop}

\begin{proof}
On prouve que tout morphisme de degr\'e $2$
\begin{equation*}
\tauobj\xrightarrow{u}\sigmaobj
\end{equation*}
poss\`ede exactement deux d\'ecompositions
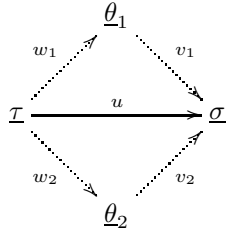
\begin{equation*}
\xymatrix{ & \thetaobj_1\ar@{.>}[dr]^{v_1} & \\
\tauobj\ar@{.>}[ur]^{w_1}\ar@{.>}[dr]_{w_2}\ar[rr]^{u} && \sigmaobj \\
& \thetaobj_2\ar@{.>}[ur]_{v_2} & }
\end{equation*}
avec $\deg v_1 = \deg v_2 = \deg w_1 = \deg w_2 = 1$,
puis on constate que les signes
\begin{equation*}
\sgn(v_1)\sgn(w_1)\quad\text{et}\quad\sgn(v_2)\sgn(w_2)
\end{equation*}
associ\'es \`a ces d\'ecompositions sont oppos\'es.
Ceci suffit pour montrer que les termes du d\'eve\-lop\-pement
\begin{equation*}
\partial\iota\{u\} = \sum_{\substack{u = u_1\cdots u_d\\ \deg u_1 = \dots = \deg u_d = 1\\1\leq i\leq d-1}}
\underbrace{\sgn(u_1)\cdots\sgn(u_d)}_{\pm}
\cdot\,\bigl\{\tauobj_0\xleftarrow{u_1}\cdots\xleftarrow{u_{i} u_{i+1}}\cdots
\xleftarrow{u_d}\tauobj_d\bigr\}
\end{equation*}
d'un morphisme $u$ de degr\'e $d$ quelconque s'annulent deux \`a deux.

On g\'en\'eralise l'analyse du~\S\ref{KoszulConstruction:IndecomposableMorphisms}
pour d\'eterminer un morphisme de degr\'e $2$
\`a partir de son domaine $\tauobj$
et des sommets de $\tauobj$ identifi\'es par $u$.
On distingue trois cas.

\begin{figure}[t]
\begin{equation*}
\xymatrix@W=0mm@H=0mm@R=4.5mm@C=3mm@M=0mm{
&&&&&& &&&&&& 
&& %
\save[].[rrrrrrrrrrddd]!C="g21" *+<16pt>[F-,]\frm{}\restore 
\ar@{}[d]|{\displaystyle\cdots}
& \ar@{-}[drr] && \ar@{}[d]|<{\sh(\alpha,\beta)} && \ar@{-}[dll] &
& \ar@{-}[dr] & \ar@{}[d]|(0.3){\gamma} & \ar@{-}[dl] &
\ar@{}[d]|{\displaystyle\cdots}
&& %
&&&& &&&& 
\\
&&&&&& &&&&&& 
&& %
\ar@{-}[drrrrrr] 
&&& \bullet\ar@{-}[drrr]\ar@{.}[rrr]\ar@{.}[lll] &&&
&& \bullet\ar@{-}[dll]\ar@{.}[rr]\ar@{.}[ll] &&
\ar@{-}[dllll]
&& %
&&&& &&&& 
\\
&&&&&& &&&&&& 
&& %
\ar@{.}[rrrrrr] &&&&&& \ar@{}[d]|{\vdots} &&&& \ar@{.}[llll] 
&& %
&&&& &&&& 
\\
&&&&&& &&&&&& 
&& %
&&&&&& &&&& 
&& %
&&&& &&&& 
\\
&&&&&& &&&&&& 
&& %
&&&&&& &&&& 
&& %
&&&& &&&& 
\\
\save[].[rrrrrrrrrrrrddd]!C="g1" *+<16pt>[F-,]\frm{}\restore 
\ar@{}[d]|{\displaystyle\cdots}
& \ar@{-}[dr] & \ar@{}[d]|(0.3){\alpha} & \ar@{-}[dl] &
& \ar@{-}[dr] & \ar@{}[d]|(0.3){\beta} & \ar@{-}[dl] &
& \ar@{-}[dr] & \ar@{}[d]|(0.3){\gamma} & \ar@{-}[dl] &
\ar@{}[d]|{\displaystyle\cdots}
&& %
&&&&&& &&&& 
&& %
\save[].[rrrrrrrrddd]!C="g3" *+<16pt>[F-,]\frm{}\restore 
\ar@{}[d]|{\displaystyle\cdots} &
\ar@{-}[drrr] &&& \ar@{}[d]|<{\sh(\alpha,\beta,\gamma)} &&& \ar@{-}[dlll] &
\ar@{}[d]|{\displaystyle\cdots}
&& %
&&&& &&&& 
\\
\ar@{-}[drrrrrr] 
&& \bullet\ar@{-}[drrrr]\ar@{.}[rr]\ar@{.}[ll] &&
&& \bullet\ar@{-}[d]\ar@{.}[rr]\ar@{.}[ll] &&
&& \bullet\ar@{-}[dllll]\ar@{.}[rr]\ar@{.}[ll] &&
\ar@{-}[dllllll]
&& %
&&&&&& &&&& 
&& %
\ar@{-}[drrrr] &&&& \bullet\ar@{-}[d]\ar@{.}[rrrr]\ar@{.}[llll] &&&& \ar@{-}[dllll] 
\\
\ar@{.}[rrrrrr] &&&&&& \ar@{}[d]|{\vdots} &&&&&& \ar@{.}[llllll] 
&& %
&&&&&& &&&& 
&& %
\ar@{.}[rrrr] &&&& \ar@{}[d]|{\vdots} &&&& \ar@{.}[llll] 
\\
&&&&&& &&&&&& 
&& %
&&&&&& &&&& 
&& %
&&&& &&&& 
\\
&&&&&& &&&&&& 
&& %
&&&&&& &&&& 
&& %
&&&& &&&& 
\\
&&&&&& &&&&&& 
&& %
\save[].[rrrrrrrrrrddd]!C="g22" *+<16pt>[F-,]\frm{}\restore 
\ar@{}[d]|{\displaystyle\cdots}
& \ar@{-}[dr] & \ar@{}[d]|(0.3){\alpha} & \ar@{-}[dl] &
& \ar@{-}[drr] && \ar@{}[d]|<{\sh(\beta,\gamma)} && \ar@{-}[dll] &
\ar@{}[d]|{\displaystyle\cdots}
&& %
&&&& &&&& 
\\
&&&&&& &&&&&& 
&& %
\ar@{-}[drrrr] 
&& \bullet\ar@{-}[drr]\ar@{.}[rr]\ar@{.}[ll] &&
&&& \bullet\ar@{-}[dlll]\ar@{.}[rrr]\ar@{.}[lll] &&&
\ar@{-}[dllllll]
&& %
&&&& &&&& 
\\
&&&&&& &&&&&& 
&& %
\ar@{.}[rrrr] &&&& \ar@{}[d]|{\vdots} &&&&&& \ar@{.}[llllll] 
&& %
&&&& &&&& 
\\
&&&&&& &&&&&& 
&& %
&&&& &&&&&& 
&& %
&&&& &&&& 
\ar "g1"!CR+<16pt,0pt>;"g3"!CL-<16pt,0pt>
\ar@{.>} "g1"!UC+<16pt,16pt>;"g21"!CL-<16pt,16pt>
\ar@{.>} "g1"!DC+<16pt,-16pt>;"g22"!CL-<16pt,-16pt>
\ar@{.>} "g21"!CR+<16pt,-16pt>;"g3"!UC+<-16pt,16pt>
\ar@{.>} "g22"!CR+<16pt,16pt>;"g3"!DC+<-16pt,-16pt>
}
\end{equation*}
\caption{}\label{Fig:ConsecutiveVertexMerging}\end{figure}

\begin{figure}[t]
\begin{equation*}
\xymatrix@W=0mm@H=0mm@R=4.5mm@C=3mm@M=0mm{
&& &&&& &&&& 
&& 
\save[].[rrrrrrrrrrdddddd]!C="g21" *+<16pt>[F-,]\frm{}\restore && 
\ar@{}[d]|{\cdots}
& \ar@{-}[drrr] &&& \ar@{}[d]|(0.3){\sh(\gamma,\delta)} &&& \ar@{-}[dlll] &
\ar@{}[d]|{\cdots}
&& 
&&&& &&&& 
\\
&& &&&& &&&& 
&& 
&& 
\ar@{-}[drrrr]
&&&& \bullet\ar@{-}[d]\ar@{.}[rrrr]\ar@{.}[llll] &&&&
\ar@{-}[dllll]
&& 
&&&& &&&& 
\\
&& &&&& &&&& 
&& 
&& \ar@{.}[rrrr] &&&& \ar@{}[d]|(0.3){\vdots} &&&& \ar@{.}[llll] 
&& 
&&&& &&&& 
\\
&& &&&& &&&& 
&& 
\ar@{}[d]|{\cdots} 
& \ar@{-}[dr] & \ar@{}[d]|(0.3){\alpha} & \ar@{-}[dl] &
& \ar@{-}[dr] & \ar@{}[d]|(0.3){\beta} & \ar@{-}[dl] &
\ar@{}[d]|{\cdots}
&&
&& 
&&&& &&&& 
\\
&& &&&& &&&& 
&& 
\ar@{-}[drrrr] 
&& \bullet\ar@{-}[drr]\ar@{.}[rr]\ar@{.}[ll] &&
&& \bullet\ar@{-}[dll]\ar@{.}[rr]\ar@{.}[ll] &&
\ar@{-}[dllll]
&&
&& 
&&&& &&&& 
\\
&& &&&& &&&& 
&& 
\ar@{.}[rrrr] &&&& \ar@{}[d]|(0.3){\vdots} &&&& \ar@{.}[llll] 
&&
&& 
&&&& &&&& 
\\
&& &&&& &&&& 
&& 
&& &&&& &&&& 
&& 
&&&& &&&& 
\\
&& &&&& &&&& 
&& 
&& &&&& &&&& 
&& 
&&&& &&&& 
\\
\save[].[rrrrrrrrrrdddddd]!C="g1" *+<16pt>[F-,]\frm{}\restore && 
\ar@{}[d]|{\cdots}
& \ar@{-}[dr] & \ar@{}[d]|(0.3){\gamma} & \ar@{-}[dl] &
& \ar@{-}[dr] & \ar@{}[d]|(0.3){\delta} & \ar@{-}[dl] &
\ar@{}[d]|{\cdots}
&& 
&& &&&& &&&& 
&& 
\save[].[rrrrrrrrdddddd]!C="g3" *+<16pt>[F-,]\frm{}\restore 
\ar@{}[d]|{\cdots}
& \ar@{-}[drrr] &&& \ar@{}[d]|(0.2){\sh(\gamma,\delta)} &&& \ar@{-}[dlll] &
\ar@{}[d]|{\cdots}
\\
&& 
\ar@{-}[drrrr]
&& \bullet\ar@{-}[drr]\ar@{.}[rr]\ar@{.}[ll] &&
&& \bullet\ar@{-}[dll]\ar@{.}[rr]\ar@{.}[ll] &&
\ar@{-}[dllll]
&& 
&& &&&& &&&& 
&& 
\ar@{-}[drrrr] 
&&&& \bullet\ar@{-}[d]\ar@{.}[rrrr]\ar@{.}[llll] &&&&
\ar@{-}[dllll]
\\
&& \ar@{.}[rrrr] &&&& \ar@{}[d]|(0.3){\vdots} &&&& \ar@{.}[llll] 
&& 
&& &&&& &&&& 
&& 
\ar@{.}[rrrr] &&&& \ar@{}[d]|(0.3){\vdots} &&&& \ar@{.}[llll] 
\\
\ar@{}[d]|{\cdots} 
& \ar@{-}[dr] & \ar@{}[d]|(0.3){\alpha} & \ar@{-}[dl] &
& \ar@{-}[dr] & \ar@{}[d]|(0.3){\beta} & \ar@{-}[dl] &
\ar@{}[d]|{\cdots}
&&
&& 
&& &&&& &&&& 
&& 
\ar@{}[d]|{\cdots} 
& \ar@{-}[drrr] &&& \ar@{}[d]|(0.2){\sh(\alpha,\beta)} &&& \ar@{-}[dlll] &
\ar@{}[d]|{\cdots}
\\
\ar@{-}[drrrr] 
&& \bullet\ar@{-}[drr]\ar@{.}[rr]\ar@{.}[ll] &&
&& \bullet\ar@{-}[dll]\ar@{.}[rr]\ar@{.}[ll] &&
\ar@{-}[dllll]
&&
&& 
&& &&&& &&&& 
&& 
\ar@{-}[drrrr] 
&&&& \bullet\ar@{-}[d]\ar@{.}[rrrr]\ar@{.}[llll] &&&&
\ar@{-}[dllll]
\\
\ar@{.}[rrrr] &&&& \ar@{}[d]|(0.3){\vdots} &&&& \ar@{.}[llll] 
&&
&& 
&& &&&& &&&& 
&& 
\ar@{.}[rrrr] &&&& \ar@{}[d]|(0.3){\vdots} &&&& \ar@{.}[llll] 
\\
&& &&&& &&&& 
&& 
&& &&&& &&&& 
&& 
&&&& &&&& 
\\
&& &&&& &&&& 
&& 
&& &&&& &&&& 
&& 
&&&& &&&& 
\\
&& &&&& &&&& 
&& 
&
\save[].[rrrrrrrrdddddd]!C="g22" *+<16pt>[F-,]\frm{}\restore 
\ar@{}[d]|{\cdots}
& \ar@{-}[dr] & \ar@{}[d]|(0.3){\gamma} & \ar@{-}[dl] &
& \ar@{-}[dr] & \ar@{}[d]|(0.3){\delta} & \ar@{-}[dl] &
\ar@{}[d]|{\cdots}
&
&& 
&&&& &&&& 
\\
&& &&&& &&&& 
&& 
&
\ar@{-}[drrrr] 
&& \bullet\ar@{-}[drr]\ar@{.}[rr]\ar@{.}[ll] &&
&& \bullet\ar@{-}[dll]\ar@{.}[rr]\ar@{.}[ll] &&
\ar@{-}[dllll]
&
&& 
&&&& &&&& 
\\
&& &&&& &&&& 
&& 
&
\ar@{.}[rrrr] &&&& \ar@{}[d]|(0.3){\vdots} &&&& \ar@{.}[llll] 
&
&& 
&&&& &&&& 
\\
&& &&&& &&&& 
&& 
&
\ar@{}[d]|{\cdots} 
& \ar@{-}[drrr] &&& \ar@{}[d]|(0.2){\sh(\alpha,\beta)} &&& \ar@{-}[dlll] &
\ar@{}[d]|{\cdots}
&
&& 
&&&& &&&& 
\\
&& &&&& &&&& 
&& 
&
\ar@{-}[drrrr] 
&&&& \bullet\ar@{-}[d]\ar@{.}[rrrr]\ar@{.}[llll] &&&&
\ar@{-}[dllll]
& 
&&&& &&&& 
\\
&& &&&& &&&& 
&& 
&
\ar@{.}[rrrr] &&&& \ar@{}[d]|(0.3){\vdots} &&&& \ar@{.}[llll] 
&
&& 
&&&& &&&& 
\\
&& &&&& &&&& 
&& 
&
&&&& &&&& 
&
&& 
&&&& &&&& 
\ar "g1"!CR+<16pt,0pt>;"g3"!CL-<16pt,0pt>
\ar@{.>} "g1"!UC+<16pt,16pt>;"g21"!CL-<16pt,16pt>
\ar@{.>} "g1"!DC+<16pt,-16pt>;"g22"!CL-<16pt,-16pt>
\ar@{.>} "g21"!CR+<16pt,-16pt>;"g3"!UC+<-16pt,16pt>
\ar@{.>} "g22"!CR+<16pt,16pt>;"g3"!DC+<-16pt,-16pt>
}
\end{equation*}
\caption{}\label{Fig:DisjointVertexMerging}\end{figure}

\begin{figure}[t]
\begin{equation*}
\xymatrix@W=0mm@H=0mm@R=4.5mm@C=3mm@M=0mm{
&&&&&&& &&&&&&& 
\save[].[rrrrrrrrrrrrdddd]!C="g21" *+<16pt>[F-,]\frm{}\restore 
\ar@{}[drr]|(0.7){\sh(\alpha',\beta')}
&&
& \ar@{-}[dr] & \ar@{}[d]|(0.3){\gamma} & \ar@{-}[dl] &
& \ar@{-}[dr] & \ar@{}[d]|(0.3){\delta} & \ar@{-}[dl] &
&&
\ar@{}[dll]|(0.7){\sh(\alpha'',\beta'')}
&&&&& &&&&& 
\\
&&&&&&& &&&&&&& 
\ar@{-}[drrrrrr] &&& \ar@{-}[drrr] 
& \bullet\ar@{-}[drr]\ar@{.}[rr]\ar@{.}[l] &&
&& \bullet\ar@{-}[dll]\ar@{.}[r]\ar@{.}[ll] &
\ar@{-}[dlll] &&& \ar@{-}[dllllll]
&&&&& &&&&& 
\\
&&&&&&& &&&&&&& 
&&&&&& \bullet\ar@{-}[d]\ar@{.}[rrrrrr]\ar@{.}[llllll] &&&&&& 
&&&&& &&&&& 
\\
&&&&&&& &&&&&&& 
&&&&&& \ar@{}[d]|{\vdots}\ar@{.}[rrrrrr]\ar@{.}[llllll] &&&&&& 
&&&&& &&&&& 
\\
&&&&&&& &&&&&&& 
&&&&&& &&&&&& 
&&&&& &&&&& 
\\
&&&&&&& &&&&&&& 
&&&&&& &&&&&& 
&&&&& &&&&& 
\\
\save[].[rrrrrrrrrrrrrrdddd]!C="g1" *+<16pt>[F-,]\frm{}\restore 
&
& \ar@{-}[dr] & \ar@{}[d]|(0.3){\gamma} & \ar@{-}[dl] &
&&&&
& \ar@{-}[dr] & \ar@{}[d]|(0.3){\delta} & \ar@{-}[dl] &
&
&&&&&& &&&&&& 
\save[].[rrrrrrrrrrdddd]!C="g3" *+<16pt>[F-,]\frm{}\restore 
\ar@{}[drr]|(0.7){\sh(\alpha',\beta')}
&&
& \ar@{-}[drr] && \ar@{}[d]|(0.1){\sh(\gamma,\delta)} && \ar@{-}[dll] &
&&
\ar@{}[dll]|(0.7){\sh(\alpha'',\beta'')}
\\
\ar@{-}[drrr] & \ar@{}[drr]|(0.05){\alpha'} & \ar@{-}[dr] 
& \bullet\ar@{-}[d]\ar@{.}[r]\ar@{.}[l] &
\ar@{-}[dl] & \ar@{}[dll]|(0.05){\alpha''} & \ar@{-}[dlll]
&&
\ar@{-}[drrr] & \ar@{}[drr]|(0.05){\beta'} & \ar@{-}[dr]
& \bullet\ar@{-}[d]\ar@{.}[r]\ar@{.}[l] &
\ar@{-}[dl] & \ar@{}[dll]|(0.05){\beta''} & \ar@{-}[dlll]
&&&&&& &&&&&& 
\ar@{-}[drrrrr] &&& \ar@{-}[drr] 
&& \bullet\ar@{-}[d]\ar@{.}[rr]\ar@{.}[ll] &&
\ar@{-}[dll] &&& \ar@{-}[dlllll]
\\
& 
&& \bullet\ar@{-}[drrrr]\ar@{.}[rrrr]\ar@{.}[lll] &&
&&&&
&& \bullet\ar@{-}[dllll]\ar@{.}[rrr]\ar@{.}[llll] &&
&
&&&&&& &&&&&& 
&&&&& \bullet\ar@{-}[d]\ar@{.}[rrrrr]\ar@{.}[lllll] &&&&& 
\\
&&&&&&& \ar@{}[d]|{\vdots}\ar@{.}[rrrrrrr]\ar@{.}[lllllll] &&&&&&& 
&&&&&& &&&&&& 
&&&&& \ar@{}[d]|{\vdots}\ar@{.}[rrrrr]\ar@{.}[lllll] &&&&& 
\\
&&&&&&& &&&&&&& 
&&&&&& &&&&&& 
&&&&& &&&&& 
\\
&&&&&&& &&&&&&& 
&&&&&& &&&&&& 
&&&&& &&&&& 
\\
&&&&&&& &&&&&&& 
\save[].[rrrrrrrrrrrrdddd]!C="g22" *+<16pt>[F-,]\frm{}\restore 
\ar@{}[drr]|(0.7){\sh(\alpha',\beta')}
&&
& \ar@{-}[dr] & \ar@{}[d]|(0.3){\delta} & \ar@{-}[dl] &
& \ar@{-}[dr] & \ar@{}[d]|(0.3){\gamma} & \ar@{-}[dl] &
&&
\ar@{}[dll]|(0.7){\sh(\alpha'',\beta'')}
&&&&& &&&&& 
\\
&&&&&&& &&&&&&& 
\ar@{-}[drrrrrr] &&& \ar@{-}[drrr] 
& \bullet\ar@{-}[drr]\ar@{.}[rr]\ar@{.}[l] &&
&& \bullet\ar@{-}[dll]\ar@{.}[r]\ar@{.}[ll] &
\ar@{-}[dlll] &&& \ar@{-}[dllllll]
&&&&& &&&&& 
\\
&&&&&&& &&&&&&& 
&&&&&& \bullet\ar@{-}[d]\ar@{.}[rrrrrr]\ar@{.}[llllll] &&&&&& 
&&&&& &&&&& 
\\
&&&&&&& &&&&&&& 
&&&&&& \ar@{}[d]|{\vdots}\ar@{.}[rrrrrr]\ar@{.}[llllll] &&&&&& 
&&&&& &&&&& 
\\
&&&&&&& &&&&&&& 
&&&&&& &&&&&& 
&&&&& &&&&& 
\ar "g1"!CR+<16pt,0pt>;"g3"!CL-<16pt,0pt>
\ar@{.>} "g1"!UC+<16pt,16pt>;"g21"!CL-<16pt,16pt>
\ar@{.>} "g1"!DC+<16pt,-16pt>;"g22"!CL-<16pt,-16pt>
\ar@{.>} "g21"!CR+<16pt,-16pt>;"g3"!UC+<-16pt,16pt>
\ar@{.>} "g22"!CR+<16pt,16pt>;"g3"!DC+<-16pt,-16pt>
}
\end{equation*}
\caption{}\label{Fig:OrderedVertexMerging}\end{figure}

On suppose d'abord que $u$ fusionne trois sommets cons\'ecutifs $a$, $b = a+1$, $c = a+2$
sur une m\^eme fibre $\tau_k^{-1}(x)$.
La Figure~\ref{Fig:ConsecutiveVertexMerging}
donne la repr\'esentation sch\'ematique de ce morphisme
et ses deux d\'ecompositions.
L'expression $\sh(\alpha,\beta,\gamma)$ repr\'esente une permutation de battage des sous-arbres $(\alpha,\beta,\gamma)$.
L'existence des deux d\'ecompositions r\'esulte de l'existence d'uniques permutations de battages \`a deux composantes
telles qu'on a la relation
\begin{equation*}
\sh(\alpha\cup\beta,\gamma)\cdot\sh(\alpha,\beta)
= \sh(\alpha,\beta,\gamma)
= \sh(\alpha,\beta\cup\gamma)\cdot\sh(\beta,\gamma)
\end{equation*}
dans le groupe des permutations de l'alphabet $\alpha\cup\beta\cup\gamma$.

On suppose maintenant que $u$ fusionne un couple de sommets cons\'ecutifs $a$, $b = a+1$,
sur une m\^eme fibre $\tau_k^{-1}(x)$,
et un couple disjoint de sommets cons\'ecutifs $c$, $d = c+1$,
sur une fibre $\tau_l^{-1}(y)$ avec \'eventuellement $x\not=y$ et $k\not=l$.
La repr\'esentation sch\'ematique de ce morphisme
et de ses d\'ecompositions est donn\'ee Figure~\ref{Fig:DisjointVertexMerging}.

Un dernier cas correspond \`a la fusion d'un couple de sommets cons\'ecutifs $a$, $b = a+1$,
sur une m\^eme fibre $\tau_k^{-1}(x)$,
et d'un couple de sommets $c$ et $d$ tels que $c\in\tau_k^{-1}(a)$ et $d\in\tau_k^{-1}(b)$.
La repr\'esentation sch\'ematique de ce morphisme
et de ses d\'ecompositions est donn\'ee Figure~\ref{Fig:OrderedVertexMerging}.

On se convainc ais\'ement qu'il n'y a pas d'autres possibilit\'es
que ces trois configurations.

On v\'erifie ensuite, par simple inspection des r\`egles de permutation du~\S\ref{KoszulConstruction:Signs},
que les signes associ\'es aux deux d\'ecompositions $u = v_1 w_1 = v_2 w_2$
d'un morphisme de degr\'e $2$ sont oppos\'es dans chaque cas
et ceci termine la d\'emonstration de la proposition.
\end{proof}

Notre objectif principal consiste maintenant \`a montrer que ce morphisme $\iota: K(\Omega_n^{epi})\rightarrow B(\Omega_n^{epi})$
d\'efinit une \'equivalence faible.
On introduit une version \`a coefficients des constructions $K(\Omega_n^{epi})$ et $B(\Omega_n^{epi})$
afin de d\'emontrer ce r\'esultat de fa\c con indirecte,
par des arguments de comparaison de complexes.

\subsection{Complexes \`a coefficients}\label{CoefficientConstructions}
Les versions \`a coefficients de $K(\Omega_n^{epi})$ et $B(\Omega_n^{epi})$
que l'on consid\`ere sont d\'efinies sur les cat\'egories de $\Omega_n^{epi}$-diagramme
covariants et contravariants.
Pour le moment,
on ne consid\`ere que des $\Omega_n^{epi}$-diagrammes en $\kk$-modules, pour lesquels les d\'efinitions
les plus classiques de l'alg\`ebre homologique s'appliquent~:
rappelons simplement que la cat\'egorie des $\Omega_n^{epi}$-diagrammes covariants en $\kk$-modules
(et la cat\'egorie des $\Omega_n^{epi}$-diagrammes contravariants de m\^eme)
est ab\'elienne
et poss\`ede un ensemble de g\'en\'erateurs projectifs
d\'efinis par les diagrammes de Yoneda $\Omega_n^{epi}(\tauobj,-)$.
On \'etendra nos constructions au cadre diff\'erentiel gradu\'e dans la section suivante.

\subsubsection{Constructions \`a coefficients}\label{CoefficientConstructions:BasicDefinition}
On se donne donc un $\Omega_n^{epi}$-diagramme contravariant $S$
et un $\Omega_n^{epi}$-diagramme covariant $T$
\`a valeurs dans les $\kk$-modules.
Le complexe bar \`a coefficients $B(S,\Omega_n^{epi},T)$
est d\'efini par le $\kk$-module
\begin{equation*}
B(S,\Omega_n^{epi},T)_d
= \bigoplus_{(u_1,\dots,u_d)}
S(\tauobj_0)\otimes\kk\{\tauobj_0\xleftarrow{u_1}\cdots\xleftarrow{u_d}\tauobj_d\}\otimes S(\tauobj_d)
\end{equation*}
en degr\'e $d$
avec la diff\'erentielle donn\'ee par la formule
\begin{multline*}
\partial(x\otimes\{\tauobj_0\xleftarrow{u_1}\cdots\xleftarrow{u_d}\tauobj_d\}\otimes y)
= u_1^*(x)\otimes\{\tauobj_1\xleftarrow{u_2}\cdots\xleftarrow{u_d}\tauobj_d\}\otimes y\\
+ \sum_{i=1}^{d-1} (-1)^i x\otimes\{\tauobj_0\xleftarrow{u_1}\cdots
\xleftarrow{u_i u_{i+1}}\cdots
\xleftarrow{u_d}\tauobj_d\}\otimes y\\
+ (-1)^d x\otimes\{\tauobj_0\xleftarrow{u_1}\cdots\xleftarrow{u_{d-1}}\tauobj_{d-1}\}\otimes u_d{}_*(y).
\end{multline*}
Le complexe de Koszul \`a coefficients $K(S,\Omega_n^{epi},T)$
est d\'efini par les $\kk$-modules
\begin{equation*}
K(S,\Omega_n^{epi},T) = \bigoplus_{\substack{ u: \tauobj\rightarrow\sigmaobj\\
(\sigmaobj,\tauobj)\in\Ob\Omega_n^{epi}\times\Ob\Omega_n^{epi}}}
S(\sigmaobj)\otimes\{u\}\otimes T(\tauobj)
\end{equation*}
munis de la graduation telle que $\deg(x\otimes\{u\}\otimes y) = \deg(u) = \deg(\tauobj)-\deg(\sigmaobj)$,
avec la diff\'erentielle donn\'ee par la formule
\begin{multline*}
\partial(x\otimes\{u\}\otimes y) = \sum_{\substack{u = v w\\
\deg(v) = 1}}
\sgn(v)\cdot v^*(x)\otimes\{w\}\otimes y\\ +
\sum_{\substack{u = v w\\
\deg(w) = 1}}
(-1)^{\deg(v)} \sgn(w)\cdot x\otimes\{v\}\otimes w_*(y)),
\end{multline*}
pour tout $x\in S(\sigmaobj)$, $y\in T(\tauobj)$, et tout morphisme $u: \tauobj\rightarrow\sigmaobj$.

L'application $\iota: K(\Omega_n^{epi})\rightarrow B(\Omega_n^{epi})$
du~\S\ref{KoszulConstruction:KoszulCategory}
poss\`ede une extension naturelle aux constructions \`a coefficients $\kappa: K(S,\Omega_n^{epi},T)\rightarrow B(S,\Omega_n^{epi},T)$
qui est d\'efinie en posant
\begin{equation*}
\kappa(x\otimes\{u\}\otimes y) = x\otimes\iota\{u\}\otimes y,
\end{equation*}
pour tout $x\otimes\{u\}\otimes y\in K(S,\Omega_n^{epi},T)$.
On constate que~:

\begin{prop}\label{CoefficientConstructions:KoszulEmbedding}
Notre application commute aux diff\'erentielles
et d\'efinit un morphisme naturel de dg-modules
\begin{equation*}
K(S,\Omega_n^{epi},T)\xrightarrow{\kappa} B(S,\Omega_n^{epi},T),
\end{equation*}
pour tout couple de diagrammes $(S,T)$.
\end{prop}

\begin{proof}
Comme on sait que $\partial\iota\{u\} = 0$,
il suffit de v\'erifier que $\kappa$ commute aux termes de la diff\'erentielle faisant intervenir l'action sur les coefficients,
et une inspection de la d\'efinition de $\iota\{u\}$
permet d'\'etablir cette propri\'et\'e de fa\c con imm\'ediate.
\end{proof}

On forme les complexes $\Gamma(S,\Omega_n^{epi},\Omega_n^{epi}(\phiobj,-))$, $\Gamma = B,K$,
associ\'es \`a un diagramme de Yoneda $\Omega_n^{epi}(\phiobj,-)$,
pour un diagramme contravariant quelconque $S$.
On observe ais\'ement que les produits de composition
\begin{equation*}
\Omega_n^{epi}(\underline{\psi},\phiobj)\otimes\Omega_n^{epi}(\phiobj,-)\rightarrow\Omega_n^{epi}(\underline{\psi},-)
\end{equation*}
induisent des morphismes
\begin{equation*}
\Omega_n^{epi}(\underline{\psi},\phiobj)\otimes\Gamma(S,\Omega_n^{epi},\Omega_n^{epi}(\phiobj,-))
\rightarrow\Gamma(S,\Omega_n^{epi},\Omega_n^{epi}(\underline{\psi},-))
\end{equation*}
qui font de la collection $\Gamma(S,\Omega_n^{epi},\Omega_n^{epi}(\phiobj,-))$, $\phiobj\in\Ob\Omega_n^{epi}$,
un $\Omega_n^{epi}$-diagramme contravariant $\Gamma(S,\Omega_n^{epi},\Omega_n^{epi})$
naturellement associ\'e \`a $S$.

On a aussi un morphisme d'augmentation
\begin{equation*}
\epsilon: \Gamma(S,\Omega_n^{epi},\Omega_n^{epi})\rightarrow S,
\end{equation*}
donn\'e en degr\'e $0$ par les morphismes
\begin{equation*}
S(\tauobj)\otimes\Omega_n^{epi}(\phiobj,\tauobj)\rightarrow S(\phiobj)
\end{equation*}
induits par l'action de $\Omega_n^{epi}$ sur $S$.
On identifie $S$ \`a un complexe concentr\'e en degr\'e $0$
pour faire de l'augmentation
un morphisme de complexes de $\Omega_n^{epi}$-diagrammes.

On sait que la construction bar v\'erifie la propri\'et\'e suivante~:

\begin{lemm}\label{CoefficientConstructions:CoefficientBarAcyclicity}
Le morphisme d'augmentation
\begin{equation*}
\epsilon: B(S,\Omega_n^{epi},\Omega_n^{epi})\rightarrow S
\end{equation*}
d\'efinit une \'equivalence faible de $\Omega_n^{epi}$-diagrammes contravariants
pour tout~$S$.
\end{lemm}

\begin{proof}
Si on oublie la structure de diagramme,
alors on a une section $\eta: S\rightarrow B(S,\Omega_n^{epi},\Omega_n^{epi})$
au morphisme d'augmentation $\epsilon: B(S,\Omega_n^{epi},\Omega_n^{epi})\rightarrow S$
donn\'ee par le produit tensoriel de l'identit\'e de $S(\tauobj)$
avec le morphisme identique $1_{\tauobj}\in\Omega_n^{epi}(\tauobj,\tauobj)$.
L'application
\begin{multline*}
\nu(x\otimes\{\tauobj_0\xleftarrow{u_1}\cdots\xleftarrow{u_d}\tauobj_d\}\otimes\{v\})
= x\otimes\{\tauobj_0\xleftarrow{u_1}\cdots\xleftarrow{u_d}\tauobj_d\xleftarrow{v}\tauobj\}\otimes\{1_{\tauobj}\}
\end{multline*}
d\'efinit une homotopie contractante $\nu: B(S,\Omega_n^{epi},\Omega_n^{epi})\rightarrow B(S,\Omega_n^{epi},\Omega_n^{epi})$
telle que $\delta(\nu) = \id - \eta\epsilon$,
ce qui entraine
la conclusion du lemme.
\end{proof}

On consid\`ere maintenant un diagramme ponctuel $\pt_{\sigmaobj}$
tel que~:
\begin{equation*}
\pt_{\sigmaobj}(\thetaobj) = \begin{cases} \kk, & \text{si $\thetaobj = \sigmaobj$}, \\
0, & \text{sinon}. \end{cases}
\end{equation*}
On utilisera dans la suite que $\pt_{\sigmaobj}$
poss\`ede une structure de diagramme covariant et contravariant sur~$\Omega_n^{epi}$.

On montre~:

\begin{lemm}\label{CoefficientConstructions:CoefficientKoszulAcyclicity}
Le complexe $K(\pt_{\sigmaobj},\Omega_n^{epi},\Omega_n^{epi}(\phiobj,-))$
se r\'eduit \`a sa composante de degr\'e $0$ lorsque $\phiobj=\sigmaobj$
et est acyclique lorsque $\phiobj\not=\sigmaobj$,
de sorte que le morphisme d'augmentation
\begin{equation*}
\epsilon: K(\pt_{\sigmaobj},\Omega_n^{epi},\Omega_n^{epi})\rightarrow\pt_{\sigmaobj}
\end{equation*}
d\'efinit une \'equivalence faible de $\Omega_n^{epi}$-diagrammes contravariants,
pour tout diagramme ponctuel $\pt_{\sigmaobj}$.
\end{lemm}

\begin{proof}
On a une identification imm\'ediate $K(\pt_{\sigmaobj},\Omega_n^{epi},\Omega_n^{epi}(\sigmaobj,-)) = \pt_{\sigmaobj}(\sigmaobj)$
qui donne la premi\`ere assertion du lemme.
La d\'emonstration de l'acyclicit\'e du complexe $K(\pt_{\sigmaobj},\Omega_n^{epi},\Omega_n^{epi}(\phiobj,-))$ pour $\phiobj\not=\sigmaobj$,
point cl\'e de notre travail,
est re\-por\-t\'ee au~\S\ref{KoszulAcyclicity}.\end{proof}

On a de plus~:

\begin{obsv}\label{CoefficientConstructions:ProjectiveComplexes}
Les complexes $\Gamma(\pt_{\sigmaobj},\Omega_n^{epi},\Omega_n^{epi})$, $\Gamma = B,K$,
sont des complexes de $\Omega_n^{epi}$-diagrammes projectifs.
\end{obsv}

Le morphisme naturel de la Proposition~\S\ref{CoefficientConstructions:KoszulEmbedding},
appliqu\'e aux couples de diagrammes
\begin{equation*}
(S,T) = (\pt_{\sigmaobj},\Omega_n^{epi}(\phiobj,-)),
\end{equation*}
d\'efinit un morphisme de complexes de $\Omega_n^{epi}$-diagrammes contravariants
\begin{equation*}
K(\pt_{\sigmaobj},\Omega_n^{epi},\Omega_n^{epi})\xrightarrow{\kappa} B(\pt_{\sigmaobj},\Omega_n^{epi},\Omega_n^{epi}).
\end{equation*}
Ce morphisme commute clairement aux augmentations
ce qui entraine que notre morphisme $\kappa$,
s'inscrivant dans un diagramme de la forme
\begin{equation*}
\xymatrix{ K(\pt_{\sigmaobj},\Omega_n^{epi},\Omega_n^{epi})\ar@{.}[rr]^{\kappa}_{\sim}\ar[dr]_{\sim} &&
B(\pt_{\sigmaobj},\Omega_n^{epi},\Omega_n^{epi})\ar[dl]^{\sim} \\
& \pt_{\sigmaobj} & },
\end{equation*}
d\'efinit lui-m\^eme une \'equivalence faible de complexes de $\Omega_n^{epi}$-diagrammes.

\subsubsection{Produit tensoriel de diagrammes}\label{CoefficientConstructions:TensorProduct}
On consid\`ere maintenant l'op\'eration de produit tensoriel $S\otimes_{\Omega_n^{epi}} T$
sur la cat\'egorie $\Omega_n^{epi}$
qui, pour tout couple $(S,T)$ constitu\'e d'un $\Omega_n^{epi}$-diagramme contravariant $S$
et d'un $\Omega_n^{epi}$-diagramme covariant $T$,
est d\'efinie par la cofin~:
\begin{equation*}
S\otimes_{\Omega_n^{epi}} T = \int^{\tauobj\in\Ob\Omega_n^{epi}} S(\tauobj)\otimes T(\tauobj).
\end{equation*}
Pour un couple de $\Omega_n^{epi}$-diagrammes covariants $(S,T)$,
on a aussi un dg-hom sur~$\Omega_n^{epi}$
d\'efini par la fin~:
\begin{equation*}
\Hom_{\Omega_n^{epi}}(S,T) = \int_{\tauobj\in\Ob\Omega_n^{epi}} \Hom_{\dg\Mod}(S(\tauobj),T(\tauobj)).
\end{equation*}

Le produit tensoriel de diagrammes sur une cat\'egorie
v\'erifie des propri\'et\'es analogues au produit tensoriel des modules sur une alg\`ebre (voir~\cite[\S 17]{Schubert}).
On a notamment un isomorphisme naturel $\Omega_n^{epi}(-,\tauobj)\otimes_{\Omega_n^{epi}} T\simeq T(\tauobj)$
pour les foncteurs de Yoneda contravariants $S = \Omega_n^{epi}(-,\tauobj)$,
et sym\'etriquement $S\otimes_{\Omega_n^{epi}}\Omega_n^{epi}(\tauobj,-)\simeq S(\tauobj)$
pour les foncteurs de Yoneda covariants $T = \Omega_n^{epi}(\tauobj,-)$.
On obtient \`a partir de ces relations~:

\begin{obsv}\label{CoefficientConstructions:CoefficientTensorProduct}
On a un isomorphisme naturel
\begin{equation*}
\Gamma(S,\Omega_n^{epi},\Omega_n^{epi})\otimes_{\Omega_n^{epi}} T\xrightarrow{\simeq}\Gamma(S,\Omega_n^{epi},T),
\end{equation*}
pour tout couple de $\Omega_n^{epi}$-diagrammes $(S,T)$.
\end{obsv}

On applique ce r\'esultat aux $\Omega_n^{epi}$-diagrammes ponctuels $(S,T) = (\pt_{\sigmaobj},\pt_{\tauobj})$.
Observons que~:

\begin{obsv}\label{CoefficientConstructions:PunctualCoefficients}
On a des identit\'es
\begin{equation*}
K(\pt_{\sigmaobj},\Omega_n^{epi},\pt_{\tauobj}) = K(\Omega_n^{epi})(\tauobj,\sigmaobj)
\quad\text{et}
\quad
B(\pt_{\sigmaobj},\Omega_n^{epi},\pt_{\tauobj}) = B(\Omega_n^{epi})(\tauobj,\sigmaobj)
\end{equation*}
pour tout couple $(\tauobj,\sigmaobj)\in\Ob\Omega_n^{epi}\times\Ob\Omega_n^{epi}$.
\end{obsv}

On obtient alors~:

\begin{thm}\label{CoefficientConstructions:KoszulEquivalence}
Le morphisme
$\iota: K(\Omega_n^{epi})(\tauobj,\sigmaobj)
\rightarrow
B(\Omega_n^{epi})(\tauobj,\sigmaobj)$
d\'efini dans la Proposition~\ref{KoszulConstruction:KoszulEmbedding}
est une \'equivalence faible de dg-modules,
pour tout couple $(\tauobj,\sigmaobj)\in\Ob\Omega_n^{epi}\times\Ob\Omega_n^{epi}$.
\end{thm}

\begin{proof}
Le morphisme $\kappa: K(\pt_{\sigmaobj},\Omega_n^{epi},\Omega_n^{epi})\rightarrow B(\pt_{\sigmaobj},\Omega_n^{epi},\Omega_n^{epi})$
induit
par produit tensoriel $-\otimes_{\Omega_n^{epi}}\pt_{\tauobj}$
une \'equivalence faible de dg-modules
\begin{equation*}
\underbrace{K(\pt_{\sigmaobj},\Omega_n^{epi},\Omega_n^{epi})\otimes_{\Omega_n^{epi}}\pt_{\tauobj}}_{= K(\pt_{\sigmaobj},\Omega_n^{epi},\pt_{\tauobj})}
\xrightarrow{\sim}
\underbrace{B(\pt_{\sigmaobj},\Omega_n^{epi},\Omega_n^{epi})\otimes_{\Omega_n^{epi}}\pt_{\tauobj}}_{= B(\pt_{\sigmaobj},\Omega_n^{epi},\pt_{\tauobj})}
\end{equation*}
puisque $K(\pt_{\sigmaobj},\Omega_n^{epi},\Omega_n^{epi})$ et $B(\pt_{\sigmaobj},\Omega_n^{epi},\Omega_n^{epi})$
sont des complexes de $\Omega_n^{epi}$-diagrammes projectifs.
Ceci entraine, d'apr\`es l'observation pr\'ec\'edente, la conclusion du th\'eor\`eme.
\end{proof}

\subsection{Complexes \`a coefficients et foncteurs Tor}\label{DGCoefficientConstructions}
On s'int\'eresse maintenant aux applications des complexes \`a coefficients
pour le calcul des foncteurs $\Tor^{\Omega_n^{epi}}_*(S,T)$
et $\Ext_{\Omega_n^{epi}}^*(S,T)$.
On formalise nos r\'esultats dans le cadre diff\'erentiel gradu\'e
en utilisant l'ensemble des id\'ees expos\'ees
dans la section pr\'eliminaire~(\S\S\ref{Background:DGModules}-\ref{Background:ModelStructure}).
On commence par montrer comment les complexes \`a coefficients
d\'efinis dans la section pr\'ec\'edente s'\'etendent aux $\Omega_n^{epi}$-diagrammes
en dg-modules.

\subsubsection{Produits tensoriels tordus et constructions \`a coefficients}\label{CoefficientConstructions:DGDiagrams}
Nos complexes \`a coefficients
\begin{equation*}
\Gamma(S,\Omega_n^{epi},T),\quad\Gamma = B,K
\end{equation*}
s'identifient en fait \`a des dg-modules tordus
de la forme
\begin{equation*}
\Gamma(S,\Omega_n^{epi},T)
= \Bigl\{\bigoplus_{(\tauobj,\sigmaobj)} S(\sigmaobj)\otimes\Gamma(\Omega_n^{epi})(\tauobj,\sigmaobj)\otimes T(\tauobj),\partial\Bigr\}
\end{equation*}
pour un certain homomorphisme de torsion $\partial$.
On peut int\'egrer les termes de la diff\'erentielle de~$B(S,\Omega_n^{epi},T)$
qui ne font pas intervenir l'action de~$\Omega_n^{epi}$
sur les coefficients~$(S,T)$
dans la diff\'erentielle interne de la construction~$B(\Omega_n^{epi})$.
L'homo\-mor\-phis\-me de torsion de notre produit tensoriel
se r\'eduit donc aux termes
\begin{multline*}
\partial(x\otimes\{\tauobj_0\xleftarrow{u_1}\cdots\xleftarrow{u_d}\tauobj_d\}\otimes y)
= u_1^*(x)\otimes\{\tauobj_1\xleftarrow{u_2}\cdots\xleftarrow{u_d}\tauobj_d\}\otimes y\\
+ (-1)^d x\otimes\{\tauobj_0\xleftarrow{u_1}\cdots\xleftarrow{u_{d-1}}\tauobj_{d-1}\}\otimes u_d{}_*(y)
\end{multline*}
dans le cas~$\Gamma = B$,
et reprend l'ensemble des termes de la formule du~\S\ref{CoefficientConstructions:BasicDefinition}
dans le cas~$\Gamma = K$.

Il suffit de former le produit tensoriel $S(\sigmaobj)\otimes\Gamma(\Omega_n^{epi})(\tauobj,\sigmaobj)\otimes T(\tauobj)$
dans la cat\'egorie des dg-modules
pour \'etendre la d\'efinition des complexes $\Gamma(S,\Omega_n^{epi},T)$, $\Gamma = B,K$,
aux $\Omega_n^{epi}$-diagrammes en dg-modules.
Les objets $S(\sigmaobj)$ et $T(\tauobj)$
sont alors munis d'une diff\'erentielle interne qui est simplement ajout\'ee
aux homomorphismes de torsion du~\S\ref{CoefficientConstructions:BasicDefinition}.

On ne consid\`ere dans la suite que des diagrammes en dg-modules $(S,T)$
dont les composantes $S(\sigmaobj)$ et $T(\tauobj)$
sont des dg-modules cofibrants,
pour tout $(\sigmaobj,\tauobj)\in\Ob\Omega_n^{epi}\times\Ob\Omega_n^{epi}$.
On dit alors que les diagrammes $(S,T)$ sont dg-cofibrants.

Le morphisme $\kappa: K(S,\Omega_n^{epi},T)\rightarrow B(S,\Omega_n^{epi},T)$
consid\'er\'e au~\S\ref{CoefficientConstructions}
est simplement d\'efini par des produits tensoriels $\id_{S(\sigmaobj)}\otimes\iota\otimes\id_{T(\tauobj)}$
induits par le morphisme~$\iota: K(\Omega_n^{epi})\rightarrow B(\Omega_n^{epi})$
de la Proposition~\ref{KoszulConstruction:KoszulEmbedding}.
On a le r\'esultat suivant~:

\begin{lemm}\label{CoefficientConstructions:CoefficientKoszulEquivalence}
Le morphisme de complexes \`a coefficients
\begin{equation*}
\kappa: K(S,\Omega_n^{epi},T)\rightarrow B(S,\Omega_n^{epi},T)
\end{equation*}
est une \'equivalence faible de dg-modules
d\`es lors que $(S,T)$ sont des $\Omega_n^{epi}$-diagrammes en dg-modules dg-cofibrants.\qed
\end{lemm}

\begin{proof}
La filtration par le degr\'e sur chaque dg-graphe $\Gamma(\Omega_n^{epi})(\tauobj,\sigmaobj)$, pour $\Gamma = B,K$,
induit une filtration naturelle au niveau des produits tensoriels
$\Gamma(S,\Omega_n^{epi},T)
= \bigl\{\bigoplus_{(\tauobj,\sigmaobj)} S(\sigmaobj)\otimes\Gamma(\Omega_n^{epi})(\tauobj,\sigmaobj)\otimes T(\tauobj),\partial\bigr\}$.
On en d\'eduit l'existence d'une suite spectrale $E^r(\Gamma(S,\Omega_n^{epi},T))\Rightarrow H_*(\Gamma(S,\Omega_n^{epi},T))$
telle que $d^0$ est d\'etermin\'e par les diff\'erentielles internes des objets $S$, $T$
et $\Gamma(\Omega_n^{epi})$.
Le morphisme $\kappa$ pr\'eserve cette filtration
et induit un morphisme de suites spectrales $\kappa: E^r(K(S,\Omega_n^{epi},T))\rightarrow E^r(B(S,\Omega_n^{epi},T))$
qui est un isomorphisme au rang $E^1$.
La conclusion s'ensuit.
\end{proof}

La cat\'egorie des $\Omega_n^{epi}$-diagrammes contravariants en dg-modules,
et la cat\'egorie des $\Omega_n^{epi}$-diagrammes covariants de m\^eme,
h\'erite d'une structure de cat\'egorie mod\`ele projective naturelle (voir~\cite[\S 11.6]{Hirschhorn}).
Les foncteurs $\Tor^{\Omega_n^{epi}}_*(S,T)$ ont une d\'efinition naturelle
dans le cadre des dg-diagrammes
comme l'homologie de produits tensoriels $Q_S\otimes_{\Omega_n^{epi}} T$
appliqu\'es \`a un remplacement cofibrant $0\rightarrowtail Q_S\xrightarrow{\sim} S$
dans la cat\'egorie des $\Omega_n^{epi}$-diagrammes en dg-modules.
On a une d\'efinition \'equivalente faisant intervenir un remplacement cofibrant de $T$.
Les foncteurs $\Ext_{\Omega_n^{epi}}^*(S,T)$
sont d\'efinis de fa\ con analogues
comme l'homologie de dg-hom $\Hom_{\Omega_n^{epi}}(Q_S,T)$
sur un remplacement cofibrant de $S$.

Les observations des~\S\S\ref{CoefficientConstructions:KoszulEmbedding}-\ref{CoefficientConstructions:CoefficientBarAcyclicity}
et~\S\ref{CoefficientConstructions:CoefficientTensorProduct}
se g\'en\'eralisent aux diagrammes dg-cofibrants.
On montre aussi que le complexe bar $B(S,\Omega_n^{epi},\Omega_n^{epi})$
est cofibrant comme $\Omega_n^{epi}$-diagramme
d\`es lors que $S$ est dg-cofibrant
et d\'efinit donc en remplacement cofibrant particulier de~$S$
dans la cat\'egorie des $\Omega_n^{epi}$-diagrammes.
On en d\'eduit que l'homologie de $B(S,\Omega_n^{epi},T) = B(S,\Omega_n^{epi},\Omega_n^{epi})\otimes_{\Omega_n^{epi}} T$
d\'etermine le foncteur Tor diff\'erentiel gradu\'e $\Tor^{\Omega_n^{epi}}_*(S,T)$.
Le Lemme~\ref{CoefficientConstructions:CoefficientKoszulEquivalence}
entraine donc~:

\begin{mainsectionthm}\label{CoefficientConstructions:TorFunctor}
On a l'identit\'e
\begin{equation*}
\Tor^{\Omega_n^{epi}}_*(S,T) = H_*(K(S,\Omega_n^{epi},T)),
\end{equation*}
d\`es lors que $(S,T)$ sont des $\Omega_n^{epi}$-diagrammes en dg-modules dg-cofibrants.\qed
\end{mainsectionthm}

On obtient de m\^eme~:

\begin{mainsectionthm}\label{CoefficientConstructions:ExtFunctor}
On a l'identit\'e
\begin{equation*}
\Ext_{\Omega_n^{epi}}^*(S,T) = H_*(\Hom_{\Omega_n^{epi}}(K(\Omega_n^{epi},\Omega_n^{epi},S),T),
\end{equation*}
d\`es lors que $(S,T)$ sont des $\Omega_n^{epi}$-diagrammes en dg-modules dg-cofibrants.\qed
\end{mainsectionthm}

\subsubsection{Relations avec le r\'esultat principal de~\cite{LivernetRichter}}\label{CoefficientConstructions:TrunkCoefficientComplex}
Notons $\underline{i}_n$ l'objet de $\Omega_n^{epi}$
repr\'e\-sen\-t\'e par l'arbre-tronc
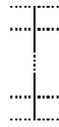
\begin{equation*}
\xymatrix@W=0mm@H=0mm@R=3mm@C=3mm@M=0mm{ \ar@{.}[r] & \ar@{-}[d] & \ar@{.}[l] \\
\ar@{.}[r] & \ar@{-}[d] & \ar@{.}[l] \\
& \ar@{.}[d] & \\
& \ar@{-}[d] & \\
\ar@{.}[r] & \ar@{-}[d] & \ar@{.}[l] \\
\ar@{.}[r] && \ar@{.}[l] }
\end{equation*}
On consid\`ere apr\`es~\cite{LivernetRichter}
le $\Omega_n^{epi}$-diagramme ponctuel $b_n = \pt_{\underline{i}_n}$
qui vaut par d\'efinition~:
\begin{equation*}
b_n(\sigmaobj) = \begin{cases} \kk, & \text{si $\sigmaobj = \underline{i}_n$} \\
0, & \text{sinon}. \end{cases}
\end{equation*}

L'arbre tronc $\underline{i}_n$ d\'efinit en fait l'objet final de $\Omega_n^{epi}$.
Par suite,
on obtient que le complexe de Koszul \`a coefficient dans $b_n$
poss\`ede un d\'eveloppement de la forme~:
\begin{equation*}
K(b_n,\Omega_n^{epi},T) = \Bigl\{\bigoplus_{\tauobj} T(\tauobj),\partial\Bigr\}.
\end{equation*}

Le complexe de Koszul $K(b_n,\Omega_n^{epi},T)$
s'identifie en fait au complexe $C_*(T)$
d\'efini dans~\cite[\S\S 3.5-9]{LivernetRichter}.
Le Th\'eor\`eme~\ref{CoefficientConstructions:TorFunctor}
donne donc une g\'en\'eralisation
de l'identit\'e
\begin{equation*}
\Tor^{\Omega_n^{epi}}_*(b_n,T) = H_*(C_*(T))
\end{equation*}
d\'emontr\'ee par les auteurs de~\cite{LivernetRichter}.

\subsection{Appendice~: la propri\'et\'e d'acyclicit\'e du complexe de Koszul}\label{KoszulAcyclicity}
Le but de cette section est d'\'etablir la propri\'et\'e suivante dont on avait report\'e la d\'e\-mons\-tra\-tion
au~\S\ref{CoefficientConstructions}~:

\begin{mainsectionlemm}[{affirmation du Lemme~\ref{CoefficientConstructions:CoefficientKoszulAcyclicity}}]\label{KoszulAcyclicity:Result}
L'homologie du dg-module
\begin{equation*}
L_n(\tauobj,\sigmaobj) = K(\pt_{\sigmaobj},\Omega_n^{epi},\Omega_n^{epi}(\tauobj,-))
\end{equation*}
associ\'e au diagramme ponctuel $\pt_{\sigmaobj}$ et au diagramme de Yoneda $\Omega_n^{epi}(\tauobj,-)$
est triviale lorsque~$\tauobj\not=\sigmaobj$.
\end{mainsectionlemm}

Ce lemme g\'en\'eralise le r\'esultat obtenu dans~\cite{LivernetRichter}
pour le diagramme ponctuel $b_{\underline{i}_n}$
associ\'e \`a un arbre tronc $\underline{i}_n$.
On suit le m\^eme plan de d\'emonstration.
On devra cependant introduire des notions ad hoc pour travailler avec diagrammes ponctuels~$b_{\tauobj}$
associ\'es \`a des arbres quelconques $\tauobj$
et des morphismes $u: \tauobj\rightarrow\sigmaobj$
susceptibles d'entrem\^eler les composantes de~$\tauobj$
de fa\c con compliqu\'ee.

On commence par reprendre les d\'efinitions pour analyser la structure du dg-module~$L_n(\tauobj,\sigmaobj)$.

\subsubsection{La d\'efinition du complexe $L_n(\tauobj,\sigmaobj)$}\label{KoszulAcyclicity:Complex}
Le dg-module
\begin{equation*}
L_n(\tauobj,\sigmaobj) = K(\pt_{\sigmaobj},\Omega_n^{epi},\Omega_n^{epi}(\tauobj,-))
\end{equation*}
est engendr\'e en degr\'e $d$ par les tenseurs $\{v\}\otimes\{w\}$
associ\'es aux couples de morphismes composables $\sigmaobj\xleftarrow{v}\thetaobj\xleftarrow{w}\tauobj$
tels que $\deg(v) = d$,
avec la diff\'erentielle
\begin{equation*}
\partial(\{v\}\otimes\{w\}) = \sum_{\substack{v = a b\\ \deg b = 1}} \pm\{a\}\otimes\{b w\}.
\end{equation*}
La somme s'\'etend sur l'ensemble des d\'ecompositions $v = a b$ telles que $\deg b = 1$.
Le facteur $\{w\}$ repr\'esente, dans la d\'efinition du~\S\ref{KoszulConstruction:KoszulCategory},
un \'el\'ement du diagramme de Yoneda $\Omega_n^{epi}(\tauobj,-)$.

On suppose dans la suite de cette \'etude que les \'el\'ements de $\In\tauobj$
sont en bijection avec un ensemble source $\eset$
muni d'un degr\'e interne
qui s'ajoute \`a la graduation du complexe $L_n(\tauobj,\sigmaobj)$
telle qu'on l'a d\'efinie.
On applique aussi la g\'en\'eralisation, mentionn\'ee au~\S\ref{KoszulConstruction:Signs},
de la d\'efinition du signe $\sgn(b)$
associ\'e \`a un morphisme de degr\'e $1$
dans la formule de la diff\'erentielle de $L_n(\tauobj,\sigmaobj)$.

Le complexe $L_n(\tauobj,\sigmaobj)$
poss\`ede un scindage naturel
\begin{equation*}
L_n(\tauobj,\sigmaobj)
= \bigoplus_{u: \tauobj\rightarrow\sigmaobj} L_n(\tauobj,\sigmaobj)_u
\end{equation*}
pour des sous-complexes $L_n(\tauobj,\sigmaobj)_u$, $u\in\Mor_{\Omega_n^{epi}}(\tauobj,\sigmaobj)$,
engendr\'es par les tenseurs $\{v\}\otimes\{w\}$
tels que~$v w = u$.

\subsubsection{Une suite spectrale}\label{KoszulAcyclicity:SpectralSequence}
L'id\'ee consiste \`a munir $L_n(\tauobj,\sigmaobj)_u$
de la filtration
\begin{equation*}
0 = F_0 L_n(\tauobj,\sigmaobj)_u\subset\cdots
\subset F_d L_n(\tauobj,\sigmaobj)_u\subset\cdots
\subset\colim_d F_d L_n(\tauobj,\sigmaobj)_u = L_n(\tauobj,\sigmaobj)_u
\end{equation*}
dont le terme $F_d L_n(\tauobj,\sigmaobj)_u$ est constitu\'e des facteurs $\kk\{v\}\otimes\kk\{w\}$
associ\'es \`a un objet milieu $\theta = \{\dset_0\xrightarrow{\theta_1}\cdots\xleftarrow{\theta_n}\dset_n\}$
tel que $d_1+\dots+d_{n-1}\leq d$.

La diff\'erentielle $d^0$ de la suite spectrale d\'efinie par cette filtration
se r\'eduit aux termes
\begin{equation*}
d^0(\{v\}\otimes\{w\}) = \sum_{v = a b_0} \pm\{a\}\otimes\{b_0 w\}
\end{equation*}
associ\'es aux d\'ecompositions $v = a b_0$
telles que, dans la repr\'esentation graphique des objets de $\Omega_n^{epi}$,
le morphisme $b_0: \thetaobj\rightarrow\underline{\rho}$
fixe le nombre de sommets de niveau $i>0$.
Dans la description du~\S\ref{KoszulConstruction:IndecomposableMorphisms},
ces morphismes $b_0$
sont donn\'es par la fusion de deux sommets cons\'ecutifs de niveau $0$,
repr\'esent\'ee sch\'ematiquement dans la Figure~\ref{Fig:TopMerging}.
L'\'etiquetage des feuilles nous permet, d'apr\`es les observations du~\S\ref{PrunedTrees:CommaPoset},
de repr\'esenter l'application induite par $b_0$
sur les ensembles sources des arbres.
\begin{figure}[h]
\begin{equation*}
\vcenter{\xymatrix@W=0mm@H=0mm@R=4.5mm@C=4.5mm@M=0mm{ \ar@{}[d]|{\displaystyle\cdots} &
\ar@{-}[drr]\ar@{.}[rr]^{\displaystyle k} & \bullet\ar@{-}[dr] && \bullet\ar@{-}[dl] & \ar@{.}[ll]_{\displaystyle l}\ar@{-}[dll] &
\ar@{}[d]|{\displaystyle\cdots} \\
\ar@{.}[rrr]\ar@{-}[drrr] &&& \ar@{-}[d] &&& \ar@{-}[dlll]\ar@{.}[lll] \\
\ar@{.}[rrr] &&& \ar@{}[d]|{\vdots} &&& \ar@{.}[lll] \\
&&&&&& }}
\quad\xrightarrow{\displaystyle b_0}
\quad\vcenter{\xymatrix@W=0mm@H=0mm@R=4.5mm@C=4.5mm@M=0mm{ \ar@{}[d]|{\displaystyle\cdots} &
\ar@{-}[drr]\ar@{.}[rrrr]^{\displaystyle k\,l} && \bullet\ar@{-}[d] && \ar@{-}[dll] &
\ar@{}[d]|{\displaystyle\cdots} \\
\ar@{.}[rrr]\ar@{-}[drrr] &&& \ar@{-}[d] &&& \ar@{-}[dlll]\ar@{.}[lll] \\
\ar@{.}[rrr] &&& \ar@{}[d]|{\vdots} &&& \ar@{.}[lll] \\
&&&&&& }}
\end{equation*}
\caption{}\label{Fig:TopMerging}\end{figure}

\medskip
On consid\`ere, comme dans~\cite[Proposition 4.7]{LivernetRichter},
le foncteur de troncature $\tr: \Omega_n^{epi}\rightarrow\Omega_{n-1}^{epi}$
d\'efini sur les objets par l'op\'eration
\begin{equation*}
\tr\{\tset_0\xrightarrow{\tau_1}\tset_1\xrightarrow{\tau_2}\cdots\xrightarrow{\tau_n}\tset_n\}
= \{\tset_1\xrightarrow{\tau_2}\cdots\xrightarrow{\tau_n}\tset_n\}.
\end{equation*}
On dit qu'un tenseur $\{\hat{v}\}\otimes\{\hat{w}\}$
couvre une d\'ecomposition tronqu\'ee $\tr u = v w$
si on a $\tr\hat{v} = v$ et $\tr\hat{w} = w$.

La diff\'erentielle $d^0$ pr\'eservant la structure aux niveaux $>0$,
on obtient~:

\begin{obsv}\label{KoszulAcyclicity:CoverSplitting}
Le module $L_n(\tauobj,\sigmaobj)_{\tr u = v w}$
engendr\'e par les tenseurs $\{\hat{v}\}\otimes\{\hat{w}\}$
tels que $u = \hat{v}\,\hat{w}$
couvre une d\'ecomposition tronqu\'ee donn\'ee $\tr u = v w$
d\'efinit un sous-complexe de $E^0 L_n(\tauobj,\sigmaobj)_u$,
de sorte que $E^0 L_n(\tauobj,\sigmaobj)_u$
admet un scindage~:
\begin{equation*}
E^0 L_n(\tauobj,\sigmaobj)_u = \bigoplus_{\tr u = v w}
(L_n(\tauobj,\sigmaobj)_{\tr u = v w},d^0).
\end{equation*}
\end{obsv}

On analyse la d\'efinition des d\'ecompositions $u = \hat{v}\,\hat{w}$
dans le prochain paragraphe en vue d'obtenir une identification des complexes $(L_n(\tauobj,\sigmaobj)_{\tr u = v w},d^0)$.

\subsubsection{Troncature et recouvrement}\label{KoszulAcyclicity:Truncation}
On se donne un morphisme
\begin{equation*}
\{\tset_0\xrightarrow{\tau_1}\cdots\xrightarrow{\tau_n}\tset_n\}
\xrightarrow{u}\{\sigma_0\xrightarrow{\sigma_1}\cdots\xrightarrow{\sigma_n}\sset_n\}
\end{equation*}
et une d\'ecomposition $\tr u = v w$
de $\tr u$
dans la cat\'egorie $\Omega_{n-1}^{epi}$.
Cette d\'ecomposition est d\'efinie par le squelette solide
d'un diagramme de la forme~:
\begin{equation*}
\xymatrix{ \tset_0\ar@/_3em/[dd]_{u}\ar@{.>}[d]_{\pi_0}\ar[r]^{\tau_1} &
\tset_1\ar[d]^w\ar[r]^{\tau_2} & \cdots\ar[r]^{\tau_n} & \tset_n\ar[d]^w \\
\dset_0\ar@{.>}[d]_{v}\ar@{.>}[r]^{\theta_1} &
\dset_1\ar[d]^v\ar[r]^{\theta_2} & \cdots\ar[r]^{\theta_n} & \dset_n\ar[d]^v \\
\sset_0\ar[r]_{\sigma_1} & \sset_1\ar[r]_{\sigma_2} & \cdots\ar[r]_{\sigma_n} & \sset_n }.
\end{equation*}
On consid\`ere l'ensemble des applications pointill\'ees qui peuvent compl\'eter un tel diagramme
d\'efinissant alors une d\'ecomposition $u = \hat{v}\,\hat{w}$
du morphisme $u$ dans $\Omega_n^{epi}$.
L'application $\pi_0: \tset_0\rightarrow\dset_0$
suffit \`a d\'eterminer toutes les autres
puisque les applications du diagramme sont suppos\'ees surjectives par d\'efinition de $\Omega_n^{epi}$.

On pourra s'aider de la Figure~\ref{Fig:CompositeTopMerging} pour suivre la construction de la factorisation $u = v\pi_0$.
\begin{figure}[t]
\centerline{\xymatrix@W=0mm@H=0mm@R=4.5mm@C=1.25mm@M=0mm{ 
\cdots\ar@{.}[r]
& \bullet_1\ar@{-}[drrr]\ar@{-}[r] & \bullet_1\ar@{-}[drr]\ar@{.}[r] &
\bullet_2\ar@{-}[dr]\ar@{-}[r] & \bullet_2\ar@{-}[d]\ar@{.}[r] &
\cdots\ar@{.}[r] & \ar@{-}[r] & \bullet_l\ar@{-}[dlll] &
\cdots\ar@{.}[l]\ar@{.}[r]
& \bullet_1\ar@{-}[drrr]\ar@{-}[r] & \bullet_1\ar@{-}[drr]\ar@{.}[r] &
\bullet_2\ar@{-}[dr]\ar@{-}[r] & \bullet_2\ar@{-}[d]\ar@{.}[r] &
\cdots\ar@{.}[r] & \bullet_l\ar@{-}[r] & \bullet_l\ar@{-}[dlll] &
\cdots\ar@{.}[l]
&& \ar@{.>}[rrrr]^{\pi_0} &&&& &&
\cdots\ar@{.}[r]
& \bullet_1\ar@{.}[drrr]\ar@{-}[r] & \bullet_1\ar@{.}[drr]\ar@{.}[r] &
\bullet_2\ar@{.}[dr]\ar@{-}[r] & \bullet_2\ar@{.}[d]\ar@{.}[r] &
\cdots\ar@{.}[r] & \bullet_l\ar@{-}[r] & \bullet_l\ar@{.}[dlll] &
\cdots\ar@{.}[l]
&& \ar@{.>}[rrrr]^{v} &&&& &&
\cdots\ar@{.}[r]
& \bullet_1\ar@{-}[drr]\ar@{.}[r] & \bullet_2\ar@{-}[dr]\ar@{.}[r] & \cdots\ar@{.}[rr] && \bullet_l\ar@{-}[dll] &
\cdots\ar@{.}[l]
\\
\cdots
&&&& \ar@{.}[llll]\bullet\ar@{.}[rrrr]\ar@{}[d]|(0.25){\displaystyle z_1}\ar@{-}[drrrr]!U &&&&
\cdots
&&&& \ar@{.}[llll]\bullet\ar@{.}[rrrr]\ar@{}[d]|(0.25){\displaystyle z_1}\ar@{-}[dllll]!U &&&&
\cdots
&& \ar[rrrr]^{w} &&&& &&
\cdots\ar@{.}[r]
&&&& \ar@{.}[llll]\bullet\ar@{.}[rrrr]\ar@{-}[d]_(0.25){\displaystyle y_1} &&&&
\cdots
&& \ar[rrrr]^{v} &&&& &&
\cdots\ar@{.}[r]
&&& \ar@{.}[lll]\bullet\ar@{-}[d]_(0.25){\displaystyle x_1}\ar@{.}[rrr] &&&
\cdots
\\
&&&& &&&&
\vdots &&&& &&&&
&& &&&& &&
&&&& \vdots &&&&
&& &&&& &&
&&& \vdots &&&
}}
\caption{Une construction de factorisation $u = v\pi_0$ donnant un recouvrement $u = \hat{v}\,\hat{w}$
de d\'ecomposition tronqu\'ee $tr u = v w$.
Les expressions $\bullet_k-\bullet_k$
repr\'esentent les intervalles de sommets fusionn\'es sur les diff\'erents \'el\'ements $\bullet_k = x_0\in\sset_0$
par l'application $u: \tset_0\rightarrow\sset_0$.}\label{Fig:CompositeTopMerging}\end{figure}
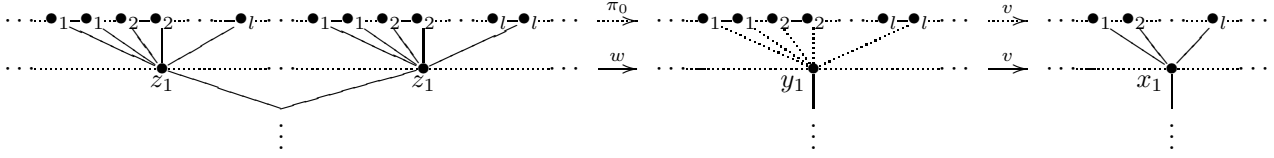

L'application $w\tau_1$ envoie les fibres $\tau_1^{-1}(z_1)\subset\tset_0$
des sommets $z_1\in w^{-1}(y_1)$
sur $y_1$, pour chaque $y_1\in\dset_1$.
La factorisation $w\tau_1 = \theta_1\pi_0$
revient \`a la donn\'ee d'une concat\'enation d'ensembles ordonn\'es $\dset_0 = \coprod_{y_1\in\dset_1}\dset_0(y_1)$,
qui repr\'esentent les fibres $\dset_0(y_1) = \theta_1^{-1}(y_1)$ de l'application $\theta_1$,
et d'applications surjectives
\begin{equation}\renewcommand{\theequation}{*}
\pi_0: \coprod_{z_1\in w^{-1}(y_1)}\tau_1^{-1}(z_1)\rightarrow\dset_0(y_1)
\end{equation}
qui repr\'esentent la restriction de $\pi_0$ \`a chaque groupe de composantes associ\'e \`a un \'el\'ement $y_1\in\dset_1$.
Ces applications doivent aussi pr\'eserver l'ordre sur chaque fibre $\tau_1^{-1}(z_1)$
par d\'efinition de la cat\'egorie $\Omega_n^{epi}$.
Dans la Figure~\ref{Fig:CompositeTopMerging},
on a repr\'esent\'e la fibre $\dset_0(y_1) = \theta_1^{-1}(y_1)$ d'un seul point $y_1\in\dset_1$
et les fibres $\tau_1^{-1}(z_1)$
des sommets $z_1\in\tset_1$ qui sont envoy\'es sur ce point $y_1$.

L'existence d'une factorisation $u = v\pi_0$ impose la relation suppl\'ementaire
$u(k)\leq u(l)\,\Rightarrow\,\pi_0(k)\leq\pi_0(l)$
pour chaque paire $\{k,l\}$
sur un m\^eme groupe de composantes $\coprod_{z_1\in w^{-1}(y_1)}\tau_1^{-1}(z_1)$
car les applications $v$ doivent \'egalement \^etre croissantes sur les fibres.
On a aussi $u(k)\not=u(l)\,\Rightarrow\,\pi_0(k)\not=\pi_0(l)$.
On note que $u^{-1}(x_0)\cap\tau_1^{-1}(z_1)$
d\'efinit un sous-intervalle de $\tau_1^{-1}(z_1)\subset\tset_0$, pour tout $x_0\in\sset_0$,
puisque l'application $u$ est aussi suppos\'ee croissante sur chaque fibre $\tau_1^{-1}(z_1)$, $z_1\in\tset_1$.
Dans la Figure~\ref{Fig:CompositeTopMerging},
on a repr\'esent\'e l'image inverse des \'el\'ements $\bullet_k = x_0\in\sigma_1^{-1}(x_1)$
dans la fibre du sommet $y_1$ et des sommets $z_1$
par des intervalles $\bullet_k - \bullet_k$
qui, d'apr\`es les observations de ce paragraphe, se doivent d'\^etre agenc\'es selon l'ordre
de l'ensemble image $\{\bullet_1,\dots,\bullet_l\} = \sigma_1^{-1}(x_1)\subset\sset_0$.
L'application $\pi_0$
est donc constitu\'ee d'un assemblage de surjections
\begin{equation*}
\bullet_k-\bullet_k\amalg\cdots\amalg\bullet_k-\bullet_k\xrightarrow{\pi_0}\bullet_k-\bullet_k
\end{equation*}
pr\'eservant l'ordre sur chaque intervalle $\bullet_k - \bullet_k$
du domaine.

In fine,
on conclut de notre analyse que l'application $\pi_0: \tset_0\rightarrow\dset_0$
s'ins\'erant dans une factorisation $u = \hat{v}\,\hat{w}$
est enti\`erement d\'etermin\'ee par la donn\'ee d'ensembles ordonn\'es $\dset_0(x_0,y_1)$, $x_0\in\sset_0$, $y_1\in\tset_1$,
qui repr\'esenteront les images inverses $v^{-1}(x_0)\cap\theta_1^{-1}(y_1)$ dans les fibres des sommets $y_1\in\tset_1$,
et d'applications surjectives
\begin{equation}\renewcommand{\theequation}{**}
\pi_0: \coprod_{z_1\in w^{-1}(y_1)} u^{-1}(x_0)\cap\tau_1^{-1}(z_1)\rightarrow\dset_0(x_0,y_1)
\end{equation}
pr\'eservant l'ordre sur chaque composante $u^{-1}(x_0)\cap\tau_1^{-1}(z_1)$.
L'ensemble ordonn\'e $\dset_0$ est alors d\'efini par la concat\'enation
$\dset_0 = \coprod_{y_1 = 1}^{d_1}\bigl\{\coprod_{x_0 = 1}^{s_0}\dset_0(x_0,y_1)\bigr\}$
dans l'ordre indiqu\'e par la somme.
L'application $\pi_0: \tset_0\rightarrow\dset_0$
est form\'ee de la somme des surjections
que l'on s'est donn\'ees~(**),
avec une permutation de battage provenant de l'ordonnancement de la somme dans la d\'efinition de $\dset_0$.
L'application $v: \dset_0\rightarrow\sset_0$
est d\'etermin\'ee par l'identit\'e $v^{-1}(x_0) = \coprod_{y_1\in\dset_1}\dset_0(x_0,y_1)$, pour chaque~$x_0\in\sset_0$,
et l'application $\theta_1: \dset_0\rightarrow\dset_1$
par l'identit\'e $\theta_1^{-1}(y_1) = \coprod_{x_0\in\sset_0}\dset_0(x_0,y_1)$, pour chaque~$y_1\in\dset_1$.
Dans la Figure~\ref{Fig:CompositeTopMerging},
le domaine $\dset_0(x_0,y_1)$ associ\'e au point $\bullet_k = x_0$
est repr\'esent\'e par l'intervalle $\bullet_k - \bullet_k$
au dessus de $y_1$,
les intervalles $\bullet_k - \bullet_k$ au dessus des diff\'erents $z_1$
repr\'esentant les images inverses $u^{-1}(x_0)\cap\tau_1^{-1}(z_1)$
dans les fibres des diff\'erents sommets $z_1\in\tset_1$.

\medskip
On a suppos\'e au d\'epart que le domaine $\tset_0 = \In\tauobj$
est en bijection avec un ensemble d'entr\'ees donn\'e $\eset$.
On va identifier chaque $e\in\eset$ au point correspondant de $\tset_0$
pour simplifier l'\'ecriture des relations qui vont suivre.
On forme, pour chaque couple $(x_0,y_1)\in\sset_0\times\dset_1$,
le produit tensoriel d'alg\`ebres associatives libres
\begin{equation*}
A_{x_0 y_1} = \bigotimes_{z_1\in w^{-1}(y_1)}\kk\langle X_e,e\in u^{-1}(x_0)\cap\tau_1^{-1}(z_1)\rangle
\end{equation*}
sur des variables $X_e$ de degr\'e $\deg(e)$, pour chaque $e\in\eset$.
On consid\`ere la construction bar $B(A_{x_0 y_1})$
de chacune de ces alg\`ebres gradu\'ees.
\begin{figure}[t]
\centerline{\xymatrix@W=0mm@H=0mm@R=6mm@C=5mm@M=0mm{ 
\ar@{-}[dr]\ar@{.}[rr]^<{e_1}^>{e_2} && \ar@{-}[dl]\ar@{.}[r] &
\ar@{-}[dr]\ar@{.}[rr]^<{e_3}^>{f_1} && \ar@{-}[dl]\ar@{.}[r] &
\ar@{-}[dr]\ar@{.}[r]^<{e_4}^>{e_5} & \ar@{-}[d]\ar@{.}[r]^>{f_2} & \ar@{-}[dl]
&& \ar@{.>}[r]^{\pi_0} &&&
\ar@{.}[dr]\ar@{.}[r]^<{e_1 e_3}^>{e_2} & \ar@{.}[d]\ar@{.}[r]^>{f_1} & \ar@{.}[dl]\ar@{.}[rr] &&
\ar@{.}[dr]\ar@{.}[r]^<{e_4}^>{e_5} & \ar@{.}[d]\ar@{.}[r]^>{f_2} & \ar@{.}[dl]
&& \ar@{.>}[r]^{v} &&&
\ar@{-}[drr]\ar@{}[d]|{x_0^1}\ar@{.}[rrrr]^<{e_1\dots e_5}^>{f_1 f_2} &&&& \ar@{-}[dll]\ar@{}[d]|{x_0^2} \\
\ar@{.}[r] & \ar@{-}[drrr]\ar@{.}[rrr] &&& \ar@{-}[d]\ar@{.}[rrr] &&& \ar@{-}[dlll]\ar@{.}[r] &
&& \ar[r]^{w} &&&
\ar@{.}[r] & \ar@{-}[drr]\ar@{}[d]|{y_1^1}\ar@{.}[rrrr] &&&& \ar@{-}[dll]\ar@{}[d]|{y_1^2}\ar@{.}[r] &
&& \ar[r]^{v} &&&
\ar@{.}[rr] && \ar@{-}[d]\ar@{.}[rr] && \\
\ar@{.}[rrrrrrrr] &&& &&& &&
&& \ar[r] &&&
\ar@{.}[rrrrrr] &&& &&&
&& \ar[r] &&&
\ar@{.}[rrrr] &&&& }}
\caption{Un exemple de d\'ecomposition $u = v w$.}\label{Fig:CompositeExample}\end{figure}
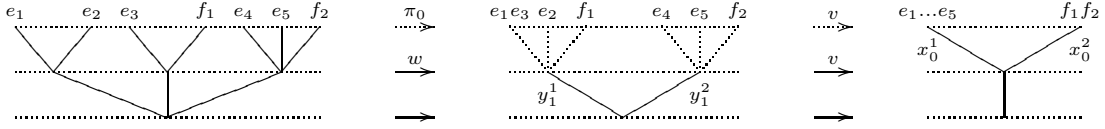
Pour la d\'ecomposition tronqu\'ee $\tr u = v w$ de la Figure~\ref{Fig:CompositeExample},
on obtient ainsi les alg\`ebres
\begin{multline*}
A_{x_0^1 y_1^1} = \kk\langle X_{e_1},X_{e_2}\rangle\otimes\kk\langle X_{e_3}\rangle,
\quad A_{x_0^2 y_1^1} = \kk\langle X_{f_1}\rangle,\\
\text{et}\quad A_{x_0^1 y_1^2} = \kk\langle X_{e_4},X_{e_5}\rangle,
\quad A_{x_0^2 y_1^2} = \kk\langle X_{f_2}\rangle.
\end{multline*}
On n'utilisera pas le produit tensoriel, que l'on r\'eserve pour la notation du complexe bar $B(A_{x_0 y_1})$,
mais un simple $\cdot$
dans l'expression des \'el\'ements de $A_{x_0 y_1}$.
On peut en fait identifier $A_{x_0 y_1}$
\`a l'alg\`ebre engendr\'ee par les variables $X_e$, $e\in u^{-1}(x_0)\cap\tau_1^{-1}(z_1)$, $z_1\in w^{-1}(y_1)$,
avec les relations de commutations $X_{e_1}\cdot X_{e_2} = X_{e_2}\cdot X_{e_1}$,
lorsque les variables $X_{e_1}$ et $X_{e_2}$
appartiennent \`a des groupements $u^{-1}(x_0)\cap\tau_1^{-1}(z_1)$
diff\'erents.

Le complexe $B(A_{x_0 y_1})$ poss\`ede un facteur direct $B(A_{x_0 y_1})^{sh}$
qui est engendr\'e par les tenseurs $\alpha_1\otimes\dots\otimes\alpha_d\in A_{x_0 y_1}^{\otimes d}$
tels que :
\begin{itemize}
\item
le mon\^ome $\alpha = \alpha_1\cdot\ldots\cdot\alpha_d$
obtenu par concat\'enation des facteurs $\alpha_i$
est de degr\'e $1$ en chaque variable $X_e$, pour tout $e\in u^{-1}(x_0)\cap\tau_1^{-1}(z_1)$, $z_1\in w^{-1}(y_1)$,
\item
les variables $X_e$ associ\'ees aux indices $e\in u^{-1}(x_0)\cap\tau_1^{-1}(z_1)$, pour un \'el\'ement $z_1\in w^{-1}(y_1)$ fix\'e,
apparaissent dans l'ordre des indices~$u^{-1}(x_0)\cap\tau_1^{-1}(z_1)\subset\tset_0$
dans le mon\^ome $\alpha$.
\end{itemize}
Dans notre exemple,
on obtient pour $B(A_{x_0^1 y_1^1})^{sh}$
un complexe de la forme~:
\begin{multline*}
\underbrace{\kk X_{e_1}\otimes X_{e_2}\otimes X_{e_3}\oplus\kk X_{e_1}\otimes X_{e_3}\otimes X_{e_2}\oplus\kk X_{e_3}\otimes X_{e_1}\otimes X_{e_2}}_{\deg=3}\\
\xrightarrow{\partial}
\quad\begin{aligned} & \kk X_{e_1} X_{e_2}\otimes X_{e_3}\oplus\kk X_{e_1} X_{e_3}\otimes X_{e_2}\oplus\kk X_{e_3} X_{e_1}\otimes X_{e_2} \\
\oplus &
\underbrace{\kk X_{e_1}\otimes X_{e_2} X_{e_3}\oplus\kk X_{e_1}\otimes X_{e_3} X_{e_2}\oplus\kk X_{e_3}\otimes X_{e_1} X_{e_2}}_{\deg=2} \\
\end{aligned}\\
\xrightarrow{\partial}
\quad\underbrace{\kk X_{e_1} X_{e_2} X_{e_3}\oplus\kk X_{e_1} X_{e_3} X_{e_2}\oplus\kk X_{e_3} X_{e_1} X_{e_2}}_{\deg=1}.
\end{multline*}

\medskip
On observe~:

\begin{obsv}
La donn\'ee de la surjection~(**) au~\S\ref{KoszulAcyclicity:Truncation}, pour un couple $(x_0,y_1)$ fix\'e,
revient \`a la donn\'ee d'un tenseur $\pm\otimes_{k\in\dset_0(x_0,y_1)}\alpha_k$
de $B(A_{x_0 y_1})^{sh}$,
les mon\^omes $\alpha_k\in A_{x_0 y_1}$, $k\in\dset_0(x_0,y_1)$,
associ\'es \`a une surjection $\pi_0$ \'etant d\'etermin\'es par les indices des images inverses $\pi_0^{-1}(k)\subset\eset$.

Cette correspondance fait implicitement appel \`a une permutation d'\'el\'ements gra\-du\-\'es,
associ\'es aux entr\'ees $e\in\eset$.
Le signe qui appara\^\i t dans l'expression de notre tenseur est produit par cette permutation.
\end{obsv}

Par exemple,
pour la surjection $\pi_0$ de la Figure~\ref{Fig:CompositeExample},
on obtient les tenseurs
\begin{multline*}
X_{e_1} X_{e_3}\otimes X_{e_2}\in B(A_{x_0^1 y_1^1})^{sh},
\quad X_{f_1}\in B(A_{x_0^2 y_1^1})^{sh},\\
\text{et}\quad X_{e_4}\otimes X_{e_5}\in B(A_{x_0^1 y_1^2})^{sh},
\quad X_{f_2}\in B(A_{x_0^2 y_1^2})^{sh}.
\end{multline*}

L'article~\cite{LivernetRichter}
utilise une repr\'esentation diff\'erente,
en termes de posets,
des surjections (**) du~\S\ref{KoszulAcyclicity:Truncation}.
L'interpr\'etation en termes de tenseurs dans le complexe bar nous permet de g\'erer les signes qui interviennent dans les diff\'erentielles
de fa\c con automatique.

\medskip
On a maintenant~:

\begin{obsv}\label{KoszulAcyclicity:TensorInterpretation}
La diff\'erentielle $d^0(\{\hat{v}\}\otimes\{\hat{w}\}) = \sum_{\hat{v} = a b_0} \sgn(b_0)\cdot\{a\}\otimes\{b_0\hat{w}\}$
du complexe $(L_n(\tauobj,\sigmaobj)_{\tr u = v w},d^0)$
s'identifie \`a la diff\'erentielle du produit tensoriel des complexes $B(A_{x_0 y_1})^{sh}$,
de sorte que l'on a une identit\'e~:
\begin{equation*}
(L_n(\tauobj,\sigmaobj)_{\tr u = v w},d^0) = \bigotimes_{(x_0,y_1)\in\sset_0\times\dset_1} B(A_{x_0 y_1})^{sh}.
\end{equation*}
\end{obsv}

On sait que la construction bar (r\'eduite) d'une alg\`ebre associative libre $A = \kk\langle Y_f,f\in\fset\rangle$
v\'erifie
\begin{equation*}
H_*(\overline{B}(A)) = \bigoplus_{f\in\fset} \kk\{Y_f\},
\end{equation*}
pour des classes $\{Y_f\}$ plac\'ees en degr\'e $\deg(f)+1$ (voir par exemple~\cite[\S 3.1]{Loday}).
On en d\'eduit, en appliquant la formule de K\"unneth (voir~\cite[\S X.7]{MacLane})
\begin{equation*}
H_*(B(A_{x_0 y_1})) = \bigotimes_{z_1\in w^{-1}(y_1)} H_*(B(\kk\langle X_e,e\in u^{-1}(x_0)\cap\tau_0^{-1}(z_1)\rangle))
\end{equation*}
que le complexe $B(A_{x_0 y_1})^{sh}$ est acyclique,
sauf lorsque $u^{-1}(x_0)\cap\tau_0^{-1}(z_1)$
est r\'eduit \`a un seul \'el\'ement pour tout $z_1$,
auquel cas on obtient $H_*(B(A_{x_0 y_1})) = \kk$.

Dans le prochain paragraphe, on introduit une nouvelle notion - la notion de morphisme injectif fibre \`a fibre au niveau~$0$ -
pour appliquer ce r\'esultat au complexe~$(L_n(\tauobj,\sigmaobj)_{\tr u = v w},d^0)$.

\subsubsection{Une construction de cycles au niveau $E^0$ de la suite spectrale}\label{KoszulAcyclicity:TopCycles}
On dit qu'un morphisme $u: \tauobj\rightarrow\sigmaobj$
est injectif fibre \`a fibre au niveau $0$
si l'application
\begin{equation*}
\tau_1^{-1}(z_1)\xrightarrow{u|_{\tau_1^{-1}(z_1)}}\sigma_1^{-1}(u(z_1))
\end{equation*}
induite par $u: \tset_0\rightarrow\sset_0$ est injective pour tout $z_1\in\tset_1$.
La Figure~\ref{Fig:FiberInjective} donne un exemple de tel morphisme.
Rappelons que l'\'etiquetage des feuilles permet, d'apr\`es les observations du~\S\ref{PrunedTrees:CommaPoset},
de d\'eterminer compl\`etement le morphisme $u$
par sa repr\'esentation graphique (en utilisant les \'etiquettes comme des variables muettes).
La condition d'injectivit\'e de ce paragraphe revient \`a assurer que $u^{-1}(x_0)\cap\tau_0^{-1}(z_1)$
est r\'eduit \`a un seul \'el\'ement pour tout $z_1\in\tset_1$.
On a donc, d'apr\`es les remarques suivant l'Observation~\ref{KoszulAcyclicity:TensorInterpretation}~:

\begin{lemm}
Pour toute d\'ecomposition tronqu\'ee $\tr u = v w$,
on a la relation $H_*(E^0 L_n(\tauobj,\sigmaobj)_{\tr u = v w},d^0) = \kk$ lorsque le morphisme $u$ est injectif fibre \`a fibre au niveau $0$
et  $H_*(E^0 L_n(\tauobj,\sigmaobj)_{\tr u = v w},d^0) = 0$ sinon.\qed
\end{lemm}

On applique la d\'efinition de l'isomorphisme de K\"unneth (voir~\cite[\S X.7]{MacLane} ou~\cite[\S 4.2]{Loday})
pour obtenir des repr\'e\-sen\-tants des classes
de
\begin{equation*}
E^1 L_n(\tauobj,\sigmaobj)_{u} = H_*(E^0 L_n(\tauobj,\sigmaobj)_u,d^0).
\end{equation*}
On consid\`ere simplement l'ensemble des d\'ecompositions $u = \hat{v}\,\hat{w}$ du~\S\ref{KoszulAcyclicity:Truncation}
telles que $\dset_1 = \tset_0$
et l'application $\pi_0$
est bijective,
pour chaque d\'ecomposition donn\'ee $\tr u = v w$
dans $\Omega_{n-1}^{epi}$.
La Figure~\ref{Fig:TruncationFillin} donne un l'ensemble de ces d\'ecompositions $u = \hat{v}\,\hat{w}$
pour une d\'ecomposition tronqu\'ee donn\'ee $\tr u = v w$
du morphisme de la Figure~\ref{Fig:FiberInjective}.
\begin{figure}[t]
\begin{equation*}
\vcenter{\xymatrix@W=0mm@H=0mm@R=4.5mm@C=4.5mm@M=0mm{ \ar@{-}[d]\ar@{.}[r]^<{\displaystyle e_1} &
\ar@{-}[dr]\ar@{.}[rr]^<{\displaystyle e_2} && \ar@{-}[dl]\ar@{.}[r]^<{\displaystyle f_1} &
\ar@{-}[dr]\ar@{.}[rr]^<{\displaystyle e_3} && \ar@{-}[dl]\ar@{.}[l]_<{\displaystyle f_2} \\
\ar@{-}[drr]\ar@{.}[rr] && \ar@{-}[d]\ar@{.}[rrr] &&& \ar@{-}[dlll]\ar@{.}[r] & \\
\ar@{.}[rrrrrr] &&&&&& }}
\quad\xrightarrow{\displaystyle u}
\quad\vcenter{\xymatrix@W=0mm@H=0mm@R=4.5mm@C=4.5mm@M=0mm{ \ar@{-}[drr]\ar@{.}[rr]^<{\displaystyle e_1\,e_2\,e_3} &&&&
\ar@{-}[dll]\ar@{.}[ll]_<{\displaystyle f_1\,f_2} \\
\ar@{.}[rr] && \ar@{-}[d] && \ar@{.}[ll] \\
\ar@{.}[rrrr] &&&& }}
\end{equation*}
\caption{}\label{Fig:FiberInjective}\end{figure}
\begin{figure}[t]
\begin{equation*}
\xymatrix@W=0mm@H=0mm@R=4.5mm@C=4.5mm@M=0mm{
&&&&&& 
&&
\save[].[rrrrrrdd]!C="g21" *+<24pt>[F-,]\frm{}\restore 
\ar@{-}[dr]\ar@{.}[r]^<{e_1} &
\ar@{-}[d]\ar@{.}[r]^<{e_2} & \ar@{-}[dl]\ar@{.}[rr]^<{f_1} &&
\ar@{-}[dr]\ar@{.}[r]^<{e_3} && \ar@{-}[dl]\ar@{.}[l]_<{f_2}
&&
&&&&&& 
\\
&&&&&& 
&&
& \ar@{-}[drr]\ar@{.}[rr]\ar@{.}[l] &&&& \ar@{-}[dll]\ar@{.}[ll]\ar@{.}[r] & 
&&
&&&&&& 
\\
&&&&&& 
&&
\ar@{.}[rrrrrr] &&&&&& 
&&
&&&&&& 
\\
&&&&&& 
&&
&&&&&& 
&&
&&&&&& 
\\
&&&&&& 
&&
&&&&&& 
&&
&&&&&& 
\\
\save[].[rrrrrrdd]!C="g1" *+<24pt>[F-,]\frm{}\restore 
\ar@{-}[d]\ar@{.}[r]^<{e_1} &
\ar@{-}[dr]\ar@{.}[rr]^<{e_2} && \ar@{-}[dl]\ar@{.}[r]^<{f_1} &
\ar@{-}[dr]\ar@{.}[rr]^<{e_3} && \ar@{-}[dl]\ar@{.}[l]_<{f_2}
&&
&&&&&& 
&&
\save[].[rrrrrrdd]!C="g3" *+<24pt>[F-,]\frm{}\restore 
& \ar@{-}[drr]\ar@{.}[rr]^<{e_1\,e_2\,e_3} &&&&
\ar@{-}[dll]\ar@{.}[ll]_<{f_1\,f_2} &
\\
\ar@{-}[drr]\ar@{.}[rr] && \ar@{-}[d]\ar@{.}[rrr] &&& \ar@{-}[dlll]\ar@{.}[r] & 
&&
&&&&&& 
&&
& \ar@{.}[rr] && \ar@{-}[d] && \ar@{.}[ll] & 
\\
\ar@{.}[rrrrrr] &&&&&& 
&&
&&&&&& 
&&
& \ar@{.}[rrrr] &&&& & 
\\
&&&&&& 
&&
&&&&&& 
&&
&&&&&& 
\\
&&&&&& 
&&
&&&&&& 
&&
&&&&&& 
\\
&&&&&& 
&&
\save[].[rrrrrrdd]!C="g22" *+<24pt>[F-,]\frm{}\restore 
\ar@{-}[dr]\ar@{.}[r]^<{e_2} &
\ar@{-}[d]\ar@{.}[r]^<{e_1} & \ar@{-}[dl]\ar@{.}[rr]^<{f_1} &&
\ar@{-}[dr]\ar@{.}[r]^<{e_3} && \ar@{-}[dl]\ar@{.}[l]_<{f_2}
&&
&&&&&& 
\\
&&&&&& 
&&
& \ar@{-}[drr]\ar@{.}[rr]\ar@{.}[l] &&&& \ar@{-}[dll]\ar@{.}[ll]\ar@{.}[r] & 
&&
&&&&&& 
\\
&&&&&& 
&&
\ar@{.}[rrrrrr] &&&&&& 
&&
&&&&&& 
\ar "g1"!UC+<16pt,16pt>;"g21"!CL-<16pt,16pt>
\ar "g1"!DC+<16pt,-16pt>;"g22"!CL-<16pt,-16pt>
\ar "g21"!CR+<16pt,-16pt>;"g3"!UC+<-16pt,16pt>
\ar "g22"!CR+<16pt,16pt>;"g3"!DC+<-16pt,-16pt>
}
\end{equation*}
\caption{}\label{Fig:TruncationFillin}\end{figure}

\medskip
On obtient alors~:

\begin{lemm}\label{KoszulAcyclicity:TopCycleDefinition}
Lorsque le morphisme $u$ est injectif fibre \`a fibre au niveau $0$,
la somme $\Z(\{v\}\otimes\{w\}) = \sum_{u = \hat{v}\,\hat{w}} \pm\{\hat{v}\}\otimes\{\hat{w}\}$
sur l'ensemble des d\'ecompositions couvrant $\tr u = v w$
telles que $\dset_1 = \tset_0$ et $\pi_0$ est bijective
d\'efinit un cycle dans $L_n(\tauobj,\sigmaobj)_u$.
L'application $\Z: \{v\}\otimes\{w\}\mapsto\Z(\{v\}\otimes\{w\})$
d\'efinit alors un isomorphisme de modules gradu\'es~:
\begin{equation*}
\bigoplus_{\tr u = v w}\kk\{v\}\otimes\{w\}\xrightarrow{\simeq} H_*(E^0 L_n(\tauobj,\sigmaobj)_u,d^0) = E^1 L_n(\tauobj,\sigmaobj)_u.\qed
\end{equation*}
\end{lemm}

Le signe $\pm$ dans la d\'efinition de $\Z(\{v\}\otimes\{w\})$
est d\'etermin\'e par la permutation de battage $\pi_0$
qui intervient dans la construction de $\hat{w}$.

\medskip
La donn\'ee d'un repr\'esentant explicite des classes des modules d'homologies $E^1 = H_*(E^0 L_n(\tauobj,\sigmaobj)_u,d^0)$
permet d'identifier la diff\'erentielle $d^1$
de la suite spectrale.
En fait, on obtient ais\'ement~:

\begin{lemm}\label{KoszulAcyclicity:Homology}
On a la relation
\begin{equation*}
d^1\Z(\{v\}\otimes\{w\}) = \sum_{\substack{v = a b\\ \deg b = 1}} \pm\Z(\{a\}\otimes\{b w\})
\end{equation*}
de sorte que l'application $\Z: \{v\}\otimes\{w\}\mapsto\Z(\{v\}\otimes\{w\})$
induit un isomorphisme de dg-modules
\begin{equation*}
L_{n-1}(\tr\sigmaobj,\tr\tauobj)_{\tr u}\xrightarrow{\simeq}(E^1 L_n(\tauobj,\sigmaobj)_u,d^1)
\end{equation*}
lorsque le morphisme $u$ est injectif fibre \`a fibre au niveau $0$.\qed
\end{lemm}

Les entr\'ees de $\tr\tauobj = \{\tset_1\xrightarrow{\tau_2}\dots\xrightarrow{\tau_n}\tset_n\}$
sont en bijection avec l'ensemble des fibres $\eset' = \{\tau_0^{-1}(z_1),\ z_1\in\tset_1\}$
de la premi\`ere surjection de $\tauobj = \{\tset_0\xrightarrow{\tau_1}\dots\xrightarrow{\tau_n}\tset_n\}$.
Pour que les signes de $L_{n-1}(\tr\sigmaobj,\tr\tauobj)_{\tr u}$
correspondent aux signes de la diff\'erentielle $d^1$,
on associe \`a ces entr\'ees $e' = \tau_0^{-1}(z_1)$
le degr\'e
\begin{equation*}
\deg(e') = \#\tau_0^{-1}(z_1) + \sum_{k\in\tau_0^{-1}(z_1)} \deg(k).
\end{equation*}
Dans cette expression,
le nombre $\#\tau_0^{-1}(z_1)$, qui d\'esigne le cardinal de l'ensemble $\tau_0^{-1}(z_1)$,
correspond, dans la repr\'esentation de~\S\ref{KoszulConstruction:Signs},
au nombre de suspensions associ\'ees aux sommets de $\tau_0^{-1}(z_1)$.

\medskip
En r\'ecapitulant les r\'esultats obtenus, on obtient~:

\begin{lemm}
On a $E^1 L_n(\tauobj,\sigmaobj)_u = L_{n-1}(\tr\sigmaobj,\tr\tauobj)_{\tr u}$
si le morphisme $u$ est injectif fibre \`a fibre au niveau $0$
et $E^1 L_n(\tauobj,\sigmaobj)_u = 0$
sinon.\qed
\end{lemm}

Une r\'ecurrence imm\'ediate entraine donc que l'homologie du complexe $L_n(\tauobj,\sigmaobj)_u$
est triviale,
sauf lorsque le morphisme $u$ est injectif fibre \`a fibre \`a tout niveau,
ce qui suppose $u = \id$
et $\tauobj=\sigmaobj$.
La conclusion du Lemme~\ref{KoszulAcyclicity:Result} s'ensuit et ceci boucle la d\'emonstration
des r\'esultats de~\S\S\ref{CoefficientConstructions}-\ref{DGCoefficientConstructions}.\qed

\section*{Partie 2. La r\'esolution de Koszul des la cat\'egorie des arbres \'elagu\'es}

On utilise le r\'esultat du Th\'eor\`eme~\ref{CoefficientConstructions:KoszulEquivalence}
pour d\'efinir un mod\`ele minimal de $\Omega_n^{epi}$
dans le cadre diff\'erentiel gradu\'e.
On r\'evise les applications de constructions de l'alg\`ebre diff\'erentielle gradu\'ee
aux cat\'egories au~\S\ref{DGConstructions}
avant de d\'efinir ce mod\`ele minimal au~\S\ref{PrunedTreeCobar}.

\subsection{Interm\`ede~: constructions sur les cat\'egories enrichies en dg-modules}\label{DGConstructions}

On consid\`ere pour nos constructions une cat\'egorie $\dg\Cat_{\X}$
form\'ee des petites cat\'egories enrichies en dg-modules (on parlera de dg-cat\'egories)
avec un ensemble d'objets fix\'e $\X$
(dans les applications de cet article, on prendra $\X = \Ob\Omega_n^{epi}$,
l'ensemble des arbres \'elagu\'es \`a niveaux).
Les morphismes de $\dg\Cat_{\X}$
sont les foncteurs de dg-cat\'egories
qui sont l'identit\'e sur l'ensemble des objets.

Le but de cette section est de passer en revue les applications des constructions bar et cobar de l'alg\`ebre diff\'erentielle gradu\'ee
aux cat\'egories de $\dg\Cat_{\X}$.
On passe rapidement sur les d\'emonstrations qui, pour la plupart, sont des g\'en\'eralisations
formelles des arguments d\'evelopp\'es dans le cadre des alg\`ebres.
On renvoie le lecteur \`a l'article~\cite{HusemollerMooreStasheff}
pour un expos\'e d\'etaill\'e de ces arguments.
On s'en rapportera \'egalement \`a la th\`ese \cite{Hasegawa} (voir plus particuli\`erement le chapitre~5 de cette th\`ese)
pour une pr\'esentation du cadre cat\'egorique permettant la g\'en\'eralisation
des constructions de~\cite{HusemollerMooreStasheff}.
Citons \'egalement l'article~\cite{Keller} pour un survol des applications des dg-cat\'egories
et une ample bibliographie sur le sujet.

On commence par expliciter la structure interne des cat\'egories et des foncteurs de~$\dg\Cat_{\X}$.

\subsubsection{La structure des dg-cat\'egories avec ensemble d'objets fix\'e}\label{DGConstructions:DGCategories}
La structure des dg-cat\'egories $\Theta\in\dg\Cat_{\X}$
est enti\`erement d\'etermin\'ee par la donn\'ee de dg-hom $\Theta(\bobj,\aobj)\in\dg\Mod$,
associ\'es aux couples $(\aobj,\bobj)\in\X\times\X$,
avec des morphismes de composition
\begin{equation*}
\Theta(\xobj,\aobj)\otimes\Theta(\bobj,\xobj)\xrightarrow{\mu}\Theta(\bobj,\aobj)
\end{equation*}
et des morphismes d'identit\'e
\begin{equation*}
\kk\xrightarrow{\eta}\Theta(\xobj,\xobj)
\end{equation*}
v\'erifiant une g\'en\'eralisation naturelle
des axiomes d'associativit\'e et d'unit\'e de la composition des morphismes dans les cat\'egories.
La donn\'ee des morphismes d'identit\'e \'equivaut \`a la donn\'ee d'\'el\'ements unit\'es $1_{\xobj}\in\Theta(\xobj,\xobj)$
tels que $\delta(1_{\xobj}) = 0$
pour tout $\xobj\in\X$.
On utilisera \'egalement les notations multiplicatives usuelles $\alpha\cdot\beta$
pour le produit de composition d'une dg-cat\'egorie.

Un morphisme $f: \Phi\rightarrow\Psi$
de la cat\'egorie $\dg\Cat_{\X}$
est d\'efini par la donn\'ee de morphismes de dg-modules
$f: \Phi(\bobj,\aobj)\rightarrow\Psi(\bobj,\aobj)$
pr\'eservant les morphismes d'identit\'e et les structures de composition
des cat\'egories.

On supposera pour des raisons techniques que l'ensemble $\X$
est munie d'une graduation $\deg: \X\rightarrow\NN$ \`a l'instar de l'ensemble des arbres \'elagu\'es.
On consid\'erera alors la sous-cat\'egorie pleine $\dg\Cat_{\X}^+\subset\dg\Cat_{\X}$
engendr\'ee par les cat\'egories $\Theta\in\dg\Cat_{\X}$
telles que $\Theta(\xobj,\xobj) = \kk 1_{\xobj}$, pour tout $\xobj\in\X$,
et $\Theta(\bobj,\aobj) = 0$
pour tout couple $\bobj\not=\aobj$
tel que $\deg(\bobj)-\deg(\aobj)\leq 0$.
On dira que $\dg\Cat_{\X}^+$
est la sous-cat\'egorie des dg-cat\'egories connexes de $\dg\Cat_{\X}$.

\subsubsection{Remarques~: une \'equivalence entre dg-cat\'egories et dg-alg\`ebres}\label{DGConstructions:DGAlgebraEquivalence}
Si on suppose que l'ensemble $\X$
est fini,
alors on a une \'equi\-va\-lence de cat\'egories entre les dg-cat\'egories $\Theta\in\dg\Cat_{\X}$
et les dg-alg\`ebres sur l'anneau commutatif $\kk_{\X} = \kk[1_{\xobj}, \xobj\in\X]$
engendr\'e comme $\kk$-module par des idempotents orthogonaux $1_{\xobj}$,~$\xobj\in\X$,
de sorte que~:
\begin{equation*}
1_{\bobj}\cdot 1_{\aobj} = \begin{cases} 1_{\aobj}, & \text{si $\aobj = \bobj$}, \\
0, & \text{sinon}. \end{cases}
\end{equation*}

On consid\'ere d'abord la cat\'egorie $\dg\Gr_{\X}$ des dg-graphes sur~$\X$
dont les objets sont les collections de dg-modules $\Gamma(\bobj,\aobj)$
index\'ees par les couples $(\bobj,\aobj)\in\X\times\X$.
On a, lorsque $\X$ est fini, une \'equivalence de cat\'egories de $\dg\Gr_{\X}$
dans la cat\'egorie des $\kk_{\X}$-bimodules
qui \`a un dg-graphe $\Gamma$ associe le dg-module
\begin{equation*}
\Gamma_{\X} = \bigoplus_{(\bobj,\aobj)\in\X\times\X} \Gamma(\bobj,\aobj),
\end{equation*}
muni de la structure de $\kk_{\X}$-bimodule telle que
\begin{equation*}
\Gamma(\bobj,\aobj) = 1_{\aobj}\cdot\Gamma_{\X}\cdot 1_{\bobj}.
\end{equation*}
Le produit tensoriel de $\kk_{\X}$-bimodules au dessus de $\kk_{\X}$
v\'erifie $N\otimes_{\kk_{\X}} M = \bigoplus_{\xobj\in\X} (N\cdot 1_{\xobj})\otimes(1_{\xobj}\cdot M)$
de sorte que pour les $\kk_{\X}$-bimodules associ\'es \`a des dg-graphes $\Gamma,\Delta\in\dg\Gr_{\X}$,
on a la relation
\begin{equation}
1_{\aobj}\cdot(\Gamma_{\X}\otimes_{\kk_{\X}}\Delta_{\X})\cdot 1_{\bobj}
= \bigoplus_{\xobj\in\X}\Gamma(\xobj,\aobj)\otimes\Delta(\bobj,\xobj).
\end{equation}

La dg-alg\`ebre associ\'ee \`a une cat\'egorie $\Theta$
est d\'efinie par le $\kk_{\X}$-bimodule $\Theta_{\X}$
associ\'e au dg-graphe sous-jacent \`a $\Theta$,
muni du morphisme unit\'e $\eta: \kk_{\X}\rightarrow\Theta_{\X}$ donn\'e par la somme des morphismes identit\'e de $\Theta$,
et du morphisme produit $\mu: \Theta_{\X}\otimes_{\kk_{\X}}\Theta_{\X}\rightarrow\Theta_{\X}$
induit composante par composante par les morphismes de composition de $\Theta$.

On utilisera cette correspondance de fa\c con heuristique pour transcrire des constructions
de l'alg\`ebre diff\'e\-ren\-tielle gradu\'ee
dans le cadre des cat\'egories, sans supposer n\'ecessairement $\X$ fini.
Le principe g\'en\'eral consiste simplement \`a remplacer le produit tensoriel $\otimes_{\kk_{\X}}$
par son d\'eveloppement~(*) pour former l'analogue cat\'egorique des constructions usuelles.
On utilisera dans la suite la notation $\otimes_{\X}$
pour d\'esigner le produit tensoriel~(*)
de dg-graphes sur~$\X$.

\subsubsection{Sur les dg-cat\'egories libres}\label{DGConstructions:FreeCategories}
Le foncteur d'oubli $U: \dg\Cat_{\X}\rightarrow\dg\Gr_{\X}$ qui applique une dg-cat\'egorie $\Theta$ sur son dg-graphe sous-jacent
poss\`ede ainsi un adjoint \`a gauche, le foncteur objet libre $\Free: \dg\Gr_{\X}\rightarrow\dg\Cat_{\X}$,
qui applique un dg-graphe $\Gamma$
sur la dg-cat\'egorie telle que~:
\begin{equation*}
\Free(\Gamma)(\bobj,\aobj)
= \bigoplus_{\substack{(\bobj,\xobj_{m-1},\dots,\xobj_{1},\aobj)\\
m\geq 0}}
\Gamma(\xobj_{1},\aobj)
\otimes\Gamma(\xobj_{2},\xobj_{1})
\otimes\dots
\otimes\Gamma(\bobj,\xobj_{m-1}).
\end{equation*}
La sommation s'\'etend sur l'ensemble des collections
\begin{equation*}
(\xobj_m,\dots,\xobj_0)\in\X\times\dots\times\X,\quad m\geq 0,
\end{equation*}
telles que $\xobj_m = \bobj$ et $\xobj_0 = \aobj$.
Lorsque $\bobj\not=\aobj$,
ces conditions entrainent $m>0$.
Lorsque $\bobj=\aobj$,
on peut avoir $m=0$
avec un produit tensoriel associ\'e d'ordre nul
retournant donc le dg-module unit\'e $\kk\in\dg\Mod$.
Le morphisme de composition de $\Free(\Gamma)$ est donn\'e par la concat\'enation des tenseurs.
Le morphisme d'identit\'e de $\Free(\Gamma)$ est donn\'e par l'inclusion des tenseurs d'ordre $m=0$
dans le d\'eveloppement des hom-objets $\Free(\Gamma)(\xobj,\xobj)$.

On a aussi un morphisme de dg-graphes $\eta: \Gamma\rightarrow\Free(\Gamma)$
qui identifie $\Gamma(\bobj,\aobj)$
au facteur de $\Free(\Gamma)(\bobj,\aobj)$
constitu\'e des tenseurs d'ordre $m=1$.
On peut ainsi identifier $\Gamma$ \`a un sous-objet de $\Free(\Gamma)$.

\subsubsection{Homomorphismes de torsion de dg-cat\'egories}\label{DGConstructions:TwistedDerivations}
On se donne maintenant une dg-cat\'egorie $\Theta\in\dg\Cat_{\X}$.
On dit qu'une collection d'homo\-mor\-phismes
$\partial: \Theta(\bobj,\aobj)\rightarrow\Theta(\bobj,\aobj)$
d\'efinit une d\'erivation de $\Theta$
si on a la relation
\begin{equation}
\partial(\alpha\cdot\beta) = (\partial\alpha)\cdot\beta + \pm\alpha\cdot(\partial\beta),
\end{equation}
pour tout couple d'\'el\'ements composables
$\alpha\in\Theta(\xobj,\aobj)$, $\beta\in\Theta(\bobj,\xobj)$,
le signe $\pm$
provenant de la commutation de l'homomorphisme $\partial$
avec l'\'el\'ement $\beta$.
On note que la relation~(*) entraine automatiquement que les \'el\'ements unit\'es $1_{\xobj}\in\Theta(\xobj,\xobj)$
sont annul\'es par~$\partial$.

On dit qu'une d\'erivation constitu\'ee d'homomorphismes de torsion
$\partial: \Theta(\bobj,\aobj)\rightarrow\Theta(\bobj,\aobj)$
est une d\'erivation de torsion de $\Theta$.
La relation de d\'erivation~(*) entraine que les morphismes de compositions et les morphismes d'identit\'es de $\Theta$
induisent des morphismes de dg-modules
sur les dg-modules tordus $(\Theta(\bobj,\aobj),\partial)$.
Par cons\'equent,
la collection de dg-modules tordus associ\'ee \`a une d\'erivation de torsion $(\Theta(\bobj,\aobj),\partial)$,
que l'on d\'esignera par la donn\'ee du couple $(\Theta,\partial)$,
h\'erite d'une structure de dg-cat\'egorie.

\subsubsection{Sur les dg-cat\'egories quasi-libres}\label{DGConstructions:QFreeCategories}
On a d\'efini la dg-cat\'egorie libre $\Free(\Gamma)$
associ\'ee \`a un dg-graphe $\Gamma$ au~\S\ref{DGConstructions:FreeCategories}.
Une dg-cat\'egorie quasi-libre est une dg-cat\'egorie tordue $(\Free(\Gamma),\partial)$
associ\'ee \`a une dg-cat\'egorie libre $\Free(\Gamma)$.

La dg-cat\'egorie libre $\Free(\Gamma)$
est,
d'apr\`es la d\'efinition du~\S\ref{DGConstructions:FreeCategories},
engendr\'ee comme dg-module par des tenseurs
\begin{equation}
\alpha_1\otimes\dots\otimes\alpha_m\in\Gamma(\xobj_1,\xobj_0)\otimes\cdots
\otimes\Gamma(\xobj_m,\xobj_{m-1}).
\end{equation}
Le dg-graphe $\Gamma$
s'identifie au facteur direct de $\Free(\Gamma)$
engendr\'e par les tenseurs d'ordre~$1$.
Le morphisme d'inclusion $\Gamma\subset\Free(\Gamma)$
s'identifie au morphisme universel de la dg-cat\'egorie libre.
Les tenseurs~(*) repr\'esentent en fait la composition des \'el\'ements $\alpha_i$
dans $\Free(\Gamma)$.
Les tenseurs~(*) d'ordre $m>1$ engendrent le dg-graphe $\Dec\Free(\Gamma)\subset\Free(\Gamma)$
des \'el\'ements d\'ecomposables de $\Free(\Gamma)$.

Une d\'erivation de torsion sur une dg-cat\'egorie libre $\partial: \Free(\Gamma)\rightarrow\Free(\Gamma)$
est donc d\'etermin\'ee par sa restriction au dg-graphe $\Gamma\subset\Free(\Gamma)$
puisque la relation de d\'erivation du~\S\ref{DGConstructions:TwistedDerivations}
entraine que l'on a l'identit\'e
\begin{equation*}
\partial(\alpha_1\otimes\cdots\otimes\alpha_m)
= \sum_{i=1}^{m}\pm\alpha_1\cdot\ldots\cdot\partial(\alpha_i)\cdot\ldots\cdot\alpha_m
\end{equation*}
pour une telle composition d'\'el\'ements.

Le morphisme de dg-cat\'egories $\phi_f: \Free(\Gamma)\rightarrow\Theta$
associ\'e \`a un morphisme de dg-graphes $f: \Gamma\rightarrow\Theta$
par la relation d'adjonction de l'objet libre
se d\'etermine par la formule
\begin{equation*}
\phi_f(\alpha_1\otimes\dots\otimes\alpha_m) = f(\alpha_1)\cdot\ldots\cdot f(\alpha_m)
\end{equation*}
pour un \'el\'ement compos\'e de $\Free(\Gamma)$.
Le morphisme $f: \Gamma\rightarrow\Theta$
repr\'esente en fait la restriction de $\phi_f$
au dg-graphe $\Gamma\subset\Free(\Gamma)$.

On \'etend la d\'efinition de $\phi_f$ aux homomorphismes $f$ de degr\'e $0$
pour construire les morphismes sur une dg-cat\'egorie quasi-libre.
On constate ais\'ement que l'homo\-mor\-phisme $\phi_f: (\Free(\Gamma),\partial)\rightarrow\Theta$
associ\'e \`a un homomorphisme de dg-graphes $f: \Gamma\rightarrow\Theta$ (de degr\'e $0$)
d\'efinit un morphisme de dg-cat\'egories si et seulement si on a la relation
\begin{equation}\renewcommand{\theequation}{**}
\delta f(\alpha) - f(\delta\alpha) = \phi_f(\partial(\alpha))
\end{equation}
pour tout g\'en\'erateur~$\alpha\in\Gamma$.

\medskip
On reprend maintenant la d\'efinition de la construction cobar de~\cite{HusemollerMooreStasheff}.
On utilisera l'extension naturelle aux dg-graphes de l'op\'eration de suspension des dg-modules $\Sigma C$
dont la d\'efinition est rappel\'ee au~\S\ref{Background:Suspension},
ainsi que l'op\'eration inverse de d\'esuspension $\Sigma^{-1}C$.

\subsubsection{La construction cobar appliqu\'ee aux dg-cocat\'egories}\label{DGConstructions:CobarCategory}
La construction cobar est un foncteur qui applique une dg-cocat\'egorie
sur une dg-cat\'egorie quasi-libre.
Les dg-cocat\'egories sont les structures qui lorsque $\X$ est fini
correspondent aux cog\`ebres augment\'ees sur $\kk_{\X}$
dans le dictionnaire du~\S\ref{DGConstructions:DGAlgebraEquivalence}.

Une dg-cocat\'egorie consiste en g\'en\'eral en la donn\'ee d'un dg-graphe $\Gamma$
muni de morphismes de diagonale
\begin{equation*}
\Gamma(\bobj,\aobj)
\xrightarrow{\Delta}\bigoplus_{\xobj\in\X}
\Gamma(\xobj,\aobj)\otimes\Gamma(\bobj,\xobj)
\end{equation*}
et de morphismes d'augmentation
\begin{equation*}
\Gamma(\xobj,\xobj)\xrightarrow{\epsilon}\kk
\end{equation*}
satisfaisant les duaux naturels des relations d'associativit\'e et d'unit\'e
des dg-cat\'egories.

On ne consid\'erera dans la suite que des dg-cocat\'egories connexes
au sens que $\Gamma(\xobj,\xobj) = \kk$, pour tout $\xobj\in\X$,
et $\Gamma(\bobj,\aobj) = 0$
pour tout couple $\bobj\not=\aobj$
tel que $\deg(\bobj)-\deg(\aobj)\leq 0$.
On forme alors le dg-graphe $\widetilde{\Gamma}$
tel que~:
\begin{equation*}
\widetilde{\Gamma}(\bobj,\aobj) = \begin{cases} 0, & \text{si $\bobj = \aobj$}, \\
\Gamma(\bobj,\aobj), & \text{sinon}.
\end{cases}
\end{equation*}

La diagonale de $\Gamma$
induit un morphisme de degr\'e $-1$
\begin{equation*}
\Sigma^{-1}\widetilde{\Gamma}(\bobj,\aobj)
\xrightarrow{\Delta}\bigoplus_{\xobj\in\X}
\Sigma^{-1}\widetilde{\Gamma}(\xobj,\aobj)
\otimes\Sigma^{-1}\widetilde{\Gamma}(\bobj,\xobj)
\subset\Free(\Sigma^{-1}\widetilde{\Gamma})(\bobj,\aobj)
\end{equation*}
qui d\'etermine une d\'erivation $\partial$
sur la dg-cat\'egorie libre $\Free(\Sigma^{-1}\widetilde{\Gamma})$.
On montre ais\'ement que cette d\'erivation est une d\'erivation de torsion (on a en fait $\delta(\partial) = 0$ et $\partial^2 = 0$).

La construction cobar de $\Gamma$
est la dg-cat\'egorie quasi-libre
\begin{equation*}
B^c(\Gamma) = (\Free(\Sigma^{-1}\widetilde{\Gamma}),\partial)
\end{equation*}
associ\'ee \`a la d\'erivation de torsion ainsi d\'efinie.

\medskip
La construction cobar d\'efinit clairement un foncteur sur la cat\'egorie des dg-coca\-t\'e\-go\-ries connexes.
De plus~:

\begin{prop}\label{DGConstructions:CobarEquivalence}
Le morphisme de dg-cat\'egories
\begin{equation*}
B^c(f): B^c(\Gamma)\rightarrow B^c(\Delta)
\end{equation*}
induit par une \'equivalence faible de dg-cocat\'egories connexes $f: \Gamma\xrightarrow{\sim}\Delta$
est une \'equivalence faible,
pourvu que les les dg-cocat\'egories $\Gamma$ et $\Delta$
soient cofibrantes comme dg-graphes.
\end{prop}

\begin{proof}
On applique l'argument spectral standard aux dg-cat\'egories.
On observera simplement que l'hypoth\`ese de connexit\'e assure la convergence
des suites spectrales utilis\'ees.
\end{proof}

\subsubsection{La construction bar appliqu\'ee aux dg-cat\'egories}\label{DGConstructions:BarCategory}
Soit $\Theta\in\dg\Cat_{\X}^+$ une dg-cat\'egorie connexe
au sens d\'efini au~\S\ref{DGConstructions:DGCategories}.
On forme le dg-graphe $\widetilde{\Theta}$
tel que
\begin{equation*}
\widetilde{\Theta}(\bobj,\aobj) = \begin{cases} 0, & \text{si $\bobj = \aobj$}, \\
\Theta(\bobj,\aobj), & \text{sinon}.
\end{cases}
\end{equation*}
La construction bar de $\Theta$
est un dg-graphe $B(\Theta)$
d\'efini par les produits tensoriels de dg-modules
\begin{equation*}
B(\Theta)(\bobj,\aobj)
= \bigoplus_{\substack{(\bobj,\xobj_{m-1},\dots,\xobj_{1},\aobj)\\
m\geq 0}}
\Sigma\widetilde{\Theta}(\xobj_{1},\aobj)
\otimes\Sigma\widetilde{\Theta}(\xobj_{2},\xobj_{1})
\otimes\dots
\otimes\Sigma\widetilde{\Theta}(\bobj,\xobj_{m-1}).
\end{equation*}
muni d'un certain homomorphisme de torsion $\partial$.
Lorsque $\bobj\not=\aobj$,
la somme s'\'etend sur les collections $(\xobj_0,\dots,\xobj_m)$
telles que $\xobj_i\not=\xobj_{i+1}$ et $\deg\xobj_{i+1} - \deg\xobj_{i}>0$, $\forall i$,
puisque les dg-modules $\widetilde{\Theta}(\xobj_{i+1},\xobj_{i})$ ne s'annulent pas
sous ces conditions seulement.
On a alors $\widetilde{\Theta}(\xobj_{i+1},\xobj_{i}) = \Theta(\xobj_{i+1},\xobj_{i})$.

On a par convention $B(\Theta)(\xobj,\xobj) = \kk$
lorsque $\bobj=\aobj=\xobj$.
On note aussi que l'on a~$B(\Theta)(\bobj,\aobj) = 0$
pour tout couple $\bobj\not=\aobj$
tel que $\deg(\bobj)-\deg(\aobj)\leq 0$.

La diff\'erentielle de $B(\Theta)$
est la somme de la diff\'erentielle naturelle des tenseurs, induite par la diff\'erentielle interne de $\Theta$,
et de l'homomorphisme de torsion
\begin{equation*}
\partial: B(\Theta)(\bobj,\aobj)\rightarrow B(\Theta)(\bobj,\aobj)
\end{equation*}
tel que
\begin{equation*}
\partial(\alpha_1\otimes\dots\otimes\alpha_m)
= \sum_{i=1}^{m-1} \pm\{\alpha_1\otimes\dots\otimes(\alpha_{i}\cdot\alpha_{i+1})\otimes\dots\otimes\alpha_m\}, 
\end{equation*}
pour un tenseur $\alpha_1\otimes\dots\otimes\alpha_m\in\Sigma\Theta(\xobj_1,\xobj_0)\otimes\cdots
\otimes\Sigma\Theta(\xobj_m,\xobj_{m-1})$.
Le signe $\pm$ est d\'etermin\'e par la commutation,
avec les facteurs $\alpha_k\in\Sigma\Theta(\xobj_{k},\xobj_{k-1})$, $k<i$,
du morphisme de composition de $\Theta$
qui, par suspension,
est \'equivalent \`a un homomorphisme de degr\'e $-1$~:
\begin{equation*}
\Sigma\Theta(\xobj_{i+1},\xobj_{i})\otimes\Sigma\Theta(\xobj_{i+2},\xobj_{i+1})
\xrightarrow{\mu}\Sigma\Theta(\xobj_{i+2},\xobj_{i}).
\end{equation*}

Lorsque $\Theta = \Omega_n^{epi}$,
on retrouve la d\'efinition utilis\'ee au~\S\ref{KoszulConstruction}.

\subsubsection{La construction cobar-bar appliqu\'ee aux dg-cat\'egories}\label{DGConstructions:CobarBarCategory}
La construction bar d'une dg-cat\'egorie $B(\Theta)$
h\'erite d'une diagonale
\begin{equation*}
B(\Theta)(\bobj,\aobj)
\xrightarrow{\Delta}\bigoplus_{\xobj\in\X}
B(\Theta)(\xobj,\aobj)
\otimes B(\Theta)(\bobj,\xobj)
\end{equation*}
d\'efinie composante par composante
par les morphismes de d\'econcat\'enation
\begin{multline*}
\underbrace{\{\Sigma\Theta(\xobj_1,\aobj)\otimes\cdots
\otimes\Sigma\Theta(\bobj,\xobj_{m-1})\}}_{\subset B(\Theta)} \\
\rightarrow
\underbrace{\{\Sigma\Theta(\xobj_1,\aobj)\otimes\cdots
\otimes\Sigma\Theta(\xobj_k,\xobj_{k-1})\}}_{\subset B(\Theta)} \\
\otimes
\underbrace{\{\Sigma\Theta(\xobj_{k+1},\xobj_k)\otimes\cdots
\otimes\Sigma\Theta(\bobj,\xobj_{m-1})\}}_{\subset B(\Theta)}.
\end{multline*}
L'identit\'e $B(\Theta)(\xobj,\xobj) = \kk$
nous donne aussi un morphisme de counit\'e sur $B(\Theta)$,
de sorte que le dg-graphe $B(\Theta)$
forme une dg-cocat\'egorie.
On note aussi que cette dg-cocat\'egorie est, d'apr\`es les observations du~\S\ref{DGConstructions:BarCategory}, connexe
au sens d\'efini au~\S\ref{DGConstructions:CobarCategory}.

On forme la construction cobar associ\'ee \`a~$B(\Theta)$.
On adoptera la notation $\widetilde{\Gamma} = \widetilde{B}(\Theta)$
pour le dg-graphe connexe associ\'ee \`a cette dg-cocat\'egorie $\Gamma = B(\Theta)$.
On a alors la proposition suivante~:

\begin{prop}\label{DGConstructions:CobarBarAugmentation}\hspace*{2mm}
\begin{enumerate}
\item\label{CobarBarAugmentation:Construction}
L'homomorphisme de dg-graphes $\epsilon: \widetilde{B}(\Theta)\rightarrow\Theta$
d\'efini par les projections sur les composantes tensorielles d'ordre $m=1$
de $B(\Theta)$
induit un morphisme de dg-cat\'egories~:
\begin{equation*}
B^c(B(\Theta)) = (\Free(\Sigma^{-1}\widetilde{B}(\Theta)),\partial)
\xrightarrow{\epsilon}\Theta.
\end{equation*}
\item\label{CobarBarAugmentation:Acyclicity}
Et le morphisme d'augmentation $\epsilon: B^c(B(\Theta))\rightarrow\Theta$
ainsi d\'efini
est une \'equi\-va\-lence faible.
\end{enumerate}
\end{prop}

\begin{proof}
La d\'emonstration de l'assertion~(\ref{CobarBarAugmentation:Construction})
se r\'eduit \`a une v\'erification facile, laiss\'ee en exercice,
de la relation~(**)
du~\S\ref{DGConstructions:QFreeCategories}.

On renvoie \`a~\cite[II \S 4]{HusemollerMooreStasheff}
pour une d\'emonstration de l'assertion~(\ref{CobarBarAugmentation:Acyclicity})
dans le cadre des dg-alg\`ebres.
\end{proof}

Cettte construction cobar-bar $B^c(B(\Theta))$ d\'efinit ainsi un mod\`ele quasi-libre naturel de $\Theta$
dans la cat\'egorie des dg-cat\'egories $\dg\Cat_{\X}$.

\subsubsection{Cat\'egories de diagrammes associ\'ees \`a une dg-cat\'egorie}\label{DGConstructions:DiagramCategories}
La notion de diagramme sur une petite cat\'egorie poss\`ede une extension naturelle dans le cadre des dg-cat\'egories~:
un $\Theta$-diagramme covariant $T: \Theta\rightarrow\dg\Mod$
consiste en la donn\'ee d'une application $T: \X\rightarrow\dg\Mod$
qui associe un dg-module $T(\xobj)\in\dg\Mod$ \`a chaque objet $\xobj\in\X$
et de morphismes de dg-modules
\begin{equation*}
\Theta(\bobj,\aobj)\otimes T(\bobj)\xrightarrow{T^{\sharp}} T(\aobj),
\end{equation*}
pour chaque couple d'objets $(\bobj,\aobj)\in\X\times\X$,
satisfaisant des relations d'unit\'e et d'associativit\'e standard~;
la notion de $\Theta$-diagramme contravariant est d\'efinie
de fa\c con sym\'etrique,
avec des morphismes
\begin{equation*}
\Theta(\bobj,\aobj)\otimes S(\aobj)\xrightarrow{S^{\sharp}} S(\bobj)
\end{equation*}
agissant en sens inverse.

La cat\'egorie des diagrammes (covariants) associ\'ee \`a $\Theta$
sera not\'ee $\dg\Mod^{\Theta}$.
On a une \'equivalence formelle entre la cat\'egorie des $\Theta$-diagrammes contravariants
et la cat\'egorie des diagrammes covariants sur la dg-cat\'egorie $\Theta^{op}$
telle que $\Theta^{op}(\bobj,\aobj) = \Theta(\aobj,\bobj)$.

\medskip
Lorsque $\X$ est fini, le dictionnaire de~\S\ref{DGConstructions:DGAlgebraEquivalence},
donne une \'equivalence de cat\'egories
entre $\dg\Mod^{\Theta}$
et la cat\'egories des modules \`a gauche sur la dg-alg\`ebre $\Theta_{\X}$
associ\'ee \`a la dg-cat\'egorie $\Theta$~:
on associe d'abord \`a toute collection $T(\xobj)$, $\xobj\in\X$,
le dg-module $T_{\X} = \bigoplus_{\xobj\in\X} T(\xobj)$
muni de la structure de $\kk_{\X}$-module telle que $T(\xobj) = 1_{\xobj}\cdot T_{\X}$~;
lorsque $T$ est un $\Theta$-diagramme,
on a un produit $\Theta_{\X}\otimes_{\kk_{\X}}T_{\X}\rightarrow T_{\X}$
donn\'e composante par composante par l'action de $\Theta$
sur $T$,
qui fait de $T_{\X}$ un module sur $\Theta_{\X}$.
On a une \'equivalence de cat\'egories analogue
entre la cat\'egorie des diagrammes contravariants $\dg\Mod^{\Theta^{op}}$
et la cat\'egories des modules \`a droite sur $\Theta_{\X}$.

\medskip
On s'int\'eresse aux diagrammes sur une construction cobar $B^c(\Gamma)$.
On section de la proposition suivante~:

\begin{prop}\hspace*{2mm}\label{DGConstructions:EndomorphismCategory}
\begin{enumerate}
\item
Une collection de dg-modules $T(\xobj)$, $\xobj\in\X$,
forme un diagramme sur la dg-cat\'egorie $\End_T\in\dg\Cat_{\X}$
telle que
\begin{equation*}
\End_T(\bobj,\aobj) = \Hom_{\dg\Mod}(T(\bobj),T(\aobj)).
\end{equation*}
\item
De plus,
munir la collection $T(\xobj)$, $\xobj\in\X$,
d'une structure de $\Theta$-diagramme
revient \`a se donner un morphisme de dg-cat\'egories $\phi: \Theta\rightarrow\End_T$.
\end{enumerate}
\end{prop}

\begin{proof}Formel.\end{proof}

Une structure de $B^c(\Gamma)$-diagramme covariant sur une collection $T(\bobj)$, $\bobj\in\X$,
est donc donn\'ee par un morphisme de dg-cat\'egories $\psi: B^c(\Gamma)\rightarrow\End_T$.
Les observations du~\S\ref{DGConstructions:QFreeCategories}
montrent que ce morphisme $\psi = \psi_g$
est d\'etermin\'e par sa restriction $g$ au dg-graphe $\Sigma^{-1}\widetilde{\Gamma}$,
laquelle associe un homomorphisme de dg-modules $\alpha_*: T(\bobj)\rightarrow T(\aobj)$
de degr\'e $d-1$
\`a chaque morphisme g\'en\'erateur $\alpha\in\Gamma(\bobj,\aobj)$
de degr\'e $d$,
de telle sorte que l'on a la relation
\begin{equation*}
\delta(\alpha_*) = \psi_g((\delta+\partial)(\alpha))
\end{equation*}
dans $\Hom_{\dg\Mod}(T(\bobj),T(\aobj))$.
L'homomorphisme $g: \Sigma^{-1}\widetilde{\Gamma}\rightarrow\End_T$
est par adjonction \'equivalent \`a une collection d'homomorphismes
\begin{equation*}
\Gamma(\bobj,\aobj)\otimes T(\bobj)\xrightarrow{g^{\sharp}} T(\aobj)
\end{equation*}
de degr\'e $-1$.

On a une observation analogue pour les $B^c(\Gamma)$-diagrammes contravariants
car on constate que la cat\'egorie sous-jacente \`a $B^c(\Gamma)^{op}$
s'identifie \`a la construction cobar $B^c(\Gamma{}^{op})$ sur le diagramme $\Gamma{}^{op}$
tel que $\Gamma{}^{op}(\aobj,\bobj) = \Gamma(\bobj,\aobj)$.
Une structure de $B^c(\Gamma)$-diagramme contravariant sur une collection $S(\aobj)$, $\aobj\in\X$,
est donc donn\'ee
par un morphisme de dg-cat\'egories $\phi: B^c(\Gamma^{op})\rightarrow \End_S$.
L'homomorphisme $f: \Sigma^{-1}\widetilde{\Gamma}{}^{op}\rightarrow\End_S$
d\'eterminant ce morphisme $\phi = \phi_f$
est par adjonction \'equivalent \`a une collection d'homomorphismes
\begin{equation*}
\Gamma(\bobj,\aobj)\otimes S(\aobj)\xrightarrow{f^{\sharp}} S(\bobj)
\end{equation*}
de degr\'e $-1$.

\subsubsection{Complexes \`a coefficients associ\'es \`a une construction cobar}\label{DGConstructions:CobarCoefficientComplexes}
On se donne toujours un couple $(S,T)$
tel que $S$ est un diagramme contravariant et $T$ est un diagramme covariant
sur la dg-cat\'egorie $B^c(\Gamma)$
associ\'ee \`a une dg-cocat\'egorie.
On note $S\otimes_{\X}\Gamma\otimes_{\X} T$
le produit tensoriel
\begin{equation*}
S\otimes_{\X}\Gamma\otimes_{\X} T
= \bigoplus_{(\bobj,\aobj)\in\X\times\X} S(\aobj)
\otimes\Gamma(\bobj,\aobj)\otimes T(\bobj)
\end{equation*}
qui dans notre dictionnaire entre dg-cat\'egorie et dg-alg\`ebre
correspond au produit tensoriel sur $\kk_{\X}$
des $\kk_{\X}$-modules associ\'es aux collections sous-jacentes des objets $S$, $\Gamma$ et $T$.
Ce produit tensoriel est muni d'un homomorphisme de torsion naturel
\begin{equation*}
\partial: S\otimes_{\X}\Gamma\otimes_{\X} T\rightarrow S\otimes_{\X}\Gamma\otimes_{\X} T
\end{equation*}
d\'efini composante par composante, pour tout couple de $B^c(\Gamma)$-diagrammes $(S,T)$,
par l'homomorphisme compos\'e
\begin{multline*}
S(\aobj)\otimes\Gamma(\bobj,\aobj)\otimes T(\bobj)\\
\begin{aligned}
& \xrightarrow{\Delta_*}
\bigoplus_{\xobj\in\X}
S(\aobj)\otimes\Gamma(\xobj,\aobj)
\otimes\Gamma(\bobj,\xobj)
\otimes T(\bobj)\\
& \xrightarrow{(f^{\sharp}_*,g^{\sharp}_*)}
\Bigl\{\bigoplus_{\xobj\in\X}
S(\xobj)
\otimes\Gamma(\bobj,\xobj)
\otimes T(\bobj)\Bigr\}
\oplus
\Bigl\{\bigoplus_{\xobj\in\X}
S(\aobj)
\otimes\Gamma(\xobj,\aobj)
\otimes T(\xobj)\Bigr\},
\end{aligned}
\end{multline*}
d\'etermin\'e par la diagonale de $\Gamma$
et l'action de $\Gamma\subset B^c(\Gamma)$
sur $S$ et $T$.
La relation des homomorphismes de torsion $\delta(\partial) + \partial^2 = 0$
se d\'eduit ais\'ement des \'equations du~\S\ref{DGConstructions:QFreeCategories}
pour les morphismes $\phi_f$ et $\phi_g$.
Par cons\'equent, on a un dg-module tordu bien d\'efini
$(S\otimes_{\X}\Gamma\otimes_{\X} T,\partial)$
naturellement associ\'e \`a tout couple de $B^c(\Gamma)$-diagrammes $(S,T)$.
Le complexe bar \`a coefficients associ\'e \`a une dg-cat\'egorie connexe $\Theta$
s'identifie au complexe tordu
$B(S,\Theta,T) = (S\otimes_{\X} B(\Theta)\otimes_{\X} T,\partial)$
associ\'e \`a la construction bar de~$\Theta$.

\medskip
Le produit tensoriel au-dessus d'une petite cat\'egorie $S\otimes_{\Theta} T$
poss\`ede une extension naturelle
dans le cadre des dg-cat\'egories.
On a aussi une extension naturelle des foncteurs $\Tor$
au diagrammes sur une dg-cat\'egorie.
Le produit tensoriel $S\otimes_{\Theta} T$ s'identifie en fait au produit tensoriel $S_{\X}\otimes_{\Theta_{\X}} T_{\X}$
des modules associ\'es \`a $S$ et $T$
sur la dg-alg\`ebre associ\'ee \`a $\Theta$
et le foncteur $\Tor$ se d\'etermine par les m\'ethodes d'alg\`ebre homologique diff\'erentielle standard
(on appliquera par exemple~\cite[\S\S 18-20]{FelixHalperinThomas}).

Les complexes tordus d\'efinis au~\S\ref{DGConstructions:CobarCoefficientComplexes}
satisfont la relation
\begin{equation*}
\Tor^{B^c(\Gamma)}_*(S,T) = H_*(S\otimes_{\X}\Gamma\otimes_{\X} T,\partial),
\end{equation*}
pour tout couple de $B^c(\Gamma)$-diagrammes dg-cofibrants $(S,T)$,
car on a une \'equivalence faible de dg-cocat\'egories $\eta: \Gamma\xrightarrow{\sim} B(B^c(\Gamma))$ d'apr\`es~\cite[II \S 4]{HusemollerMooreStasheff},
et un argument spectral standard nous permet de comparer le complexe $(S\otimes_{\X}\Gamma\otimes_{\X} T,\partial)$
au complexe bar \`a coefficients
$B(S,B^c(\Gamma),T) = (S\otimes_{\X} B(B^c(\Gamma))\otimes_{\X} T,\partial)$
calculant $\Tor^{B^c(\Gamma)}_*(S,T)$.

\subsection{Construction cobar et mod\`ele minimal de la cat\'egorie des arbres \'elagu\'es}\label{PrunedTreeCobar}
La construction bar $B(\Omega_n^{epi})$
du~\S\ref{KoszulConstruction}
s'identifie \`a la construction bar de la cat\'egorie enrichie en $\kk$-modules $\Omega_n^{epi}$
vue comme une dg-cat\'egorie concentr\'ee en degr\'e $0$.
Cette construction bar h\'erite d'apr\`es l'observation du~\S\ref{DGConstructions:CobarBarCategory}
d'une diagonale
\begin{equation*}
B(\Omega_n^{epi})(\tauobj,\sigmaobj)
\xrightarrow{\Delta}\bigoplus_{\thetaobj}
B(\Omega_n^{epi})(\thetaobj,\sigmaobj)\otimes B(\Omega_n^{epi})(\tauobj,\thetaobj),
\end{equation*}
qui en fait une dg-cocat\'egorie connexe.
Si on revient \`a la d\'efinition du~\S\ref{KoszulConstruction},
alors cette diagonale est donn\'ee par la formule
\begin{equation*}
\Delta\{\tauobj_0\xleftarrow{u_1}\cdots\xleftarrow{u_d}\tauobj_d\}
= \sum_{k = 0}^{d} \{\tauobj_0\xleftarrow{u_1}\cdots\xleftarrow{u_k}\tauobj_k\}
\otimes\{\tauobj_k\xleftarrow{u_{k+1}}\cdots\xleftarrow{u_d}\tauobj_d\}
\end{equation*}
pour les g\'en\'erateurs de $B(\Omega_n^{epi})$.

Le dg-graphe $K(\Omega_n^{epi})$
est muni, \`a l'instar de la construction bar $B(\Omega_n^{epi})$,
d'une diagonale coassociative
\begin{equation*}
K(\Omega_n^{epi})(\tauobj,\sigmaobj)
\xrightarrow{\Delta}\bigoplus_{\thetaobj}
K(\Omega_n^{epi})(\thetaobj,\sigmaobj)\otimes K(\Omega_n^{epi})(\tauobj,\thetaobj),
\end{equation*}
qui est simplement d\'efinie
par la formule
\begin{equation*}
\Delta\{u\} = \sum_{u = v w} \{v\}\otimes\{w\}
\end{equation*}
pour tout morphisme $u: \tauobj\rightarrow\sigmaobj$.
On a aussi l'identit\'e $K(\Omega_n^{epi})(\tauobj,\tauobj) = \kk$,
pour tout $\tauobj\in\Ob\Omega_n^{epi}$.

La construction de Koszul $K(\Omega_n^{epi})$ forme ainsi une dg-cocat\'egorie connexe.
On utilise la notation $\widetilde{\Gamma} = \widetilde{K}(\Omega_n^{epi})$
pour le dg-graphe connexe associ\'ee \`a cette dg-cocat\'egorie $\Gamma = K(\Omega_n^{epi})$.

On peut donc appliquer la construction cobar \`a $K(\Omega_n^{epi})$
pour obtenir une dg-cat\'egorie quasi-libre
\begin{equation*}
B^c(K(\Omega_n^{epi})) = (\Free(\Sigma^{-1}\widetilde{K}(\Omega_n^{epi})),\partial)
\end{equation*}
associ\'ee \`a $\Omega_n^{epi}$.
La d\'erivation de torsion de $B^c(K(\Omega_n^{epi}))$
est donn\'ee par la formule
\begin{equation*}
\partial\{u\} = \sum_{u = v w} (-1)^{\deg v}\cdot\{v\}\otimes\{w\},
\end{equation*}
pour tout \'el\'ement g\'en\'erateur $\{u\}\in K(\Omega_n^{epi})(\tauobj,\sigmaobj)$.
Le signe additionel $\pm = (-1)^{\deg v}$
provient de la commutation implicite d'une d\'esuspension
avec le facteur $\{v\}\in\break\widetilde{K}(\Omega_n^{epi})(\thetaobj,\sigmaobj)$
dans la d\'efinition cat\'egorique de $\partial$
au~\S\ref{DGConstructions:CobarCategory}.

\medskip
On constate imm\'ediatement que~:

\begin{obsv}\label{PrunedTreeCobar:DiagonalKoszulEmbedding}
Le morphisme de dg-graphes $K(\Omega_n^{epi})\xrightarrow{\iota} B(\Omega_n^{epi})$
d\'efini par la Proposition~\ref{KoszulConstruction:KoszulEmbedding}
est un morphisme de dg-cocat\'egories.
\end{obsv}

On a donc, par fonctorialit\'e de la construction cobar, un morphisme de dg-cat\'egories
\begin{equation*}
B^c(K(\Omega_n^{epi}))\xrightarrow{B^c(\iota)}
B^c(B(\Omega_n^{epi}))
\end{equation*}
qui, par composition avec l'augmentation de $B^c(B(\Omega_n^{epi}))$,
munit la dg-cat\'egorie $B^c(K(\Omega_n^{epi}))$
d'une augmentation sur $\Omega_n^{epi}$.
On obtient, par inspection des constructions,
que le morphisme d'augmentation
\begin{equation*}
\epsilon: B^c(K(\Omega_n^{epi}))\xrightarrow{\sim}\Omega_n^{epi}
\end{equation*}
est d\'efini sur les g\'en\'erateurs du dg-graphe $K(\Omega_n^{epi})$ par~:
\begin{equation*}
\epsilon\{u\} = \begin{cases} \sgn(u)\cdot\{u\}, & \text{si $\deg u = 1$}, \\ 0, & \text{sinon}. \end{cases}
\end{equation*}

Le r\'esultat du Th\'eor\`eme~\ref{CoefficientConstructions:KoszulEquivalence}
entraine~:

\begin{lemm}\label{PrunedTreeCobar:CobarKoszulEquivalence}
Le morphisme de dg-cocat\'egories $\iota: K(\Omega_n^{epi})\rightarrow B(\Omega_n^{epi})$
induit une \'equivalence faible au niveau des constructions cobar
\begin{equation*}
\iota_*: B^c(K(\Omega_n^{epi}))\xrightarrow{\sim}B^c(B(\Omega_n^{epi})).\qed
\end{equation*}
\end{lemm}

On en conclut~:

\begin{mainsectionthm}\label{PrunedTreeCobar:KoszulModel}
La construction $\Res(\Omega_n^{epi}) = B^c(K(\Omega_n^{epi}))$
d\'efinit un mod\`ele quasi-libre de $\Omega_n^{epi}$
dans la cat\'egorie des dg-cat\'egories.\qed
\end{mainsectionthm}

Ce mod\`ele quasi-libre $\Res(\Omega_n^{epi}) = B^c(K(\Omega_n^{epi}))$
est minimal
au sens que le dg-graphe $K(\Omega_n^{epi})$
est muni d'une diff\'erentielle interne triviale
et la diff\'erentielle
de la construction cobar $B^c(K(\Omega_n^{epi})) = (\Free(\Sigma^{-1}\widetilde{K}(\Omega_n^{epi})),\partial)$
applique les \'el\'ements g\'en\'erateurs sur des \'el\'ements d\'ecomposables.

On sait que $\Omega_n^{epi}$
est engendr\'ee comme cat\'egorie par les morphismes de degr\'e $1$.
La relation $H_0(B^c(K(\Omega_n^{epi}))) = \Omega_n^{epi}$
entraine la propri\'et\'e suppl\'ementaire suivante
qui n'est pas apparente dans la d\'efinition du~\S\ref{PrunedTrees:Definition}~:

\begin{mainsectioncor}
Les relations entre morphismes de degr\'e $1$
de $\Omega_n^{epi}$
sont engendr\'ees par les identit\'es quadratiques
d\'efinies par les diagrammes des figures~\ref{Fig:ConsecutiveVertexMerging},~\ref{Fig:DisjointVertexMerging},~\ref{Fig:OrderedVertexMerging}
de la Proposition~\ref{KoszulConstruction:KoszulEmbedding}.
\end{mainsectioncor}

Cette assertion, d\'emontr\'ee au niveau des $\kk$-modules, reste valable pour la cat\'egorie ensembliste $\Omega_n^{epi}$
puisque la version enrichie en $\kk$-modules de $\Omega_n^{epi}$
est d\'efinie par les $\kk$-modules librement engendr\'es par les morphismes ensemblistes de $\Omega_n^{epi}$.

On a donc, en corollaire du Th\'eor\`eme~\ref{PrunedTreeCobar:KoszulModel},
une pr\'esentation par g\'en\'erateurs et relations de $\Omega_n^{epi}$.
Le r\'esultat de ce corollaire peut \'egalement se d\'eduire, dans l'esprit de~\cite{BergerWreathCategories},
d'une identification de $\Omega_n^{epi}$
avec un produit en couronne it\'er\'e de cat\'egories.

\medskip
On s'int\'eresse maintenant aux diagrammes sur la dg-cat\'egorie $B^c(K(\Omega_n^{epi}))$.
Les observations suivant la Proposition~\ref{DGConstructions:EndomorphismCategory}
entrainent le r\'esultat suivant~:

\pagebreak
\begin{prop}\hspace*{2mm}\label{PrunedTreeCobar:Diagrams}
\begin{enumerate}
\item
Munir une collection de dg-modules $T(\tauobj)$, $\tauobj\in\Ob\Omega_n^{epi}$,
d'une structure de diagramme covariant sur $\Res(\Omega_n^{epi}) = B^c(K(\Omega_n^{epi}))$
revient \`a associer \`a tout morphisme $u\in\Mor_{\Omega_n^{epi}}(\tauobj,\sigmaobj)$
un homomorphisme de dg-modules de degr\'e $\deg(u)-1$
\begin{equation*}
u_*: T(\tauobj)\rightarrow T(\sigmaobj)
\end{equation*}
de telle sorte que les \'equations
\begin{equation*}
\delta(u_*) = \sum_{u = v w} (-1)^{\deg v}\cdot v_* w_*
\end{equation*}
sont satisfaites dans $\Hom_{\dg\Mod}(T(\tauobj),T(\sigmaobj))$,
la somme s'\'etendant sur l'ensemble des d\'ecompositions $u = v w$
dans la cat\'egorie $\Omega_n^{epi}$.
\item
Sym\'etriquement, munir une collection de dg-modules $S(\tauobj)$, $\tauobj\in\Ob\Omega_n^{epi}$,
d'une structure de diagramme contravariant sur $\Res(\Omega_n^{epi})$
revient \`a associer \`a tout morphisme $u\in\Mor_{\Omega_n^{epi}}(\tauobj,\sigmaobj)$
un homomorphisme de dg-modules de degr\'e $\deg(u)-1$
\begin{equation*}
u^*: S(\sigmaobj)\rightarrow S(\tauobj)
\end{equation*}
de telle sorte que les \'equations
\begin{equation*}
\delta(u^*) = \sum_{u = v w} (-1)^{\deg v}\cdot w^* v^*
\end{equation*}
sont satisfaites dans $\Hom_{\dg\Mod}(S(\sigmaobj),S(\tauobj))$.\qed
\end{enumerate}
\end{prop}

\subsubsection{Construction cobar de la construction de Koszul et complexes \`a coefficients}\label{PrunedTreeCobar:CoefficientComplexes}
On a d\'efini au~\S\S\ref{CoefficientConstructions}-\ref{DGCoefficientConstructions}
des complexes \`a coefficients $B(S,\Omega_n^{epi},T)$
et $K(S,\Omega_n^{epi},T)$
tels que
\begin{equation*}
H_*(K(S,\Omega_n^{epi},T)) = H_*(B(S,\Omega_n^{epi},T)) = \Tor^{\Omega_n^{epi}}_*(S,T),
\end{equation*}
pour tout couple de $\Omega_n^{epi}$-diagrammes dg-cofibrants.

Le complexe $B(S,\Omega_n^{epi},T)$ utilis\'e aux~\S\S\ref{CoefficientConstructions}-\ref{DGCoefficientConstructions}
s'identifie au complexe tordu associ\'e \`a la dg-cocat\'egorie $\Gamma = B(\Omega_n^{epi})$
dans le~\S\ref{DGConstructions:CobarCoefficientComplexes}.
Le complexe $K(S,\Omega_n^{epi},T)$
s'identifie de m\^eme au complexe tordu du~\S\ref{DGConstructions:CobarCoefficientComplexes}
associ\'e \`a la dg-cocat\'egorie $\Gamma = K(\Omega_n^{epi})$
La construction de~\S\ref{DGConstructions:CobarCoefficientComplexes} donne donc une extension
des complexes \`a coefficients du~\S\ref{CoefficientConstructions}
aux $B^c(\Gamma(\Omega_n^{epi}))$-diagrammes, pour $\Gamma = B,K$.
Tout $B^c(K(\Omega_n^{epi}))$-diagramme h\'erite d'une structure de $B^c(B(\Omega_n^{epi}))$-diagramme par restriction de structure.
On peut montrer par un argument spectral standard que le morphisme de dg-cocat\'egories
$\iota: K(\Omega_n^{epi})\rightarrow B(\Omega_n^{epi})$
induit une \'equivalence faible
\begin{equation*}
\kappa: K(S,\Omega_n^{epi},T)\xrightarrow{\sim} B(S,\Omega_n^{epi},T),
\end{equation*}
pour tout couple de $B^c(K(\Omega_n^{epi}))$-diagrammes dg-cofibrants,
et on en d\'eduit l'existence de relations
\begin{equation*}
H_*(K(S,\Omega_n^{epi},T)) = H_*(B(S,\Omega_n^{epi},T)) = \Tor^{B^c(K(\Omega_n^{epi}))}_*(S,T).
\end{equation*}
Ces identit\'es \'etendent les r\'esultats obtenus
au~\S\ref{DGCoefficientConstructions}
lorsque $(S,T)$ est un couple de $\Omega_n^{epi}$-diagrammes.

\section*{Postlude : applications aux complexes bar it\'er\'es}\label{IteratedBarApplications}

Le complexe bar $n$-it\'er\'e $B^n(A)$ d'une alg\`ebre commutative $A$
s'identifie, dans l'id\'ee de~\cite{LivernetRichter},
au complexe $C_*(\underline{B}^n(A)) = K(b_n,\Omega_n^{epi},\underline{B}^n(A))$,
pour un $\Omega_n^{epi}$-diagramme $\underline{B}^n(A)$
associ\'e \`a $A$.
On a explicitement $\underline{B}^n(A)(\tauobj) = A^{\otimes\In\tauobj}$
et le morphisme $u_*: A^{\otimes\In\tauobj}\rightarrow A^{\otimes\In\sigmaobj}$
associ\'e \`a un morphisme $u: \tauobj\rightarrow\sigmaobj$
de $\Omega_n^{epi}$
effectue le produit des facteurs $A^{\otimes u^{-1}(i)}$, $i\in\In\sigmaobj$,
associ\'es aux fibres de l'application $u: \In\tauobj\rightarrow\In\sigmaobj$~:
\begin{equation*}
u_*(\bigotimes_{j\in\In\tauobj} a_j)
= \bigotimes_{i\in\In\sigmaobj}\Bigl\{\mu(\bigotimes_{j\in u^{-1}(i)} a_j)\Bigr\},
\end{equation*}
o\`u l'application $\mu$ renvoie au produit de $A$.

On montre dans~\cite{FresseIteratedBar} que la d\'efinition du complexe bar $n$-it\'er\'e $B^n(A)$
s'\'etend aux alg\`ebres sur une $E_n$-op\'erade
et d\'etermine l'homologie $H^{E_n}_*(A)$
naturellement associ\'ee \`a cette cat\'egorie d'alg\`ebre.
Ce complexe bar $n$-it\'er\'e est d\'efini par un dg-module tordu $B^n(A) = (T^n(A),\partial)$
pour un foncteur sous-jacent de la forme $T^n(A) = \bigoplus_{\tauobj} A^{\otimes\In\tauobj}$,
avec un d\'ecalage en degr\'e implicite donn\'e par l'insertion de suspensions
sur les sommets de $\tauobj$.

On fixe une $E_n$-op\'erade, soit $\EOp_n$.
Si on reprend soigneusement la construction effective de l'homomorphisme de torsion de $B^n(A)$,
telle qu'elle est donn\'ee dans~\cite[\S A.13]{FresseIteratedBar},
alors on constate que cet homomorphisme $\partial: T^n(A)\rightarrow T^n(A)$
poss\`ede une composante $\partial_u: A^{\otimes\In\tauobj}\rightarrow A^{\otimes\In\sigmaobj}$,
associ\'ee \`a chaque morphisme $u: \tauobj\rightarrow\sigmaobj$
de $\Omega_n^{epi}$,
qui est donn\'ee par une expression de la forme
\begin{equation*}
\partial_u(\bigotimes_{j\in\In\tauobj} a_j)
= \bigotimes_{i\in\In\sigmaobj}\Bigl\{\pi_i(\bigotimes_{j\in u^{-1}(i)} a_j)\Bigr\},
\end{equation*}
o\`u les applications $\pi_i$ renvoient \`a des op\'erations de l'op\'erade $\EOp_n$.
On comprend en outre que l'\'equation donn\'ee dans~\cite[\S A.13]{FresseIteratedBar}
s'interpr\`ete comme une relation $\delta(\partial_u) = \sum_{u = v w} \partial_v\partial_w$
dans $\Hom_{\dg\Mod}(A^{\otimes\In\tauobj},A^{\otimes\In\sigmaobj})$.
On en conclut,
par application de la Proposition~\ref{PrunedTreeCobar:Diagrams}
et de la d\'efinition du~\S\ref{DGConstructions:CobarCoefficientComplexes}~:
\begin{itemize}
\item la collection de dg-modules $\underline{B}^n(A)(\tauobj) = A^{\otimes\In\tauobj}$ associ\'ee \`a une $\EOp_n$-alg\`ebre $A$
h\'erite d'une structure de $B^c(K(\Omega_n^{epi}))$-diagramme
telle que $u_* = \partial_u$
donne l'action des \'el\'ements g\'en\'erateurs $\{u\}\in K(\Omega_n^{epi})$
sur $\underline{B}^n(A)$,
\item le complexe bar it\'er\'e associ\'e \`a $A$
s'identifie au complexe de Koszul\break $K(b_n,\Omega_n^{epi},\underline{B}^n(A))$
associ\'e \`a ce $B^c(K(\Omega_n^{epi}))$-diagramme $\underline{B}^n(A)$,
pour toute $\EOp_n$-alg\`ebre $A$.
\end{itemize}

On note que le diagramme $\underline{B}^n(A)$ est dg-cofibrant d\`es lors que la $\EOp_n$-alg\`ebre $A$
est cofibrante comme dg-module.
On obtient donc par la g\'en\'eralisation du Th\'eor\`eme~\ref{CoefficientConstructions:TorFunctor}
\'etablie au~\S\ref{PrunedTreeCobar:CoefficientComplexes}~:

\begin{mainthm}
On a les identit\'es~:
\begin{equation*}
H_*(B^n(A)) = H_*(K(b_n,\Omega_n^{epi},\underline{B}^n(A))) = \Tor^{B^c(K(\Omega_n^{epi}))}_*(b_n,\underline{B}^n(A))
\end{equation*}
pour toute $\EOp_n$-alg\`ebre $A$
qui est cofibrante comme dg-module.\qed
\end{mainthm}

Ce th\'eor\`eme donne la g\'en\'eralisation du r\'esultat
principal de~\cite{LivernetRichter}
aux alg\`ebres sur une op\'erade $E_n$.

\medskip
On d\'eduit ensuite du th\'eor\`eme principal de~\cite{FresseIteratedBar}~:

\begin{mainthm}\label{IteratedBarApplications:EnHomology}
On a l'identit\'e~:
\begin{equation*}
H^{\EOp_n}_*(A) = \Tor^{B^c(K(\Omega_n^{epi}))}_*(b_n,\underline{B}^n(A))
\end{equation*}
pour toute $\EOp_n$-alg\`ebre $A$
qui est cofibrante comme dg-module.\qed
\end{mainthm}

\subsubsection*{Remarque}
La relation du Th\'eor\`eme~\ref{IteratedBarApplications:EnHomology}
suppose que le foncteur $\Tor$
s'annule en degr\'e $*>0$
lorsque $A$ est une $\EOp_n$-alg\`ebre libre $A = \EOp_n(C)$.
La d\'emonstration de l'acyclicit\'e de $B^n(\EOp_n(C))$
dans~\cite[\S 8]{FresseIteratedBar}
fait appel \`a une op\'eration de Browder intervenant
dans la composante $\partial: A^{\otimes\In\underline{y}_n}\rightarrow A^{\otimes\In\underline{i}_n}$
de la diff\'erentielle de $B^n(A)$
associ\'ee au morphisme~:
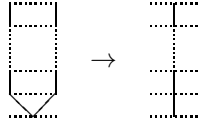
\begin{equation*}
\vcenter{\xymatrix@M=0mm@R=3mm@C=3mm{ \ar@{.}[rr]\ar@{-}[d] && \ar@{-}[d] \\
\ar@{.}[rr]\ar@{.}[d] && \ar@{.}[d] \\
\ar@{.}[d] && \ar@{.}[d] \\
\ar@{.}[rr]\ar@{-}[d] && \ar@{-}[d] \\
\ar@{.}[rr]\ar@{-}[dr] && \ar@{-}[dl] \\
\ar@{.}[rr] && }}
\quad\rightarrow\quad\vcenter{\xymatrix@M=0mm@R=3mm@C=3mm{ \ar@{.}[rr] & \ar@{-}[d] & \\
\ar@{.}[rr] & \ar@{.}[d] & \\
& \ar@{.}[d] & \\
\ar@{.}[rr] & \ar@{-}[d] & \\
\ar@{.}[rr] & \ar@{-}[d] & \\
\ar@{.}[rr] && }}
\end{equation*}
(les notations $\underline{y}_n$ et $\underline{i}_n$ reprennent la forme de ces arbres).
Ce morphisme doit donc agir par une op\'eration de degr\'e $*>0$
pour que le diagramme $\underline{B}^n(A)$
donne le bon r\'esultat. Ceci donne une obstruction \`a avoir une action stricte de la cat\'egorie $\Omega_n^{epi}$
sur la collection $A^{\otimes\In\tauobj}$, $\tauobj\in\Ob\Omega_n^{epi}$,
lorsqu'on \'etend la construction aux alg\`ebres sur une $E_n$-op\'erade.

\end{document}